\documentclass[11pt]{preprint}
\usepackage[utf8x]{inputenc}
\usepackage[T1]{fontenc}

\usepackage{amssymb}
\usepackage{amsmath}

\usepackage[greek.ancient, english]{babel}
\usepackage{newtxtext,newtxmath}

\usepackage{breakurl}
\usepackage{mhenvs}
\usepackage{mhequ}
\usepackage{mhsymb}

\usepackage{booktabs}
\usepackage{tikz}
\usepackage{tikz-cd}
\usepackage{mathrsfs}

\usepackage{microtype}
\usepackage{comment}
\usepackage{slashed}
\usepackage{mathtools}
\usepackage{centernot}
\usepackage{leftidx}
\usepackage{accents}
\usepackage{footnote}

\usepackage{wasysym}
\usepackage{centernot}

\usepackage{enumitem}
\usepackage{stackrel}
\usepackage{longtable}
\usepackage{cprotect}
\usepackage{xstring}

\usepackage{hyperref}

\usetikzlibrary{external}
\DeclareSymbolFont{timesoperators}{T1}{ptm}{m}{n}
\SetSymbolFont{timesoperators}{bold}{T1}{ptm}{b}{n}
\DeclareMathAlphabet{\mathbb}{U}{jkpsyb}{m}{n}
\SetMathAlphabet{\mathbb}{bold}{U}{jkpsyb}{bx}{n}

\def\N{\mathbb{N}}
\def\R{\mathbb{R}}
\def\C{\mathbb{C}}
\def\Z{\mathbb{Z}}
\def\T{\mathbb{T}}

\makeatletter
\renewcommand{\operator@font}{\mathgroup\symtimesoperators}
\makeatother

\colorlet{symbolsgrey}{blue!30!black!50}
\colorlet{testcolor}{green!60!black}
\definecolor{purple}{rgb}{0.55,0.05,0.8}
\definecolor{symbols}{rgb}{0.55,0.05,0.8}

\usetikzlibrary{shapes.misc}
\usetikzlibrary{shapes.symbols}
\usetikzlibrary{shapes.geometric}
\usetikzlibrary{snakes}
\usetikzlibrary{decorations}
\usetikzlibrary{decorations.markings}

\usetikzlibrary{calc}
\usetikzlibrary{external}

\newtheorem{example}[lemma]{Example}

\let\oldskull\skull
\def\skull{\mathord{\oldskull}}

\def\sol{{\mathop{\mathrm{sol}}}}

\newcommand{\hooklongrightarrow}{\lhook\joinrel\longrightarrow}

\def\restr{\mathbin{\upharpoonright}}

\def\wotimes{\mathbin{\widehat{{\otimes}}}}

\def\snabla{\slashed{\nabla}}

\DeclareMathAlphabet{\mathbbm}{U}{bbm}{m}{n}

\DeclareFontFamily{U}{BOONDOX-calo}{\skewchar\font=45 }
\DeclareFontShape{U}{BOONDOX-calo}{m}{n}{
  <-> s*[1.05] BOONDOX-r-calo}{}
\DeclareFontShape{U}{BOONDOX-calo}{b}{n}{
  <-> s*[1.05] BOONDOX-b-calo}{}
\DeclareMathAlphabet{\mcb}{U}{BOONDOX-calo}{m}{n}
\SetMathAlphabet{\mcb}{bold}{U}{BOONDOX-calo}{b}{n}

\setlist{noitemsep,topsep=4pt}

\makeatletter 
\newcommand*{\bigcdot}{}
\DeclareRobustCommand*{\bigcdot}{%
  \mathbin{\mathpalette\bigcdot@{}}%
}
\newcommand*{\bigcdot@scalefactor}{.5}
\newcommand*{\bigcdot@widthfactor}{1.15}
\newcommand*{\bigcdot@}[2]{%
  \sbox0{$#1\vcenter{}$}
  \sbox2{$#1\cdot\m@th$}%
  \hbox to \bigcdot@widthfactor\wd2{%
    \hfil
    \raise\ht0\hbox{%
      \scalebox{\bigcdot@scalefactor}{%
        \lower\ht0\hbox{$#1\bullet\m@th$}%
      }%
    }%
    \hfil
  }%
}
\makeatother

\def\symbol#1{\textcolor{symbolsgrey}{#1}}
\def\1{\mathbf{\symbol{1}}}

\def\bone{\mathbf{1}}

\def\bPsi{\boldsymbol{\Psi}}
\def\bpsi{\boldsymbol{\psi}}
\def\bbpsi{\boldsymbol{\bar{\psi}}}

\def\spacetime{\underline{\Lambda}}
\def\Gr{\mathrm{Gr}}

\newcommand{\mrd}{\mathrm{d}}

\newcommand\bilin[1]{B \big(#1 \big)}

\colorlet{darkblue}{blue!90!black}
\colorlet{darkgreen}{green!50!black}
\let\comm\comment

\let\D\CD
\def\s{\mathfrak{s}}

\def\K{\mathfrak{K}}

\def\${|\!|\!|}
\def\Wick#1{\mathord{\kern0.1em{:}#1{:}\kern0.1em}}
\def\E{\mathbf{E}}

\def\d{\mbox{d}}
\def\-{\mbox{-}}

\def\f{\mathfrak{f}}

\def\reminit{ \CC^{\textnormal{\tiny{init}}}}
\def\treminit{ \widetilde{\CC}^{\textnormal{\tiny{init}}}}

\newcommand{\mfB}{\mathfrak{B}}

\newcommand{\mfI}{\mathfrak{I}}

\newcommand{\mfN}{\mathfrak{N}}
\newcommand{\mfL}{\mathfrak{L}}

\newcommand{\mfA}{\mathfrak{A}}
\newcommand{\mfC}{\mathfrak{C}}

\newcommand{\mfr}{\mathfrak{r}}

\newcommand{\mfm}{\mathfrak{m}}

\newcommand{\mfS}{\mathfrak{S}}
\newcommand{\mfp}{\mathfrak{p}}

\newcommand{\mfh}{\mathfrak{h}}
\newcommand{\mfq}{\mathfrak{q}}
\newcommand{\mfb}{\mathfrak{b}}

\newcommand{\mfG}{\mathfrak{G}}

\newcommand{\mfP}{\mathfrak{P}}

\newcommand{\mfZ}{\mathfrak{Z}}
\newcommand{\mfH}{\mathfrak{H}}
\newcommand{\mfK}{\mathfrak{K}}

\newcommand{\mcK}{\mathcal{K}}

\newcommand{\loc}{\mathrm{loc}}
\DeclareMathOperator{\Trace}{Tr}

\def\cA{\mathscr{A}}

\def\cC{\mathscr{C}}
\def\cA{\mathscr{A}}
\def\cG{\mathscr{G}}

\def\cD{\mathscr{D}}

\def\cL{\mathscr{L}}
\def\cM{\mathscr{M}}
\def\cN{\mathscr{N}}

\newcommand{\VERT}{\vert\!\vert\!\vert}

\DeclareMathOperator{\sgn}{sgn}

\def\emptyset{{\centernot\ocircle}}

\DeclareMathOperator{\Spin}{Spin}

\DeclareMathOperator{\GL}{GL}

\DeclareMathOperator{\dists}{dist_{\mathfrak{s}}}

\numberwithin{equation}{section}

\def\dash{\leavevmode\unskip\kern0.18em--\penalty\exhyphenpenalty\kern0.18em}
\def\slash{\leavevmode\unskip\kern0.15em/\penalty\exhyphenpenalty\kern0.15em}

\def\balpha{\boldsymbol{\alpha}}
\def\balphas{\boldsymbol{\alpha}^\dagger}
\def\bomega{\boldsymbol{\omega}}
\renewcommand{\digamma}{\text{\foreignlanguage{greek}{\ddigamma}}}

\let\f\frac
\let\F\relax

\usetikzlibrary{calc}
\usetikzlibrary{shapes.misc}
\usetikzlibrary{shapes.symbols}
\usetikzlibrary{shapes.geometric}
\usetikzlibrary{snakes}
\usetikzlibrary{decorations}
\usetikzlibrary{decorations.markings}

\tikzset{
	root/.style={circle,fill=testcolor,inner sep=0pt, minimum size=2mm},
	dot/.style={circle,fill=symbols,draw=symbols,inner sep=0pt,minimum size=0.5mm},
	bdot/.style={circle,fill=symbols,draw=symbols,inner sep=0pt,minimum size=1mm},
	bdotsml/.style={circle,fill=symbols,draw=symbols,inner sep=0pt,minimum size=0.75mm},
	square/.style={regular polygon,regular polygon sides=4,fill=black,draw=black,inner sep=0pt,minimum size=1.2mm},
	wsquare/.style={regular polygon,regular polygon sides=4,fill=white,draw=black, inner sep=0pt,minimum size=1.2mm},
	squaresml/.style={regular polygon,regular polygon sides=4,fill=black,draw=black,inner sep=0pt,minimum size=0.9mm},
	wsquaresml/.style={regular polygon,regular polygon sides=4,fill=white,draw=black, inner sep=0pt,minimum size=0.9mm},
	eps/.style={circle,fill=white,draw=symbols,inner sep=0pt,minimum size=1mm},
	int/.style={circle,fill=black,draw=black,inner sep=0pt,minimum size=0.7mm},
	var/.style={circle,fill=black!10,draw=black,inner sep=0pt, minimum size=2mm},
	dotred/.style={circle,fill=black!50,inner sep=0pt, minimum size=2mm},
	generic/.style={semithick,shorten >=1pt,shorten <=1pt},
	dist/.style={ultra thick,draw=testcolor,shorten >=1pt,shorten <=1pt},
	testfcn/.style={ultra thick,testcolor,shorten >=1pt,shorten <=1pt,<-},
	testfcnx/.style={ultra thick,testcolor,shorten >=1pt,shorten <=1pt,<-,
		postaction={decorate,decoration={markings,mark=at position 0.6 with {\drawx}}}},
	keps/.style={semithick,shorten >=1pt,shorten <=1pt,densely dashed,->},
	kprimex/.style={semithick,shorten >=1pt,shorten <=1pt,densely dashed,->,
		postaction={decorate,decoration={markings,mark=at position 0.4 with {\drawx}}}},
	kernel/.style={semithick,shorten >=1pt,shorten <=1pt,->},
	multx/.style={shorten >=1pt,shorten <=1pt,
		postaction={decorate,decoration={markings,mark=at position 0.5 with {\drawx}}}},
	kernelx/.style={semithick,shorten >=1pt,shorten <=1pt,->,
		postaction={decorate,decoration={markings,mark=at position 0.4 with {\drawx}}}},
	kernel1/.style={->,semithick,shorten >=1pt,shorten <=1pt,postaction={decorate,decoration={markings,mark=at position 0.45 with {\draw[-] (0,-0.1) -- (0,0.1);}}}},
	kernel2/.style={->,semithick,shorten >=1pt,shorten <=1pt,postaction={decorate,decoration={markings,mark=at position 0.45 with {\draw[-] (0.05,-0.1) -- (0.05,0.1);\draw[-] (-0.05,-0.1) -- (-0.05,0.1);}}}},
	kernelBig/.style={semithick,shorten >=1pt,shorten <=1pt,decorate, decoration={zigzag,amplitude=1.5pt,segment length = 3pt,pre length=2pt,post length=2pt}},
	rho/.style={dotted,semithick,shorten >=1pt,shorten <=1pt},
	renorm/.style={shape=circle,fill=white,inner sep=1pt},
	labl/.style={shape=rectangle,fill=white,inner sep=1pt},
	xi/.style={circle,fill=symbols!10,draw=symbols,inner sep=0pt,minimum size=1.2mm},
	xix/.style={crosscircle,fill=symbols!10,draw=symbols,inner sep=0pt,minimum size=1.2mm},
	xib/.style={circle,fill=symbols!10,draw=symbols,inner sep=0pt,minimum size=1.6mm},
	xibx/.style={crosscircle,fill=symbols!10,draw=symbols,inner sep=0pt,minimum size=1.6mm},
	not/.style={circle,fill=symbols,draw=symbols,inner sep=0pt,minimum size=0.5mm},
cumu2n/.style={inner sep=3pt},
cumu2/.style={draw=red!80,fill=red!40},
cumu2b/.style={draw=blue!80,fill=blue!40},
cumu2nv/.style={inner sep=3pt},
cumu2v/.style={draw=red!80,fill=white,very thick},
cumu3/.style={regular polygon, regular polygon sides=3,draw=red!80,rounded corners=3pt,fill=red!40,minimum size=5mm},
cumu4/.style={regular polygon, regular polygon sides=4,draw=red!80,rounded corners=3pt,fill=red!40,minimum size=7mm},
cumu5/.style={regular polygon, regular polygon sides=5,draw=red!80,rounded corners=3pt,fill=red!40,minimum size=7mm},
	>=stealth,
	}
\makeatletter
\def\DeclareSymbol#1#2#3{%
	\expandafter\gdef\csname MH@symb@#1\endcsname{\tikzsetnextfilename{symbol#1}%
	\tikz[baseline=#2,scale=0.15,draw=symbols,line join=round]{#3}}%
	\expandafter\gdef\csname MH@symb@#1s\endcsname{\scalebox{0.75}{\tikzsetnextfilename{symbol#1s}%
	\tikz[baseline=#2,scale=0.15,draw=symbols,line join=round]{#3}}}%
	\expandafter\gdef\csname MH@symb@#1ss\endcsname{\scalebox{0.65}{\tikzsetnextfilename{symbol#1ss}%
	\tikz[baseline=#2,scale=0.15,draw=symbols,line join=round]{#3}}}%
	}
\def\<#1>{\ifmmode\mathchoice{\csname MH@symb@#1\endcsname}{\csname MH@symb@#1\endcsname}{\csname MH@symb@#1s\endcsname}{\csname MH@symb@#1ss\endcsname}\else\csname MH@symb@#1\endcsname\fi}
\makeatother

\DeclareSymbol{1}{0}{ \draw (0,0)  -- (0,1.5) node[bdot] {};} 
\DeclareSymbol{1F}{0}{\draw[black] (0,0)  -- (0,1.5) node[square] {}; \draw[black] (-0.22,0.75) -- (0.22,0.75);} 
\DeclareSymbol{1IF}{0}{ \draw[black] (0,0)  -- (0,1.5) node[square] {};} 
\DeclareSymbol{1A}{0}{\draw[black] (0,0)  -- (0,1.5) node[wsquare] {}; \draw[black] (-0.22,0.75) -- (0.22,0.75);} 
\DeclareSymbol{1IA}{0}{ \draw[black] (0,0)  -- (0,1.5) node[wsquare] {};} 

\DeclareSymbol{1IA}{0}{\draw[black] (0,0)  -- (0,1.5) node[wsquare] {};} 
\DeclareSymbol{2}{0}{\draw (-0.75,1.5) node[bdot] {} -- (0,0) -- (0.75,1.5) node[bdot] {};}
\DeclareSymbol{2F}{0}{\draw (-0.75,1.5) node[bdot] {} -- (0,0); \draw[black] (0,0) -- (0.75,1.5) node[square] {};
\draw[black] (0.571774, 0.651613) -- (0.178226, 0.848387);
}
\DeclareSymbol{2A}{0}{\draw (-0.75,1.5) node[bdot] {} -- (0,0); \draw[black] (0,0) -- (0.75,1.5) node[wsquare] {};
\draw[black] (0.571774, 0.651613) -- (0.178226, 0.848387);
}
\DeclareSymbol{2AF}{0}{\draw[black] (-0.75,1.5) node[wsquare] {} -- (0,0); \draw[black] (0,0) -- (0.75,1.5) node[square] {};
\draw[black] (0.571774, 0.651613) -- (0.178226, 0.848387);
\draw[black] (-0.571774, 0.651613) -- (-0.178226, 0.848387);
}

\DeclareSymbol{3}{0}{\draw (0,0) -- (0,1.5) node[bdot] {}; \draw (-0.860365,1.22873) node[bdot] {} -- (0,0) -- (0.860365,1.22873) node[bdot] {};}

\DeclareSymbol{2F2A}{-3}{\draw[black] (.55,0) node[squaresml] {} -- (0,-1); \draw[black] (0,-1) -- (-.55,0) ;  \draw (-1.1,1) node[bdotsml] {} -- (-0.55,0) ; \draw[black] (-0.55,0) -- (0,1) node[wsquaresml] {};
\draw[black] (0.0997568, -0.403616) -- (0.450243, -0.596384);
\draw[black] (-0.0997568, -0.403616) -- (-0.450243, -0.596384);
\draw[black] (-0.0997568, 0.403616) -- (-0.450243, 0.596384);
}

\DeclareSymbol{1F2A}{-3}{\draw[black] (.55,0) node[squaresml] {} -- (0,-1); \draw[black] (0,-1) -- (-.55,0) ; \draw[black] (-0.55,0) -- (0,1) node[wsquaresml] {};
\draw[black] (0.0997568, -0.403616) -- (0.450243, -0.596384);
\draw[black] (-0.0997568, -0.403616) -- (-0.450243, -0.596384);
\draw[black] (-0.0997568, 0.403616) -- (-0.450243, 0.596384);
}

\DeclareSymbol{2A2F}{-3}{\draw[black] (-.55,0) node[wsquaresml] {} -- (0,-1); \draw[black] (0,-1) -- (.55,0) ;  \draw[black] (1.1,1) node[squaresml] {} -- (0.55,0) ; \draw (0.55,0) -- (0,1) node[bdotsml] {};
\draw[black] (0.0997568, -0.403616) -- (0.450243, -0.596384);
\draw[black] (-0.0997568, -0.403616) -- (-0.450243, -0.596384);
\draw[black] (0.649757, 0.596384) -- (1.00024, 0.403616);
}

\DeclareSymbol{1A2F}{-3}{\draw[black] (-.55,0) node[wsquaresml] {} -- (0,-1); \draw[black] (0,-1) -- (.55,0) ;  \draw[black] (1.1,1) node[squaresml] {} -- (0.55,0) ;
\draw[black] (0.0997568, -0.403616) -- (0.450243, -0.596384);
\draw[black] (-0.0997568, -0.403616) -- (-0.450243, -0.596384);
\draw[black] (0.649757, 0.596384) -- (1.00024, 0.403616);
}

\makeatletter

\DeclareRobustCommand{\TitleEquation}[2]{\texorpdfstring{\StrLeft{\f@series}{1}[\@firstchar]$\if%
b\@firstchar\boldsymbol{#1}\else#1\fi$}{#2}}

\makeatother

\begin{document}

\title{A Dynamical Yukawa\TitleEquation{_{2}}{2} Model}

\author{Ajay Chandra$^{1}$, Martin Hairer$^{1,2}$, and Martin Peev$^{1}$}

\institute{Imperial College London, UK 
\and EPFL, Lausanne, Switzerland\\[.5em]
\email{\{a.chandra,m.hairer,m.peev21\}@imperial.ac.uk}}

\maketitle

\begin{abstract}
We prove local (in space and time) well-posedness for a mildly regularised version of the stochastic
quantisation of the Yukawa$_{2}$ Euclidean field theory with a self-interacting boson. 
Our regularised dynamic is still singular
but avoids non-local divergences, allowing us to use a version of 
the Da Prato--Debussche argument \cite{DPD2}.
This model is a test case for a non-commutative probability framework for formulating the kind of singular SPDEs
arising in the stochastic quantisation of field theories mixing both bosons and fermions. \\[.4em]
\end{abstract}

\setcounter{tocdepth}{2}
\tableofcontents

\section{Introduction}\label{sec:intro}

One approach to constructing interacting quantum field theories satisfying the Wightman axioms
was pioneered by Osterwalder and Schrader \cite{OS73G1,OS75}. The idea is to consider ``imaginary time'', hence
turning the ill-posed Feynman path integrals into a formal expression that has at least a chance
of being interpretable as a probability measure on a space of tempered distributions over space-time.
Amazingly, this operation can be reversed: given such a probability measure satisfying a number
of rather natural properties, as well as reflection positivity, it is possible to recover a
Hilbert space of states as well as a unitary evolution on that space and natural field operators.

Still, the construction of these ``interacting'' probability measures is far from trivial
and one of the milestones in this area was the construction of the $\Phi^4_3$ measure in the 70's
\cite{Feldman,GlimmJaffe}. In the 80's, Parisi and Wu \cite{ParisiWu} proposed to interpret these
measures as the invariant measures of an associated Markovian dynamic (the ``time'' appearing there
has no relation to physical time, which is treated as one of the spatial coordinates in this picture).
While this is a very appealing perspective, it led to rather little progress for about three
decades (but see notably
\cite{MR815192,AlbRock91,DPD2}) mainly due to a lack of understanding of the workings of renormalisation
for the stochastic PDEs arising in this way.

All this changed in the early 2010's with the appearance of several techniques
including regularity structures \cite{Hai14,BHZ19,CH16,BCCH21} and paracontrolled
analysis \cite{GIP15,GP17} that provide
an unequivocal meaning to these stochastic PDEs which is consistent with previous
works in constructive field theory. 
However, nearly all of these works
deal with purely \textit{bosonic} quantum field theories, thus excluding any of the
more realistic theories that combine both bosons and fermions. From a purely formal perspective,
the situation is quite well-understood \cite{Wein95,Tic99}: a non-interacting pair of particle \slash antiparticle
fields $(\psi,\bar \psi)$ can be described by a type of centred ``Gaussian'' generalised random field
such that
\begin{equ}\label{eq:free_fermions}
\scal{\psi(f)\psi(g)} = \scal{\bar \psi(f)\bar \psi(g)} = 0\;,\quad
\scal{\bar \psi(f)\psi(g)} = C(f,g)
\end{equ}
for some bilinear form $C(\bigcdot,\bigcdot)$, and with higher-order correlations given by the Wick rule.
The difference with the bosonic case,
however, is that these fields \textit{anticommute} in the sense that
\begin{equ}
\bar \psi(f)\psi(g) = - \psi(g)\bar \psi(f)\;,
\end{equ}
so that they can only be realised as observables on a non-commutative probability space.
In this case, a natural realisation is the Grassmann algebra $\Lambda(\fH)$ over some\footnote{One can essentially make any choice of $\fH$ here, see Remark~\ref{rem:covariance}. } Hilbert space
$\fH$.
Namely, $\Lambda(\fH)$\label{p:Grassmann} is the free algebra generated by $ \fH$ subject to the relations
$F \wedge G = - G\wedge F$ for all $F,G \in \fH$.
The problem of constructing interacting field theories with bosons and fermions can then be framed as starting from a state $\scal{ \,\bigcdot\, }$ for a Gaussian bosonic field $\phi$ and an independent Gaussian fermionic field $(\psi,\bar\psi)$ and then building the Gibbsian state
\begin{equ}\label{interacting_state}
\scal{ \, \bigcdot \,}_{V} = \frac{1}{Z} \scal{ \, \bigcdot \,  \exp[ -V(\phi,\psi,\bar{\psi}) ] } \;, 
\end{equ}
where $V(\phi,\psi,\bar{\psi})$ is an interaction potential or action with a non-quadratic dependence on the fields and $Z$ is a normalising factor so that $\scal{\1}_{V} = 1$. 
However, in important cases of interest \eqref{interacting_state} is ill-defined. 
Our focus in this article is analysis of the ultraviolet \slash small scale problems with \eqref{interacting_state}. 
The state $\scal{ \, \bigcdot \, }$ will be supported on spaces of distributional bosons and fermions which makes the non-linear functional $V$ ill-defined \dash requiring the use of renormalisation. 
As mentioned earlier, methods from stochastic analysis have provided robust tools for attacking this problem in the purely bosonic setting. 
\subsection{Stochastic Analysis for Fermions}
The most frequently used formalism in probability for working with fermions is the Grassmann--Berezin calculus \cite{Berezin1987}. 
While it has been a fundamental tool defining and computing correlation functions for fermionic fields, it does not seem well-suited to describe fermions as objects of infinite dimensional probability theory and analysis, and therefore unsuitable for studying fermions using stochastic analysis. 

In the recent article \cite{Gub20}, the authors performed the first steps towards fermionic stochastic analysis using the framework of non-commutative probability. 
We quickly review how non-commutative probability can be seen as a generalisation of classical (commutative) probability.  
In the latter setting, given a probability space $(\Omega,\mu)$, we can define a corresponding ``algebra of bounded observables'' $L^{\infty}(\Omega,\mu)$, that is the algebra of (equivalence classes of) bounded measurable complex-valued functions. 
In particular,  $L^{\infty}(\Omega,\mu)$ is a (commutative) $C^{*}$-algebra equipped with a state $\omega_{\mu} \colon L^{\infty}(\Omega,\mu) \rightarrow \C$ given by $f \mapsto \int f\; \mathrm{d} \mu$. 
The non-commutative probability setting allows one to take non-commutative algebras of observables, the basic ingredients of a non-commutative probability space being a $C^{*}$-algebra $\CA$ (playing the 
role of $L^{\infty}(\Omega,\mu)$) equipped with a positive definite state $\omega \colon \CA \to \R$ 
(playing the role of expectation). 
In practice, one often realises $\CA$ as an algebra of \emph{bounded} operators on some Hilbert space and the state $\omega$ can be written in terms of this representation. 

In Section~\ref{sec:CAR} we will review how one can define a Gaussian fermionic process with reproducing kernel Hilbert space $\mfH$ by choosing our $C^{*}$-algebra to be a CAR\footnote{Canonical Anticommutation Relations} algebra $\CA\F = \CA(\mfH)$. 
One can realise $\CA\F$ as an algebra of bounded operators generated by creation and annihilation operators on the antisymmetric Fock space generated by $\mfH$ and the state $\omega\F$ is given by $A \mapsto \scal{\1,A \1}$ where $\1$ denotes the vacuum of the Fock space. 
The fermionic fields\footnote{The fields themselves are operator-valued distributions that map $\cD \ni f \mapsto  \psi(f), \bar{\psi}(f) \in \CA\F \otimes \C^2$ where $\cD$ is some space of test functions. The same is true of the regular white noise, where $\xi(z)$ is not a random variable but instead a distribution that takes test functions to random variables by mapping $f \mapsto \xi(f)$.} $\psi, \bar{\psi}$ will then be generators of a Grassmann subalgebra of $\CA\F$. 

\begin{remark}
One difference a reader may find between our approach and that of  \cite{Gub20} is that theirs is dual to the one we use in this article. 
Given a probability space $(\Omega,\mu)$ and a target space $V$, a $V$-valued 
random variable $X$ is usually seen as a measurable map $\Omega \to V$.
However, there is a dual perspective which identifies the random variable $X$
with the map $X^{\ast}$ taking
 bounded measurable functions on $V$ to bounded measurable functions 
on $\Omega$ by precomposing with $X$. 
In the fermionic setting the role of bounded measurable functions on $V$ is played by the Grassmann algebra $\Lambda(V)$ 
and that of the bounded measurable functions on $(\Omega,\mu)$ is played by some 
$C^*$-algebra $\CA$ endowed with a state $\omega\colon \CA \to \C$.
In \cite{Gub20}, the role of (totally anticommuting) random variables in $V$ is then
played by algebra homomorphisms $\Phi \colon \Lambda(V) \rightarrow \CA$.

However, we will not take up this dual approach, and work directly on our CAR algebra $\CA\F$ along with extensions of this algebra. 
We find this more convenient for field theory where $V$ is infinite dimensional, for combining probabilistic and analytic arguments, and for working with bosons and fermions simultaneously (without integrating out either). 
\end{remark}

Some of the key achievements of \cite{Gub20} are formulating stochastic differential equations and It\^{o} calculus for fermions which are used to formulate Langevin dynamics for a fermionic field theory, however, this field theory has been regularised at small scales to avoid the need for renormalisation. 

Our main objective in the present article is to investigate how to develop a local well-posedness theory for singular fermionic stochastic partial differential equations. 
An immediate obstacle we encounter is that non-commutative probability does not quite generalise all aspects of classical probability 
as the latter theory makes heavy use of the fact that there is underlying measure space of realisations $(\Omega,\mu)$. 
For algebras of bounded random variables, one can take $\CA = L^{\infty}(\Omega,\mu)$. 
However, most random variables appearing in classical probability are unbounded,
so one usually works in a much larger algebra $\CM(\Omega,\mu)$ of all (equivalence classes modulo 
$\mu$-null sets of) 
measurable maps from $\Omega$ to $\C$ equipped with the topology of convergence in probability, 
or in the spaces $L^{q}(\Omega,\mu) \subset \CM(\Omega,\mu)$ 
for $1 \leqslant p < \infty$. 
There are related notions in non-commutative probability, such as unbounded operators affiliated 
with $\CA$ or non-commutative $L^{q}$ spaces \cite{Tak03,dSil19} but these frameworks are 
currently lacking for our purposes and might be ill-suited to apply to non-linear problems.

When aiming for a robust local well-posedness theory for (commutative) non-linear SDEs \slash SPDEs it is natural to leverage $\Omega$ even more strongly and work ``pathwise'' \dash that is one performs all the analysis after freezing an arbitrary \text{point} $\mcb{p} \in \Omega$. 
This allows one to treat any stochastic driving noise as fixed and bounded and also bypasses the issue that non-linear operations don't behave well with $L^{q}(\Omega,\mu)$ spaces for $q < \infty$. 
This point of view is well-suited for the analysis of rough \slash singular equations, there fixing $\mcb{p} \in \Omega$ allows one to treat an enhanced noise (including renormalised products) as fixed and bounded.

The problem that appears in non-linear singular SPDEs is that, while fermionic Gaussians can be thought of as bounded random variables, this is typically not true for the higher order Gaussian chaoses which appear when renormalising non-linear expressions of rough Gaussian fields. 
While these objects can be constructed as affiliated unbounded operators or elements of a non-commutative $L^q$ space, one is no longer able to solve equations in which such objects appear. 
In particular, the approach to solving non-linear stochastic differential equations in \cite{Gub20} 
appears to depend crucially on the fact that $\CA$ is a Banach algebra, or at least an $m$-convex topological algebra.

Working ``pathwise'' in $\Omega$ in the commutative setting is analogous to the method of ``localisation'' in algebraic geometry.
The primary contribution of our article is the development of a ``localisation'' approach to non-commutative 
probability (drawing on \cite{DV75,DV77}) and its application to obtain a local (in space and time) well-posedness theory for 
an SPDE that includes fermions and requires the renormalisation of products involving fermions. 

Finally, we mention that \cite{FGV22} has used stochastic analysis methods to  construct purely fermionic models that are ill-posed at small scales, in particular they are able to completely remove small scale cut-offs (and also work in infinite volume). 
Their approach is not to use a Langevin dynamic but to instead use the Polchinski flow. 
In this flow, the role of time is played by scale and the evolving objects are effective actions $V_{s}$ and effective fields $\Psi_{s}$, these describe a marginal of the full interacting theory with small scales integrated out \dash recovering the full interacting field would then correspond to taking $s \rightarrow \infty$. 
Since the Polchinski flow is written in terms of effective actions and fields, one 
then does not need to explicitly construct any renormalised non-linearities in this approach. 

\subsection{Overview of Main Results}

\subsubsection{Localising Non-Commutative Probability}

Our localisation approach to non-commutative probability aims to replace $\mathcal{A}\F$ by a 
Fr\'echet space $\cA\F$ with a topology weak enough to include as elements the unbounded operators 
that arise when we renormalise products. 
We also do not want $\cA\F$ to be too big, we hope that is has a topology strong enough to help one prove the convergence of correlation functions if one can prove sufficiently strong bounds related to global well-posedness of our dynamic. 
Finally, to allow us to prove local well-posedness of our dynamic in $\cA\F$, we want this Fr\'echet space to be topologised by a family of \emph{multiplicative} seminorms. 
In non-commutative algebraic geometry, irreducible representations can serve as analogues of points \cite{KR00,Ros95}.
We also know that while renormalised polynomials of the fermion fields are unbounded operators on Fock space, each such operator can be controlled on subspaces of Fock space with finite particle number. 

With this in mind our ``points'' will be indexed by $\mfG(\mfH)$, that is the collection of finite dimensional subspaces  $b \subset \mfH$. 
For each such $b$ we will localise by working in a representation of infinite dimensional space on $\CA(b)$, the finite dimensional CAR algebra generated by $b$. 
An obstacle is that our infinite dimensional CAR algebra does not admit such representations because of its relations, but here we follow the approach of \cite{DV75} and replace $\mathcal{A}$ with the free $\star$-algebra $\widehat{\mfA}$ generated by $\mfH$ \dash this amounts to forgetting the canonical anticommutation relations.  
We then have ``projection'' representations $\pi_{b}: \widehat{\mfA} \rightarrow \CA(b)$ where $\pi_{b}(a)$ should be thought of as evaluating $a$ at $b$. 

We (partially) recover\footnote{This is made precise in Lemma~\ref{lem:commutation}.} the anticommutation relations by quotienting $\widehat{\mfA}$ by the ideal of elements that vanish everywhere (that is, those $a \in \widehat{\mfA}$ with $\pi_{b}(a) = 0$ for all $b$) and completing this quotient with respect to the seminorms $\big( \|\bigcdot\|_{n}  \big)_{n \in \N}$ where
\begin{equ}
	\|a\|_{n} \eqdef \sup \big\{ \| \pi_{b}(a)\|_{\CA(b)} \, \big| \, b \in \mfG(\mfH),\; \dim(b) \leqslant n \big\}\;.
\end{equ}
The resulting space $\cA$ is a locally $C^*$-algebra\footnote{A locally $C^*$-algebra is a complete topological $\star$-algebra whose topology is generated by a family of sub-multiplicative seminorms, each satisfying the $C^*$-identity.}  and serves as a non-commutative algebra of unbounded observables. 
It contains a subalgebra $\mfA_{\infty}$ of bounded observables which is itself a $C^*$-algebra and is an extension\footnote{In fact, this extension is ``central'' in the sense that the kernel of the 
map $\mfA_{\infty}\to \CA$ is an ideal generated by some elements in the center of $\mfA_{\infty}$.} 
of the original CAR algebra $\CA$.
The space $\cA$ carries the original fermionic fields $(\bpsi,\bbpsi)$ but unlike $\CA$ it also carries renormalised polynomials of $(\bpsi,\bbpsi)$. 
The following theorem summarises the key properties of our construction.

\begin{theorem}\label{thm:main_thm1}
There is a locally $C^{*}$-algebra $\cA\F$, equipped with $C^{*}$-seminorms $\big( \|\bigcdot\|_{n} \big)_{n=1}^{\infty}$ with the following properties. 
\begin{enumerate}
\item $\cA\F$ contains a dense $C^*$-algebra $\mfA_{\infty} = \left\{ a \in  \cA\F \, \big| \, \sup_{n} \|a\|_{n} < \infty \right\}$, which carries a surjective $C^{*}$-homomorphism $\digamma \colon \mfA_{\infty}  \rightarrow \CA\F$\footnote{The letter $\digamma$ is pronounced ``digamma''.}. 
Under $\digamma$ the vacuum state $\omega\F$ on $\CA\F$ pulls back to give a densely defined state $\bomega\F$ on $\cA\F$ and one has canonical isomorphisms\footnote{Note here that $\CL^{2}(\CA,\omega)$ denotes the non-commutative $\CL^{2}$ spaces on $\CA$ with state $\omega$, not a Lebesgue space of functions on $\CA$, see Section~\ref{sec:GNS} for definitions.} 
\begin{equ}
	\CL^{2}(\CA\F,\omega\F) \simeq \CL^{2}(\mfA_{\infty},\bomega\F)\;.
\end{equ}
\item Sequences of polynomials $P_{k}(\psi,\bar\psi) \in \CA\F$ of the fermionic fields $\psi,\bar\psi$ that converge in  $\CL^{2}(\CA\F,\omega\F)$ can be be pulled back to polynomials $\mathbf{P}_{k}(\bpsi,\bbpsi) \in \mfA_{\infty}$ that converge in $\cA\F$.
\end{enumerate}
\end{theorem} 

\begin{proof}
The first statement is the content of Theorem~\ref{thm:ContinuityRep}, and Theorem~\ref{thm:local_L2}. 
The second statement is obtained in Proposition~\ref{proposition:Wick_Product}. 
\end{proof}

$(\cA\F,\bomega\F)$ is a non-commutative probability space with observables that are bounded locally (in every seminorm $\|\bigcdot\|_{n}$) but not globally ($\sup_{n} \|\bigcdot\|_{n}$ need not be finite). 
Statement 1 puts some limitation on how ``big'' $\cA\F$ is, while statement 2 tells us our topology is weak enough to allow the convergence of renormalised products. 
We localise by choosing to work with a fixed seminorm $\|\bigcdot\|_{n}$ and this serves as a substitute for the pathwise approach that is lost in going from commutative to non-commutative probability. 

\begin{remark}\label{rem:weak_topology}
We point out that the ``localised non-commutative probability'' formalism we try to develop in this article has a fair amount of freedom in how one chooses to localise. 
Theorem~\ref{thm:main_thm1} above (and also Theorem~\ref{thm:main_thm} below) hold for a very wide class of collections of $C^{*}$-seminorms $\big( \|\bigcdot\|_{n} \big)_{n=1}^{\infty}$ \dash see the discussion starting from Definition~\ref{def:filtration} and ending with \eqref{eq:seminorm_def}. 
 
Different allowable choices of families of seminorms give rise to different topologies and thus very different spaces $\cA$ \dash some of which are likely unreasonable.
In particular, some of these topologies make it unnecessary to perform renormalisation, and are therefore likely to be too weak to be of real interest.  

A natural question that is not fully answered in this work is what a reasonable choice of seminorms and corresponding space $\cA$ is. 

However, one thing we do show is that with a good choice of state, i.e.\ a faithful one, we can define a ``local'' non-commutative $\CL^2$-space, which imposes that one must renormalise appropriately. In particular, if one does not renormalise then the products diverge and are not in this $\CL^2$-space, see Remarks~\ref{rem:localL2} and \ref{rem:topology_renormalisation}.

\end{remark}

\subsubsection{The Higgs--Yukawa\TitleEquation{_2}{2} Model}

We now present the specific model we investigate in this article and summarise our main result 
in Theorem~\ref{thm:main_thm} \dash a more precise formulation of our model and main result can 
respectively be found in Section~\ref{sec:the_model} and Theorem~\ref{thm:precise_main_thm}. 

Our aim is to perform the first steps in the stochastic quantisation of the
full Higgs--Yukawa$_2$ model describing a simple interaction of two Dirac fermions mediated via a self-interacting (real)
scalar boson. Formally this is described by the Euclidean action
\begin{equ}\label{eq:Yukawa_Action}
	\int\limits_{\R^2} \Big(\frac{1}{2} |\nabla \phi|^2 + \frac{m^2}{2} \phi^2 + \bilin{\bar u, (-\snabla + M) u}+ g \phi \bilin{\bar u,u} + \frac{\lambda}{4} \phi^4 \Big)\,\d x\;,
\end{equ}
where $\phi$ is a real\footnote{In the context of stochastic quantisation below, this means that the bosonic noise driving the dynamics is a random variable on the space of real-valued distributions.} scalar field, $u$ and $\bar{u}$ are two independent, two-dimensional Euclidean Dirac spinor fields (again, corresponding to a fermionic particle and antiparticle), $\bilin{\bigcdot,\bigcdot}$ is the bilinear (not sesquilinear) extension of the usual scalar product on $\R^2$ to $\C^2$, $\snabla$ is the Dirac operator given by
\begin{equ}[eq:DiracOp]
	\snabla \eqdef \begin{pmatrix}
		 0 & -\partial_1+i \partial_2\\
		-\partial_1-i\partial_2 & 0
	\end{pmatrix} \; ,
\end{equ}
and the constants are $g \in \R$ and $m^{2}, M > 0$.

This model has been studied since the earliest days of constructive quantum field theory (QFT) and was first constructed in the Hamiltonian \slash Minkowski setting by Glimm--Jaffe  \cite{GJ70, GJ71} and Schrader \cite{Schrader72}.
It was later formulated in the language of Euclidean field theory by Osterwalder--Schrader in \cite{OS73} and has been heavily studied as a paradigmatic model mixing bosons and fermions \cite{Seiler74, SeilerSimon75, CooperRosen1977}.

From \eqref{eq:Yukawa_Action} one can formally derive the dynamical Yukawa$_2$
model given by the system of stochastic partial differential equations (SPDEs)
\begin{equs}
\label{eq:Model_V1}
		\partial_t \phi &= (\Delta-m^2) \phi - g  \bilin{\bar u,u} - \lambda \phi^3 + \xi	\\
		\partial_t u &= (\snabla-M) u - g \phi u + \chi\\
		\partial_t \bar u & =  \overline{ (\snabla-M)^\dagger} \bar u - g \phi \bar u + \bar\chi  \\
		& =  (-\overline{\snabla}-M) \bar u - g \phi \bar u + \bar\chi\;.
\end{equs}
Here $\overline{\snabla}$ denotes the complex conjugate (but not adjoint) of $\snabla$, $\dagger$ denotes the adjoint operation,
$\chi$ and $\bar{\chi}$ are two fermionic space-time white noises\footnote{Here we view the fermionic fields as $\mathcal{A}$-valued distributions over space-time.} for the particle and antiparticle \dash these are free fermions as in  \eqref{eq:free_fermions} where, for $f,g \in \cD(\R \times \T^2; \C^2)$,
\begin{equ}
 C(f,g) \eqdef \omega\big(\bar{\chi}(f) \chi(g)\big) = \int_{\R} \int_{\T^2} \scal{f(t,x),g(t,x)}_{\C^2} \; \mrd x\; \mrd t\;,
\end{equ}
and $\xi$ is a bosonic (i.e.\ ``standard'') space-time white noise.
In principle, the ``law'' of the ``stationary''
solutions of \eqref{eq:Model_V1} should be consistent with the Euclidean formulation of the Yukawa$_2$ model studied earlier.

The recent work \cite{Gub20} discussed in the first section of the introduction also considered the Yukawa$_2$ model (but in the smooth regime\footnote{With a fixed continuum regularisation (in field theory language, with a UV cut-off).}) via a stochastic quantisation approach.
We point out that \cite{Gub20}, following the approach of \cite{Les87}, integrate out the bosons in the field theory and then perform the stochastic quantisation of the resulting ``effective'' purely fermionic theory.\footnote{Note that \cite{Seiler74, SeilerSimon75, CooperRosen1977} went in the opposite direction and started their analysis with a bosonic field theory where the fermions had been integrated out.}
Our method has no problem dealing with the $\phi^4$ potential term in \eqref{eq:Yukawa_Action}, which prevents one from easily integrating out the bosons from the start.

In the present article we do not study the dynamic \eqref{eq:Model_V1} itself, but a similar equation with two main differences:
\begin{itemize}
\item The first difference is also found in \cite{Gub20}, this modification gives a different dynamic but, modulo a change of variable, gives the same invariant state \dash see Remark~\ref{rem:stoch_quantisation} for more details.
In order to have equations where all the components behave well under parabolic scaling we (i) change the bilinear form that is used to define both the fermionic white noise and the functional derivative of the action \eqref{eq:Yukawa_Action} in our stochastic quantisation and (ii) we change the fermionic variables, introducing $(\upsilon,\bar{\upsilon})$ where
\begin{equ}\label{eq:Replace_Spinors}
	u = (\snabla + M) \upsilon \quad \text{and} \quad \bar u = (-\overline{\snabla} + M) \bar\upsilon\;.
\end{equ}
\item The second change is more substantial.
We improve the regularity of the fermionic noises so that the equation remains singular but can be treated by a Da Prato--Debussche type perturbative ansatz.
We still have to carry out a renormalisation procedure for our equations, but all the renormalisations needed are Wick renormalisations.
With this change, our dynamic is \textit{not} a stochastic quantisation of the action \eqref{eq:Yukawa_Action}.
If we did not introduce this regularisation, our analysis would involve more difficult divergences\footnote{This would mean that we have to go beyond Wick renormalisation. In particular, without this regularisation the equation would look like bosonic $\Phi^4_3$ whereas with our regularisation it looks more like bosonic $\Phi^4_2$.} which would require us to employ a fermionic extension to the framework of \cite{Hai14} \dash this is left for future work.
\end{itemize}
The actual dynamic we then investigate is
\begin{equs}\label{eq:lin_sig_model}
    \partial_t \phi &= (\Delta-m^2) \phi - g  \bilin{(-\overline{\snabla}+M)\bar \upsilon , (\snabla+M)\upsilon} - \lambda \phi^3 +  \xi,	\\
    \partial_t \upsilon &=  (\Delta-M^2) \upsilon - g  \phi (\snabla+M) \upsilon  + \psi,\\
    \partial_t \bar\upsilon & =   (\Delta-M^2) \bar\upsilon- g \phi (-\overline{\snabla}+M)\bar\upsilon + \bar\psi\;,
\end{equs}
where  $\xi$ is again a bosonic space-time white noise, and $(\psi, \bar{\psi})$ are free Dirac fermions just like \eqref{eq:free_fermions}, where, for $f,g \in \cD(\R \times \T^2 ; \C^2)$ we set
\begin{equ}
	C(f,g) \eqdef \omega\big(\bar{\psi}(f) \psi(g)\big)=
	\int_{\R}
	\int_{\T^2}
	\scal{f(t,x) , (\CQ g)(t,x) }_{\C^2}\; \mrd x\; \mrd t \;,
\end{equ}
for $\CQ \eqdef (-\Delta+M^2)^{-\delta/2} ( -\snabla+M)^{-1}$ and $\delta > 0$ represents the regularisation described in the second bullet point above.
In particular, if $\delta = 0$, then one would formally expect $\bigl(\phi, \bigl( (\snabla +M)\upsilon,  (-\overline{\snabla}+M) \bar{\upsilon}\bigr) \bigr)$ to admit the Higgs--Yukawa$_2$ state as its equilibrium state.

For the rest of the article we fix a smooth, compactly supported function $\rho \colon \R\times \R^2  \rightarrow [0,1]$ integrating to $1$ and write $\rho_{\eps}(t,x) \eqdef \eps^{-4}\rho(\eps^{-2}t,\eps^{-1}x)$.
We define $\xi_{\eps} \eqdef \xi \ast \rho_{\eps}$, $\bar{\psi}_{\eps} \eqdef \bar{\psi} \ast \rho_{\eps}$, and $\psi_\eps \eqdef \psi \ast \rho_{\eps}$. Our main result is as follows.

\begin{theorem}\label{thm:main_thm}
	There exist constants $\big(C^1_{\eps}\big)_{ \eps \in (0,1]}$ and $\big(C^2_{\eps}\big)_{ \eps \in (0,1]}$, such that, for each $n \in \N$, the local in time solutions $\big(\phi_{\eps},( \upsilon_{\eps},\bar{\upsilon}_{\eps}) \big)$ to the regularised system of equations
\begin{equs}[eq:Yukawa_eps]
		\partial_t \phi_{\eps} &= (\Delta-m^2 + C^1_\eps) \phi_{\eps} - g\Big(  \bilin{\bigl(-\overline{\snabla}+M \bigr) \bar \upsilon_{\eps} \; , \bigl(\snabla + M \bigr)\upsilon_{\eps}} - C^2_\eps \Big) - \lambda \phi_\eps^3  +  \xi_{\eps},	\\
		\partial_t \upsilon_{\eps} &=  (\Delta-M) \upsilon_{\eps}- g  \phi_{\eps} \bigl( \snabla + M \bigr) \upsilon_{\eps}  + \psi_{\eps} \; ,\\
		\partial_t \bar\upsilon_{\eps} & =  (\Delta-M) \bar\upsilon_{\eps}- g \phi_{\eps} \bigl( -  \overline{\snabla} + M \bigr) \bar\upsilon_{\eps} + \bar\psi_{\eps}\;,
\end{equs}
seen as random fields with respect to the Bosonic space-time white noise $\xi$ and taking values in $\CA\F$, can be lifted to random local in time solutions in $\cA_{n}$ that converge in probability on $\cA_{n}$ as $\eps \downarrow 0$. 

Here $\cA_{n}$ is the Banach space obtained by quotienting $\cA\F$  by the kernel of $\|\bigcdot\|_{n}$ and completing with respect to $\|\bigcdot\|_{n}$. 
\end{theorem}
\begin{proof}
	This is the content of Theorem~\ref{thm:precise_main_thm}.
\end{proof}

\begin{remark}
	In the above theorem, fixing $n$ localises in the fermionic noise. 
	Note that the existence time of our solution may depend on our initial data, our realisation of the bosonic noise, and 
	on our choice of $n$. This is similar to the fact that the local existence time for a regular stochastic PDE may depend on a suitably chosen
	notion of ``size'' for the realisation of the driving noise.
	Obtaining local in time solutions in $\cA\F$ instead of $\cA_{n}$ likely requires making progress on global well-posedness of the dynamic, see Remark~\ref{rem:NotGlobSol} and Section~\ref{subsec:original_eqn}. 
\end{remark}

\subsection*{Acknowledgements}

{\small
MH gratefully acknowledges support by the Royal Society through a research professorship, RP\textbackslash R1\textbackslash 191065.
MP gratefully acknowledges support by the EPSRC through the ``Mathematics of Random Systems'' CDT EP/S023925/1. 
AC gratefully acknowledges partial support by the EPSRC through EP/S023925/1. 
}

\subsection{Notations and Conventions}\label{subsec:notation}
Unless we explicitly state otherwise, all of our vector spaces are \textit{complex}.
Accordingly, by ``scalar'' we will usually mean complex-valued.
All of our algebras will be associative. 

We use the notation $[n] \eqdef \{1,\dots, n\}$ for $n \in \N$.
Given a vector space $V$ we write $\bone = \bone_{V}$ to denote the identity map on $V$.
We overload notation, and for a unital algebra $A$ we also write $\bone = \bone_{A}$ for unit of $A$.
For a topological vector space $V$ we denote by $\CB(V)$ the space of bounded linear maps $V \to V$.

Given any measurable space $S$ we write $\cM(S)$ for the space of measurable scalar functions on $\Omega$
Given any measure space $(S,\mu)$ we write $\CM(S,\mu)$ for equivalence classes of elements of $\cM(S)$ which are equal $\mu$-almost everywhere. We view $\CM(S,\mu)$ is a topological vector space equipped with the topology of convergence in probability. 

For $q \in [1,\infty)$ we write $L^{q}(S,\mu)$ for the standard $L^{q}$ space of (equivalence classes) of scalar functions.
We define the vector space $L^{\infty -}(S,\mu) \eqdef \bigcap_{q\geqslant 1} L^{q}(S,\mu)$, which is a
locally convex topological vector space when equipped with the family of norms $\big( \| \bigcdot \|_{L^q} \big)_{q\geqslant 1}$. 
We also drop $\mu$ from the notation when it is clear from context.

Given a complete, locally convex space $E$ with a collection of seminorms $\mfP$, we write $\cL^{q}(S,\mu ; E)$ for the closure of the $E$-valued measurable step-functions on $S$  with respect to the topology induced by the seminorms
\begin{equ}[e:defLq]
	\| f \|_{\mfp, q} \eqdef \biggl( \int_S \bigl(\mfp(f(x))\bigr)^q \d \mu (x) \biggr)^{\frac{1}{q}}
\end{equ}
for all $\mfp \in \mfP$. We write $L^{q}(S,\mu ; E)$ for the set of equivalence classes of functions in $\cL^q(S,\mu;E)$ that are equal almost everywhere.
When $\mfP$ is countable (that is, $E$ is a Fr\'echet space) we have that $L^{q}(S,\mu ; E)$ is complete since one has to take into account only countably many seminorms and the usual measure theoretic constructions carry through, cf.\ \cite[Section~46.1]{Trev67}.\footnote{When $E$ is not separable, 
this is in general smaller than the space of all appropriately integrable $E$-valued measurable functions.
The functions we defined above are usually called strongly measurable, cf.\ \cite{Hyt16}.} 
We again write $L^{\infty - }(S,\mu; E) \eqdef \bigcap_{q\geqslant 1} L^q(S,\mu; E)$.
We analogously define $\CM(S,\mu;E)$. 

Given an algebra $\CA$ and state $\omega$ we will also use the notation $\CL^2(\CA,\omega)$, see Section~\ref{sec:GNS}. 

We write $\spacetime \eqdef \R \times \T^{2}$ for the primary space-time domain we work on in this article.
We sometimes also work on $\R^3$ as our space-time, where the first coordinate denotes time,
so the corresponding spatial domain $\T^2$ is replaced with $\R^2$. We also introduce the subsets $\spacetime_T \eqdef [0,T] \times \T^2$ and $\spacetime_+ \eqdef \R_+ \times \T^2$, where $\R_+ \eqdef [0,\infty)$.

Given an open domain $A$ in space(-time) and a topological vector space $E$, we write  $\cD(A;E)$ for the space of smooth, compactly
 supported $E$-valued functions on $A$.
We write $\cD'(A;E)$ for the corresponding space of $E$-valued distributions, that is continuous linear maps $\cD(A;\C) \to E$.
We write $\cC(A;E)$ for the space of continuous $E$-valued functions on $A$ and, for $r \in \N$, $\cC^r(A;E)$ for the space of $r$-times continuously differentiable $E$-valued functions on $A$. When $E = \C$ we drop $E$ from the notation.

When we need to quantify regularity on space(-time) we will use H\"older--Besov spaces which we quickly recall now.
In the space-time setting we scale space and time parabolically\footnote{In the language of regularity structures, we use the space-time scaling  $\s = (2,1,1)$}
that is for $(t,x_1,x_2) \in \R^3$ and  $\lambda > 0$ we set
\begin{equ}
	\lambda^\s (t,x_1,x_2) \eqdef  (\lambda^2 t, \lambda x_1, \lambda x_2)\;.
\end{equ}
We use an associated parabolic ``norm'' by setting, for $z = (t,x_1,x_2) \in \R^3$, 
\begin{equ}
	|z|_\s \eqdef \sqrt{|t|} + |x_1|+|x_2|\;,
\end{equ}
so that $|\lambda^\s z|_\s = \lambda |z|_\s$.
With this scaling, $d_{\s} \eqdef 2 + 1 + 1 = 4$ is the scaling dimension of our space-time, in particular the function $|z|_{\s}^{j}$ is locally integrable near $0$ if and only if $j > -d_{\s}$.

For $\alpha \in (0,1)$, we define $\CC^\alpha_{\s}(\R^3)$ to be the space of all functions $f$ on $\R^3$ such that, for every compact set $\mathfrak{K} \subset \R^3$,
\begin{equ}
\sup_{\substack{z,z' \in \mathfrak{K}\\ z \neq z'}} \frac{|f(z) - f(z')|}{|z-z'|_\s^{\alpha}} < \infty\;.
\end{equ}
For $\alpha  < 0$, we define $\CC^\alpha_{\s}(\R^3)$ to consist of all distributions $\xi \in \cD'(\R^3)$ such that, for every compact set $\K \subset \R^3$, one has
\begin{equ}
\sup
\left\{
\lambda^{- \alpha}  | \xi \big( \CS^\lambda_{\s,z} \eta \big)| \, 
\middle| \
\begin{array}{c}
\eta \in \cD(\R^3),\ \|\eta\|_{\cC^r} \leqslant 1,\; \\
z \in \K
\end{array}
\right\}
<
\infty\;,
\end{equ}
where we write $(\CS^\lambda_{\s,z} \eta)(w) \eqdef \lambda^{-d_{\s}} \eta((\lambda^{-1})^{\s} (w-z) )$ and $r = - \lfloor \alpha \rfloor $.
We write $\CC^\alpha_{\s}(\spacetime)$ for the space obtained when we also enforce periodicity in space.

We analogously write $\CC^{\alpha}(\T^{2}) \slash \CC^{\alpha}(\R^{2})$ for the corresponding H\"older--Besov spaces of functions \slash distributions of a spatial variable \dash here we use the regular Euclidean scaling $(1,1)$ for our two dimensions of space so we drop the scaling index $\s$.

In Appendix~\ref{sec:Holder_Space} we define vector-valued H\"older--Besov spaces for general $\alpha \in \R$ in more detail along with recalling the Schauder estimate in Theorem~\ref{SchdrThm} and the Young multiplication estimate Theorem~\ref{YngMultThm}.

We fix a smooth, radially symmetric, compactly supported function $\rho \colon \R^3 \to \R$ that integrates to $1$ as a mollifier, and define
\begin{equ}[eq:mollified_def]
	\rho_\eps \eqdef \CS^{\eps}_{\s, 0} \rho \; .
\end{equ}

Finally, we will use $\Omega$ to refer to the probability space for the bosonic space-time white noise.

\section{Operator Algebras and Fock Spaces}

\subsection{Basic Definitions of Operator Algebra Theory}
\label{section:AlgebraDefn}

Throughout this paper all of our algebras will be unital. For elements $a$, $b$ of an algebra $\CA$ we write 
\begin{equ}{}
	[a,b]_{\pm} \eqdef ab \pm ba \; .
\end{equ}

A topological algebra $\CA$ is an algebra that is a topological vector space with a multiplication that is jointly continuous in both arguments. 
We add the attributes Hausdorff, locally convex (LC), complete, Fr\'echet if the underlying topological vector space satisfies the corresponding conditions.

We say that a topological algebra is a topological $\star$-algebra if it is equipped with a continuous antilinear\footnote{Recall a map $M$ on a complex vector space $V$ is antilinear if, for every $\lambda \in \C$, $u,v \in V$, one has $M(\lambda u + v) = \bar{\lambda} M(u) + M(v)$.} involution $a \mapsto a^{\dagger}$ that satisfies $(ab)^{\dagger} = b^{\dagger} a^{\dagger}$.

A locally convex topological algebra $\CA$ is ``locally $m$-convex'' if and only if there exists a family of seminorms $\mfP$ inducing the topology of $\CA$, such that for all $a,b \in \CA$ and all $\mfp \in \mfP$
\begin{equ}[eq:m-convexity]
	\mfp(ab) \leqslant \mfp(a)\mfp(b) \; .
\end{equ}
Generically, in a LC algebra the continuity of multiplication only requires that for each $\mfp \in \mfP$, there exist $\mfq, \mfr \in \mfP$ such that $\mfp(ab) \leqslant  \mfq(a) \mfr(b)$. 
The stronger property \eqref{eq:m-convexity} will be crucial in allowing us to formulate our SPDE as a well-posed fixed-point problem.

A Banach algebra $\mathcal{A}$ is a locally $m$-convex algebra that is also a Banach space with a norm $\| \bigcdot \|$ for which one has, for all $a,b \in \CA$, 
\begin{equ}
	\|ab\| \leqslant \|a\| \|b\| \; .
\end{equ}
A $C^*$-algebra $\CA$ is a Banach $\star$-algebra where the $\star$-operation is an isometry and in addition satisfies the relation
\begin{equ}\label{cstar-identity}
	\| a^\dagger a\| = \|a\|^2
\end{equ}
for all $a \in \CA$. Equation~\eqref{cstar-identity} is called the $C^*$-identity.

A complete, locally $m$-convex $\star$-algebra $\cA$ is called a locally $C^*$-algebra when its topology can be induced by an $m$-convex family of seminorms $\mfP$, which satisfy for all $a \in \cA$ and all $\mfp \in \mfP$
\begin{equ}
	\mfp(a^\dagger ) = \mfp(a) \quad \text{and} \quad \mfp(a^\dagger a) = \mfp(a)^2 \; .
\end{equ}
Such seminorms are called $C^*$-seminorms. A locally $C^*$-algebra $\cA$ contains a dense $C^*$-algebra 
\begin{equ}
	\mfA_\infty \eqdef \Bigl\{ a \in \cA \, \Big| \, \|a\|_\infty \eqdef \sup_{\mfp \in \mfP} \mfp(a) < \infty \Bigr\}
\end{equ}
with $C^*$-norm $\|\bigcdot\|_\infty$, cf.\ \cite[Lemma~5]{DV75}.

A state on a topological $\star$-algebra $\CA$ is a continuous linear functional $\omega \colon \CA \to \C$ that satisfies for all $a \in \CA$
\begin{equ}
	\omega(a^\dagger a ) \geqslant 0  \quad \text{and} \quad \omega(\bone) = 1 \; .
\end{equ}


\begin{example}
	Let $(S, \mu)$ be a probability space. 
	The space $L^\infty(S, \mu)$ of essentially bounded complex-valued measurable functions is a $C^*$-algebra with the $\star$-operation being complex conjugation. 
	A state on this algebra is given by the integration functional, i.e.\
	\begin{equ}
		\omega_\mu (f) = \int f \d \mu \; .
	\end{equ}

	We can also define a space $L^\infty_{\loc}(S,\mu)$ of locally bounded measurable complex-valued functions, mirroring our main construction of the extended CAR algebra in Section~\ref{sec:extended_car}. Choose a family $\CU$ of measurable sets, s.t.\ $\mu\big( \bigcup \CU \big) = 1$ and $\mu(U)>0$ for all $U \in \CU$. For a measurable function $f$ and $U \in \CU$, let 
	\begin{equ}
		\| f \|_{U} \eqdef \| f \bone_U \|_{L^\infty}
	\end{equ}
	where $\bone_U$ is the indicator function of $U$. We let $L^\infty_{\loc}(S,\mu)$ be the space of measurable functions $f$, s.t.\ $\|f\|_U < \infty$ for all $U \in \CU$, equipped with the topology induced by the seminorms $\| \bigcdot \|_U$. One can easily check that this is a locally $C^*$-algebra. Furthermore, one can view multiplication by $\bone_U$ as an algebra representation $L^\infty_{\loc}(S,\mu) \to L^\infty(U, \mu\big|_U  )$.

	In the non-commutative probability setting we will not be able to view our random variables as functions on a space $S$, but we will have a canonical analogue of $L^\infty(S, \mu)$, that is an algebra of bounded random variables. 
	A key step will be finding a space to play the role of  $L^\infty_\loc(S,\mu)$ so that we can work ``locally'' with respect to unbounded random variables. 
\end{example}

\subsection{Non-Commutative \TitleEquation{\CL^2}{L2}-Space of \TitleEquation{C^*}{C*}-Algebra}

\label{sec:GNS}

For this section we fix an arbitrary (unital) $C^{*}$-algebra $\CA$ with involution denoted $\CA \ni A \mapsto A^{\dagger} \in \CA$. 

Let $\omega \colon \CA \to \C$ be a state on $\CA$,  we can the define an associated positive semi-definite inner product on $\CA$
\begin{equs}[e:scalarProd]
	\scal{\bigcdot, \bigcdot} \colon 	\CA \times \CA & \longrightarrow \C\\
		(A,B) & \longmapsto \omega(A^\dagger B)\;.
\end{equs}
By quotienting out the left ideal
\begin{equ}[e:def_of_null_ideal]
	\CN_\omega \eqdef \big\{ A \in \CA \, \big| \, \omega(A^\dagger A) = 0 \big\}
\end{equ}
and completing the resulting space under the norm determined by \eqref{e:scalarProd},
one obtains the Hilbert space
\begin{equ}
	\mathcal{L}^2(\CA, \omega) \eqdef \overline{\CA/\CN_\omega}\;.
\end{equ}
Writing $\CA \ni B \mapsto [B] \in \overline{\CA/\CN_\omega}$ for the quotient map, we note that the algebra $\CA$ naturally acts on this Hilbert space via $A[B] \eqdef [AB]$.

\subsection{CAR Algebras and Antisymmetric Fock Space}\label{sec:CAR}

A CAR algebra is a $C^*$-algebra that encodes a set of relations called the \textit{canonical anticommutation relations} (CAR).
We present a short overview here, but for an introduction see Chapter 5 of \cite{BR87} and for an in-depth analysis see~\cite{Der06}.

Let $\mfH$\label{p:mfH} be a separable, complex Hilbert space with antiunitary involution $\kappa$, i.e.\ $\kappa$ is antilinear and satisfies for all $f,g \in \mfH$
\begin{equ}
	\scal{\kappa f, \kappa g} = \scal{g,f} \; .
\end{equ}
We will always use the convention compatible with the physicists' bra-ket convention
that scalar products $\scal{g,f}$ are linear in $f$ and antilinear in $g$.
	We first define a complex (non-commutative) $\star$-algebra $\mathring{\mathcal{A}}\F(\mfH)$ generated by symbols $\alpha(f), \alpha^\dagger(f)$ for every $f \in \mfH$, such that the assignment $f \mapsto \alpha^\dagger(f)$ is linear, $f \mapsto \alpha(f)$ is antilinear, $\alpha(f)$ is the $\star$-adjoint of $\alpha^\dagger(f)$, and these elements satisfy the anticommutation relations
\begin{equs}[eq:CAR]{}
	[\alpha(f),\alpha^\dagger(g)]_+ &= \alpha(f) \alpha^\dagger(g) + \alpha^\dagger(g) \alpha(f) = \scal{f,g}_{\mfH} \bone\;,\\{}
[\alpha(f),\alpha(g)]_+ &= [\alpha^\dagger(f), \alpha^\dagger(g)]_+  = 0\;,
\end{equs}
where $\bone$ denotes the unit of $\mathring{\mathcal{A}}\F(\mfH)$.

Note that $\mathring{\mathcal{A}}\F(\mfH)$ is a special type of non-commutative algebra called a \textit{super-algebra} or a $\Z_{2}$-graded algebra.
That is, we have a decomposition $\mathring{\mathcal{A}}\F(\mfH) = \mathring{\mathcal{A}}^{0}\F(\mfH) \oplus \mathring{\mathcal{A}}^1\F(\mfH)$ into an \textit{even} subalgebra $\mathring{\mathcal{A}}^{0}\F(\mfH)$ and \textit{odd} subspace $\mathring{\mathcal{A}}^1\F(\mfH)$ 
\dash this is done by postulating that every instance of $\alpha(\bigcdot)$ or $\alpha^{\dagger}(\bigcdot)$ is odd and by using the natural product rule.

The CAR algebra corresponding to $\mfH$, denoted here by $\CA\F(\mfH)$\label{p:CAR}, is the unique $C^*$-algebra generated by $\mathring{\mathcal{A}}\F(\mfH)$.
Here, uniqueness comes from the fact that \eqref{eq:CAR} fixes the $C^*$-algebra norm.
The canonical anticommutation relations imply that
	\begin{equ}
		\big( \alpha^\dagger(f) \alpha(f) \big)^2 = \alpha^\dagger(f) \big[ \alpha^\dagger(f), \alpha(f) \big]_+ \alpha(f) =  \|f\|^2_{\mfH} \alpha^\dagger(f) \alpha(f)
	\end{equ}
	and thus, because $\alpha^\dagger(f) \alpha(f)$ is self-adjoint, the $C^*$-norm  $\| \bigcdot \|$ on $\mathring{\mathcal A}\F(\mfH)$ must satisfy (if it exists)
\begin{equs}
\|\alpha(f)\|^4 &= \|\alpha^\dagger(f) \alpha(f) \|^2 = \|(\alpha^\dagger(f) \alpha(f))^2 \| = \|f\|^2_{\mfH} \|\alpha^\dagger(f) \alpha(f) \| \\
{}&= \|f\|^2_{\mfH} \|\alpha(f)\|^2\;.
\end{equs}
Combining this with the fact that \eqref{eq:CAR} forces $f \not = 0 \Rightarrow \alpha(f) \not = 0$, we can conclude that $\|\alpha(f)\| = \|\alpha^\dagger(f)\| = \|f\|_{\mfH}$.
We again point the reader to~\cite[Chapter~5.2]{BR13} for a complete proof that the relations above do indeed uniquely determine a $C^*$-algebra (up to canonical isomorphism).
Note that $\CA\F(\mfH)$ inherits the super-algebra structure from $\mathring{\mathcal{A}}\F(\mfH)$.
We sometimes suppress $\mfH$ from the notation.

One way to realise $\CA\F(\mfH)$ is as a subalgebra of the bounded linear operators on $\CF_a (\mfH)$, the fermionic Fock space generated by $\mfH$, i.e.\
\begin{equ}[e:Fock]
	\CF_a(\mfH) \eqdef \bigoplus_{n = 0}^\infty \mfH^{\wedge n} \; , 
\end{equ}
where $\mfH^{\wedge n}$ denotes the $n$-fold antisymmetric (Hilbert space) tensor product of $\mfH$ with itself, and $\mfH^{\wedge 0} \eqdef \C$. 
In particular, $\mfH^{\wedge n}$ can be viewed as the closed subspace of the Hilbert space $\mfH^{\otimes n}$ left invariant by the antisymmetric action of the permutation group. Explicitly, given $f_1,\dots, f_n \in \mfH$ we define the element of $\mfH^{\wedge n}$ given by
\begin{equ}\label{eq:Def_Asym_Prod}
	f_1 \wedge \cdots \wedge f_n \eqdef  S_n^- \left( f_1 \otimes \cdots \otimes f_n\right) \eqdef \frac{1}{n!} \sum_{\sigma \in \mfS_n} \sgn(\sigma) f_{\sigma(1)} \otimes \cdots \otimes f_{\sigma(n)}
\end{equ}
where $S_n^- \colon \mfH^{\otimes n} \to \mfH^{\wedge n} \subset \mfH^{\otimes n}$ extends by linearity to a projection operator. For $F,G \in \mfH^{\wedge n}$ we define
\begin{equ}
	\scal{F,G}_{\mfH^{\wedge n}} \eqdef \scal{F,G}_{\mfH^{\otimes n}}
\end{equ}
For product vectors $f_1 \wedge \cdots \wedge f_n, g_1 \wedge \cdots \wedge g_n$ in $\mfH^{\wedge n}$, where all $f_i, g_i \in \mfH$, one furthermore has the expression
\begin{equs}
	\scal{f_1 \wedge \cdots \wedge f_n, g_1 \wedge \cdots \wedge g_n}_{\mfH^{\wedge n}} &= \frac{1}{n!} \sum_{\sigma \in \mfS_n} \sgn(\sigma) \prod_{i = 1}^n \scal{f_i, g_{\sigma(i)}}_{\mfH} =\\
		&= \frac{1}{n!}  \det \left( \left( \scal{f_i,g_j}_{\mfH} \right)_{1 \leqslant i,j \leqslant n} \right) \; .
\end{equs}

We call $\1 \in \mfH^{\wedge 0} \subset \CF_a(\mfH)$ the \textit{vacuum}, where $\1$ denotes the empty exterior product.

We then set, for any $f \in \mfH$, $\alpha^{\dagger}(f) \1 = f$ and $\alpha(f) \1 = 0$ along with
\begin{equs}[e:defaa*]
\alpha^{\dagger}(f) f_{1} \wedge \cdots \wedge f_{n} &\eqdef \sqrt{n+1} f \wedge f_{1} \wedge \cdots \wedge f_{n}\;,\\
\alpha(f) f_{1} \wedge \cdots \wedge f_{n} & \eqdef \sqrt{n}  \sum_{j=1}^{n} (-1)^{j-1} \scal{f,f_j}_{\mfH} \bigwedge_{\substack{1 \leqslant i \leqslant n \\ i \not = j}} f_{i}\;.
\end{equs}
One can easily check that $\alpha(f)$ and $\alpha^\dagger(f)$ are in fact adjoints as operators on $\CF_a(\mfH)$ and that they indeed satisfy \eqref{eq:CAR}.
We remark that the above construction also gives a corresponding vacuum \textit{state}
\begin{equs}[e:vacuum_state]
\omega\F \colon \CA\F(\mfH) &\longrightarrow \C\\
		A & \longmapsto  \scal{\1, A \1}_{\CF_{a}(\mfH)}\;.
\end{equs}

In the context of quantum field theory, $\mfH$ is called the single particle Hilbert space and $\alpha^\dagger(f)$ and $\alpha(f)$ are  called \textit{creation} and \textit{annihilation} operators on $\CF_a(\mfH)$ respectively.

Finally, an unbounded operator on $\CF_{a}(\mfH)$ that will be of importance to us will be the number operator $N$ given by
\begin{equ}\label{eq:num_operator}
N = \sum_{j}  \alpha^{\dagger}(e_{j}) \alpha(e_{j})\;,
\end{equ}
where $(e_{j})_{j=1}^{\infty}$ can be taken to be any orthonormal basis of $\mfH$. 
One can take the algebraic direct sum $\oplus_{n=0}^{\infty} \mfH^{\wedge n}$ as a core for the number operator and we note that on $\mfH^{\wedge n}$, the operator $N$ is just given by multiplication by $n$. In particular, the definition of $N$ does not depend on the choice of orthonormal basis.

\subsection{Euclidean Fermionic Field Operators}\label{subsec:fields}
In quantum field theories describing bosons (or Majorana fermions), the quantum field $\Phi$ is a self-adjoint operator-valued distribution given by an expression of the type $\Phi = \alpha + \alpha^{\dagger}$ on the bosonic\slash fermionic Fock space. However, for Dirac fields, which include pairs of particles and antiparticles, the field operators are generally not self-adjoint and mix the creation and annihilation operators of the particles and antiparticles. In particular, this means that we do not have functional calculus or something similar at our disposal, leading to our definition of non-commutative localisation below in Section~\ref{sec:extended_car}.

Furthermore, our algebra and state need to be defined so that Euclidean fermionic fields anticommute inside of expectations under this state. 

Following \cite{OS73} and \cite{Gub20}, we can achieve this by introducing a unitary operator $U$ on $\mfH$ that is $\kappa$-antisymmetric, i.e.\ $U^T \eqdef \kappa U^\dagger \kappa = -  U $ and then defining the fermionic noise to be the (complex-linear) operator $\Psi \colon \mfH \to \CA\F(\mfH)$
\begin{equ}\label{eq:fermionic_field}
	\Psi(f) \eqdef  \alpha(\kappa U f) + \alpha^\dagger(f) \; .
\end{equ}

\begin{remark}\label{rem:covariance}
	We note that the choice of $\scal{ \bigcdot, \bigcdot}_{\fH}$ and $U$ \textit{together} determine the ``covariance'' of our fermionic noise.
	A quick calculation shows that, for $f,g \in \mfH$,
	\begin{equ}
		\omega\F(\Psi(f)\Psi(g)) = \scal{\1, \Psi(f)\Psi(g) \1}_{\CF_{a}(\mfH)} = \scal{\kappa U f, g}_{\mfH} = - \scal{\kappa f, U g}_{\mfH} \;  .
	\end{equ}
In particular, we have the desired anticommutation in expectation. 

Also note that there are many choices of $\scal{ \bigcdot, \bigcdot}_{\mfH}$ and skew-symmetric $U$ that lead to fermionic fields $\Psi$ with the same covariance.
\end{remark}

\begin{remark}
	If one has a polynomial $F(\Psi(f_1), \dots, \Psi(f_n))$ for $f_i \in \mfH$, then \begin{equ}
		\omega\F\bigl(F(\Psi(f_1), \dots, \Psi(f_n))\bigr)
	\end{equ}
	is essentially a rigorous expression for the formal infinite-dimensional Berezin integral against a fermionic ``measure'' of the form
	\begin{equ}
		e^{\scal{\Psi,U\Psi}} \CD\Psi \; .
	\end{equ}
	In particular one has the well-known\footnote{The identity can be proven by using a straightforward induction argument using \eqref{eq:CAR}, see for instance \cite{Gub20}.} fermionic Wick rule,
	\begin{equ}
		\omega\F\big(\Psi(f_1) \cdots \Psi(f_{2n})\big) = \sum_{\sigma \in \mfS_{2n}} \sgn(\sigma) \prod_{i = 1}^n \omega\F\big(\Psi(f_{\sigma(2i-1)})\Psi(f_{\sigma(2i)})\big) \; .
	\end{equ}
\end{remark}

One can easily check that $[\Psi(f),\Psi(g)]_+ = 0$ holds for all $f,g \in \mfH$ from which we can conclude $\left\{ \Psi(f) \, \big| \, f \in \mfH\right\}$ generates a super-subalgebra $ \mathring{\mfG}\F(\mfH) \subset \CA\F(\mfH)$ \dash in particular it is super-commutative, i.e.\ for homogeneous (purely even or odd) elements $x,y$ of this subalgebra we have $x y = (-1)^{|x| \cdot |y|} y x$ where $|x|, |y| \in \Z_2$ are the degrees of the elements.
However,  $\mathring{\mfG}\F(\mfH)$ is not closed under adjoints so it is not a sub-$\star$-algebra of $\CA\F(\mfH)$. 
We denote by $\mfG\F(\mfH)$\label{p:fields} the closure of $\mathring{\mfG}\F(\mfH)$ with respect to the norm of $\CA\F(\mfH)$.

\begin{remark}
All the physical content of fermionic field theories is contained in observables built from $\Psi$, and much of the literature concerning Euclidean fermions pays little attention to the $C^{*}$-algebra $\CA\F$ and instead focuses exclusively on $\mathring{\mfG}\F$ (often called the Grassmann algebra) or some metric closure of  $\mathring{\mfG}\F$.

In fact, since it does not have a $\star$-structure, there exist many different reasonable choices of norms on $\mathring{\mfG}\F$ under which one may get different
completions; see for example \cite{FKT02} for a different approach.
Even though a $C^{*}$-algebra introduces many non-physical observables, we have the following motivations for focusing on them in our framework:

\begin{itemize}
				

\item As a subalgebra of $\CA\F$, $\mfG\F$ is a more concrete object than the abstractly defined Grassmann algebra usually found in the literature. 
In particular, one has an explicitly defined state $\omega\F$ on $\CA\F$ while defining some analogue of the state on an abstract Grassman algebra can lead to having to define infinite-dimensional Pfaffians of operators. 

\item Our extended CAR algebra framework given in the next section, which is needed for formulating SPDEs including unbounded renormalised physical observables built from $\Psi$, involves building a locally $C^{*}$-algebra\footnote{Locally $C^{*}$-algebras are defined in Section~\ref{section:AlgebraDefn}.} which enlarges the $C^{*}$-algebra $\CA\F$. 

\item It is unclear how to define renormalised fields in just the Grassmann setting as the norm of the counterterm is independent of the norm of the product because they live in different degrees of the $\N$-gradation of the Grassmann algebra.  However, such fields do make sense as unbounded operators affiliated with von Neumann closures of $\CA\F$, see Remark~\ref{rem:UnbndOp}. It would be interesting in the future to make a strong connection between the extended CAR algebra and such unbounded operators.

\item This approach is closer to that of \cite{OS73} and thus it is a bit clearer how one 
would have to proceed to implement the OS axioms.

\end{itemize}
\end{remark}

\subsection{Araki-Wyss Representations}
\label{sec:ArakiWyss}

This section will explore a slightly more intricately defined quasi-free state $\omega_\rho$ on $\CA(\mfH)$ that is, however, faithful, i.e.\ for all $A \in \CA(\mfH)$ 
\begin{equ}
	\omega_\rho(A^\dagger A) > 0 \; ,
\end{equ}
and thus the corresponding non-commutative $\CL^2$-space is a completion of $\CA(\mfH)$\footnote{After submitting this article, we noticed the possibility of using this state within our framework after seeing the recent article \cite{Gub23}}. Recall that the this is a Hilbert space with inner product $(A,B) \mapsto \omega(A^\dagger B)$. In particular, with this state one has $\CN_{\omega_\rho} = \{0\}$. 

However, we will not use this in the rest of the article except for isolated remarks, \ref{rem:ItoUnitary}, \ref{rem:localL2} and \ref{rem:topology_renormalisation}. 

Given a $C^*$-algebra there are typically (infinitely) many different ways to complete it to a von Neumann algebra. This is readily seen when one considers the different spaces $L^\infty(\Omega;\mu)$ for some probability measures $\mu$ that are all von Neumann completions of $\cC(\Omega)$, cf.\ \cite{Tak02}. The fermionic CAR algebra $\mfA(\mfH)$ constructed in Section~\ref{sec:CAR} is an example of a $C^*$-algebra that admits many such completions. In particular, we shall describe the so-called Araki-Wyss representations, a class of quasi-free representations, i.e.\ representations defined via a state that satisfies Wick's rule. 

\begin{definition}	
	Let $\rho \in \mfL(\mfH)$ be a positive self-adjoint operator satisfying $0 \leqslant \rho \leqslant \bone$. Let $\alpha_2, \alpha_2^\dagger$ denote the annihilation and creation operators of $\mfA(\mfH \oplus \mfH)$. The (left) Araki-Wyss representation of $\mfA(\mfH)$ is the homomorphism $\pi_\rho \colon \mfA(\mfH) \to \mfA(\mfH \oplus \mfH)$  defined via 
	\begin{equs}
		\alpha^\dagger(f) & \longmapsto \alpha_2^\dagger\left( \sqrt{\bone-\rho} f , 0 \right) + \alpha_2\left( 0 , \kappa \sqrt{\rho} f \right)\\
		\alpha(f) & \longmapsto \alpha_2^\dagger \left( 0 , \kappa \sqrt{\rho} f \right) + \alpha_2\left( \sqrt{\bone-\rho} f , 0 \right) 
	\end{equs}

	We shall denote by $\omega_2$ the Fock vacuum state on $\mfA(\mfH \oplus \mfH)$.
\end{definition}

\begin{proposition}
	$\pi_\rho$ is a $C^*$-algebra homomorphism.
\end{proposition}
\begin{proof}
	We only check that $\pi_\rho$ preserves the anticommutation relations, for more details see \cite{Der06}.
	\begin{equs}
		\left[ \pi_\rho(\alpha^\dagger(f)) , \pi_\rho(\alpha(g)) \right]_+ & = \left[ \alpha_2^\dagger( \sqrt{\bone-\rho} f,0) , \alpha_2( \sqrt{\bone-\rho} g,0) \right]_+ + \\
		& \qquad  + \left[ \alpha_2(0, \kappa \sqrt{\rho} f) , \alpha_2^\dagger(0, \kappa \sqrt{\rho} g) \right]_+ = \\
		&= \scal{\sqrt{\bone-\rho} g, \sqrt{\bone-\rho} f}_{\mfH} \bone + \scal{\kappa\sqrt{\rho} f, \kappa\sqrt{\rho} g}_{\mfH} \bone = \scal{g,f}_{\mfH} \bone \\
		\left[ \pi_\rho(\alpha^\dagger(f)) , \pi_\rho(\alpha^\dagger(g)) \right]_+ &= \left[ \alpha_2^\dagger( \sqrt{\bone-\rho} f,0) , \alpha^\dagger_2( \sqrt{\bone-\rho} g,0) \right]_+ +\\
		& \qquad  +  \left[ \alpha_2(0, \kappa \sqrt{\rho} f) , \alpha_2(0, \kappa \sqrt{\rho} g) \right]_+ = 0
	\end{equs}
	Here we have only included terms with both input vectors belonging either to the first or second factor of $\mfH \oplus \mfH$ as the other ones are necessarily $0$ by orthogonality.

\end{proof}

\begin{proposition}
	The state $\omega_\rho \eqdef \omega_2 \circ \pi_\rho$ is faithful if and only if the kernel of $\rho$ and $1-\rho$ are trivial.
\end{proposition}
\begin{proof}
	See \cite[Theorem~43]{Der06}.
\end{proof}
Note that the state depends on the choice of representation $\pi_\rho$ and thus the operator $\rho$. From now on we shall assume that $\rho$ induces a faithful state. In particular this means that $\CL^2(\CA(\mfH),\omega_\rho) \supset \CA(\mfH)$.

\begin{remark}
	This contrasts with the vacuum state where many elements of $\CA(\mfH)$ are identified with one another in $\CL^2(\CA,\omega)$. For example, $\|\alpha(f)\|_{\CL^2(\omega)} = 0$ but $\|\alpha(f)\|_{\CL^2(\omega_\rho)} = \|\sqrt{\rho} f \|_{\mfH}$.
\end{remark}

The image of the fermionic field $\Psi(f)$ under the representation $\pi_\rho$ is given 
\begin{equ}
	\pi_\rho(\Psi(f)) = \alpha_2^\dagger\left( \sqrt{\bone-\rho} f, \kappa \sqrt{\rho} \kappa U f \right) + \alpha_2 \left( \sqrt{\bone-\rho} \kappa U f, \kappa \sqrt{\rho} f \right)
\end{equ}

Now suppose that $\kappa U$ commutes with the one-particle density $\rho$. 
Then
\begin{equ}
	\omega_\rho \left( \Psi(f) \Psi(g) \right) = \scal{\kappa U f, (\bone-2\rho)g} \; .
\end{equ}
If $\bone-2\rho$ is invertible, e.g.\ if $\rho < \frac{1}{2}\bone$, then we can modify $U$ in such a way that the resulting fermionic fields $\Psi_\rho(f)$ are still Gaussian and have covariance $U$ w.r.t.\ $\omega_\rho$. In particular,
\begin{equ}
	\Psi_{\rho}(f) \eqdef \alpha^\dagger(f) + \alpha(\kappa \widetilde{U} f)
\end{equ}
Where
\begin{equ}
	\widetilde{U}\eqdef U(\bone-2\rho)^{-1} \; .
\end{equ}
One easily checks, that $\widetilde{U}^T = - \widetilde{U}$. One easy way to ensure that $\rho$ commutes with $\kappa U$ is to set $\rho = \lambda \bone$ with $\lambda \in (0,\frac{1}{2})$.

\subsection{Wick Products and Fock Space}\label{sec:fock_renorm}
In much of the singular SPDE literature, given a multilinear functional of singular Gaussian (bosonic) noises, in order to prescribe the renormalisation of this functional and prove its probabilistic convergence (for instance, in some $L^{q}\big(\Omega,\mu ; \CC_{\s}^{\alpha}(\spacetime) \big) $ topology), the standard way to proceed is to rewrite objects as chaos expansions \slash sums of Hermite polynomials and then remove divergent terms in these sums.

In comparison, our renormalisation prescription for multilinear functionals of the fermionic noise will involve super-commutative analogues of these Hermite polynomials of the fermionic noise. 

We start by introducing our joint field operator\slash noise process.
For this section we again view $\mfH$ as generic complex Hilbert spaces and introduce a second bosonic Hilbert space $\mfB$\label{p:boson}.
As before, given a $U$ on $\mfH$, we then have a fermionic noise which is a continuous map $\bPsi \colon \mfH \to \mfA_\infty$.

Fixing a Gaussian measure $\mu$ on some real topological vector space $S$ with reproducing kernel Hilbert space\footnote{That is the Hilbert space generated by the scalar product $\E_\mu(\kappa(f^*(\xi))g^*(\xi))$ in the complexification of $S'$, the dual space of $S$, after quotienting out vectors of norm $0$ and taking its completion.} $\mathfrak{B}$, we have a corresponding bosonic noise field (process)\footnote{Note that there is an asymmetry in how we define $\Psi$ and $\xi$ here since $\xi$ does not take values in some algebra of observables $\cA_{B}$ (see Remark~\ref{rem:BosFermProb}) that is defined independently of the state \slash measure, so $\xi(\bigcdot)$ is more analogous to $[\Psi(\bigcdot)]$ where $[\bigcdot]$ is defined as in Section~\ref{sec:GNS}.} $\xi$ which is a continuous map $\xi\colon \mfB \rightarrow  L^{\infty - }(S,\mu)$. 

\subsubsection{It\^{o}--Segal--Wiener  Isomorphisms}

We will review standard facts about the relationship between Gaussian measures\slash states and bosonic and fermionic Fock spaces \dash for more details we refer the reader to \cite{Nualart} for the bosonic case and \cite{MeyerQuantum} discussing both fermions and bosons.

For any $n \in \N$, we denote by $\mfB^{\otimes_s n}$\label{p:symPow} the $n$-fold symmetric tensor power of $\mfB$ and $\mfB^{\otimes_s 0} \eqdef \C$. 
Here we write, in analogy to \eqref{eq:Def_Asym_Prod}, for $f_1,\dots, f_n \in \mfB$
\begin{equ}
	f_1 \otimes_s \cdots \otimes_s f_n \eqdef \frac{1}{n!} \sum_{\sigma \in \mfS_n} f_{\sigma(1)} \otimes \cdots \otimes f_{\sigma(n)}
\end{equ}
and for $g_1 \otimes_s \cdots \otimes_s g_n \in \mfB^{\otimes_s n}$ we have the scalar product
\begin{equ}
	\scal{f_1 \otimes_s \cdots \otimes_s f_n, g_1 \otimes_s \cdots \otimes_s g_n}_{\mfB^{\otimes_s n}} =\frac{1}{n!} \sum_{\sigma\in \mfS_n} \prod_{i = 1}^n \scal{f_i, g_{\sigma(i)}}.
\end{equ}
As before, $\mfB$ is called the single particle Hilbert space, and $\CF_{s}(\mfB)$\label{p:FockB} is the corresponding bosonic or symmetric Fock space.

The It\^{o}--Segal--Wiener isomorphism
\begin{equ}[eq:bosonic_iso]
	\iota_{s}\colon  \mathcal F_s(\mfB) \eqdef \bigoplus_{n = 0}^\infty \mfB^{\otimes_s n} \rightarrow  L^2(S, \mu)\;,
\end{equ}
is then given by 
\begin{equ}
	\iota_{s} \eqdef \bigoplus_{n=0}^{\infty} \frac{1}{\sqrt{n!}} \xi^{\diamond n}
\end{equ}
where, for $n \in \N$, the maps $\xi^{\diamond n} \colon \mfB^{\otimes_s n} \rightarrow L^2(S, \mu)$, the Wick\footnote{Another common notation for $\xi^{\diamond n}(f_1, \dots, f_n)$ is $\colon\xi(f_1) \cdots \xi(f_n)\colon$.} or normal ordered products, are defined by the inductive relation 
\begin{equs}
	\xi^{\diamond n}\big(f_{1} \otimes_{s}& \cdots \otimes_{s} f_{n}\big)
	\eqdef \xi(f_1) \xi^{\diamond (n-1)} \big(f_{2} \otimes_{s} \cdots \otimes_{s} f_{n}\big) - \label{eq:bosonic_wick}\\
	&- \sum_{i = 2}^n \E_{\mu}[ \xi(f_1) \xi( f_i)]\ \xi^{\diamond (n-2)} \big(f_{2} \otimes_{s} \cdots \otimes_{s} \widehat{ f_{i} } \otimes_{s} \cdots \otimes_{s} f_{n} \big)\;,
\end{equs}
where we write $\widehat{f_i}$ above to denote the omission of the factor $f_i$ from the tensor product. 
This induction
``starts'' with $\xi^{\diamond 1} \eqdef \xi$ and $\xi^{\diamond 0}(c) \eqdef c \1$ where $\1$ is the function that is identically $1$ on $\Omega$.
It is not hard to check that $\xi^{\diamond n}(\bigcdot)$ satisfies the appropriate symmetry properties so that it is well-defined on the algebraic $n$-fold symmetric tensor product and that moreover it preserves the $\mfB^{\otimes_{s} n}$ norm so that it indeed extends from the algebraic $n$-fold symmetric tensor product to all of $\mfB^{\otimes_{s} n}$.


We now turn to discussing the fermionic analogue $\iota_{a}$ of \eqref{eq:bosonic_iso}. 
Recall Section~\ref{sec:CAR} and our fermionic noise $\Psi$ as defined in \eqref{eq:fermionic_field}. As we shall see shortly, there is an isomorphism
\begin{equ}\label{eq:fermionic_isomorphism}
\iota_a \colon\CF_{a}(\mfH) \rightarrow  \CL^2(\CA\F(\mfH), \omega\F)\;,
\end{equ}
where the space $\CL^2(\CA\F(\mfH), \omega\F)$ is defined as in Section~\ref{sec:GNS}. 


We first inductively define for $n \in \N$ the maps
\begin{equ}[e:defPsiBasic]
	\Psi^{\diamond n}\colon \Lambda_{n}(\mfH) \rightarrow \CA\F(\mfH)
\end{equ}
where $\Lambda_{n}(\mfH) \subset \mfH^{\wedge n}$ is the algebraic $n$-fold antisymmetric tensor product. (This inclusion is strict unless $\dim \mfH < \infty$ since the 
Hilbert space tensor product appearing in $\mfH^{\wedge n}$ is strictly larger than the algebraic
tensor product.)
We set
\begin{equs}[eq:fermionic_wick]
	\Psi^{\diamond n}\big(f_{1} \wedge & \cdots \wedge f_{n}\big)
	\eqdef \Psi(f_1) \Psi^{\diamond (n-1)} \big(f_{2} \wedge \cdots \wedge f_{n}\big)- \\
	&- \sum_{i = 2}^n (-1)^{i} \omega\F \big( \Psi(f_1) \Psi(f_i) \big) \Psi^{\diamond (n-2)} \big(f_{2} \wedge \cdots \wedge \widehat{ f_{i} } \wedge \cdots \wedge f_{n} \big)\;,
\end{equs}
where we use the same conventions as in \eqref{eq:bosonic_wick} namely $\Psi^{\diamond 0}(c) = c \bone$ and $\Psi^{\diamond 1} = \Psi$.
It is not hard to check that $\Psi^{\diamond n}(\bigcdot)$ satisfies the appropriate antisymmetry properties so it is indeed well-defined on $\Lambda_{n}(\mfH)$, but for $n \geqslant 2$ one cannot expect to extend $\Psi^{\diamond n}$ to a map from $\mfH^{\wedge n}$ to $\CA\F(\mfH)$ \dash see Theorem~\ref{thm:otte} below. 
However, writing $[\bigcdot] \colon \CA\F(\mfH) \rightarrow \CL^{2}(\CA\F(\mfH),\omega\F)$ for the map taking algebra elements to their equivalence classes, one can\footnote{This is an easy consequence of calculating the $\CL^{2}$ norm of the right-hand side of \eqref{eq:L2_wick}.}
 extend $[\Psi^{\diamond n}(\bigcdot)]$ to an isometry from $\mfH^{\wedge n}$ to $\CL^{2}(\CA\F(\mfH),\omega\F)$.
We can then set
\begin{equ}
	\iota_{a}(\bigcdot) \eqdef \bigoplus_{n=0}^{\infty} \frac{1}{\sqrt{n!}} [\Psi^{\diamond n}(\bigcdot)]\;.
\end{equ}

\begin{remark}
\label{rem:ItoUnitary}
	In the context of the faithful state $\omega_\rho$, there is no need to take equivalence classes when defining $\CL^2(\CA(\mfH),\omega_\rho)$. In that case $\Psi^{\diamond n}(F)$ is directly an element of $\CL^2(\CA(\mfH),\omega_\rho)$ and thus also an unbounded operator affiliated (with the von Neumann closure) of $\CA(\mfH)$. Furthermore, $\iota_a$ is again a unitary map in this case. 
\end{remark}

\subsubsection{Wick Powers as Unbounded Operators}
We remark that, for $f_{1} \wedge \cdots \wedge f_{n} \in \mfH^{\wedge n}$ 
\begin{equ}\label{eq:L2_wick}
[\Psi^{\diamond n}(f_1 \wedge \cdots \wedge f_n)] = [ \alpha^{\dagger}(f_1) \cdots \alpha^{\dagger}(f_n)]\;,
\end{equ} 
as elements of $\mathcal{L}^{2}(\CA\F(\mfH),\omega\F)$. 
To see this, we first note that \eqref{eq:L2_wick}
is trivial for $n \in \{0,1\}$, so it suffices to show that it is stable under the induction
\eqref{eq:fermionic_wick}. 
Assuming \eqref{eq:L2_wick} holds up to $n-1$, it follows from the fact
that  $\CN_{\omega\F}  = \left\{ A \in \CA\F(\mfH) \, \big| \, A \Omega = 0\right\}$ 
(cf.\ \eqref{e:def_of_null_ideal} and \eqref{e:vacuum_state}) is a left ideal that 
\begin{equs}{}
[\Psi^{\diamond n}\big(f_{1} \wedge & \cdots \wedge f_{n}\big)]
= \big[\Psi(f_1) \alpha^\dagger(f_{2}) \cdots \alpha^\dagger(f_{n})\big] - \\
	&- \sum_{i = 2}^n (-1)^{i} \omega\F \big( \Psi(f_1) \Psi(f_i) \big) \big[\alpha^\dagger(f_{2})\cdots  \widehat{ \alpha^\dagger(f_{i})} \cdots \alpha^\dagger(f_{n})\big]\;.
\end{equs}
As a consequence of \eqref{eq:fermionic_field}, we then see that 
\begin{equs}
\big[\Psi(f_1) &\alpha^\dagger(f_{2}) \cdots \alpha^\dagger(f_{n})\big]
- \big[\alpha^\dagger(f_{1}) \cdots \alpha^\dagger(f_{n})\big]
= \big[\alpha(\kappa U f_1) \alpha^\dagger(f_{2}) \cdots \alpha^\dagger(f_{n})\big]=\\
&= \scal{\kappa U f_1,f_2}\big[\alpha^\dagger(f_{3}) \cdots \alpha^\dagger(f_{n})\big] -\big[\alpha^\dagger(f_{2})\alpha(\kappa U f_1)\alpha^\dagger(f_{3}) \cdots \alpha^\dagger(f_{n})\big]=\\
&= \scal{\kappa U f_1,f_2}\big[\alpha^\dagger(f_{3}) \cdots \alpha^\dagger(f_{n})\big]
- \scal{\kappa U f_1,f_3}\big[\alpha^\dagger(f_{2})\alpha^\dagger(f_{4}) \cdots \alpha^\dagger(f_{n})\big] +\\ &\qquad +\big[\alpha^\dagger(f_{2})\alpha^\dagger(f_{3})\alpha(\kappa Uf_1)\alpha^\dagger(f_{4}) \cdots \alpha^\dagger(f_{n})\big] =\\
&= \sum_{i=2}^n (-1)^i \scal{\kappa U f_1,f_i}\big[\alpha^\dagger(f_{2})\cdots  \widehat{ \alpha^\dagger(f_{i})} \cdots \alpha^\dagger(f_{n})\big] -\\
&\qquad - (-1)^n\big[\alpha^\dagger(f_{2}) \cdots \alpha^\dagger(f_{n})\alpha(\kappa U f_1)\big] =\\
&= \sum_{i=2}^n (-1)^i \omega\F \big( \Psi(f_1) \Psi(f_i) \big)\big[\alpha^\dagger(f_{2})\cdots  \widehat{ \alpha^\dagger(f_{i})} \cdots \alpha^\dagger(f_{n})\big]\;,
\end{equs}
and the claim follows.
It is important to note that we do \textit{not} have $\Psi^{\diamond n}(f_1 \wedge \cdots \wedge f_n) = \alpha^{\dagger}(f_1) \cdots \alpha^{\dagger}(f_n)$ in $\CA\F(\mfH)$. 
However, they are equivalent in  $\mathcal{L}^{2}(\CA\F(\mfH),\omega\F)$, see Section~\ref{sec:extended_car}. 

As is the case for $\Psi^{\diamond n}$, one does not expect the map $\Lambda_{n}(\mfH) \ni f_{1} \wedge \cdots \wedge f_{n} \mapsto \alpha^{\dagger}(f_1) \cdots \alpha^{\dagger}(f_n) \in \CA\F(\mfH)$ to extend to $\mfH^{\wedge n}$. 
However, we can extend this mapping to $\mfH^{\wedge n}$ if we are willing to map to unbounded operators bounded by the number operator. 
This is the content of the estimate \cite[Proposition~1.2.3]{GJ85} which will be crucial for us, we restate this as follows. 
\begin{theorem}
\label{thm:LocalUltraContr}
	For all $r,s \in \N$ the multilinear map $W_{r,s} \colon \mfH^{r+s} \to \CB\big(\CF_a(\mfH)\big)$
	\begin{equ}
		(f_1,\dots, f_{r+s}) \longmapsto \alpha^\dagger(f_1) \cdots  \alpha(f_r)^{\dagger}  \alpha(\kappa f_{r+1}) \cdots \alpha(\kappa f_{r+s}) (1+N)^{-\frac{r+s-1}{2}}
	\end{equ}
extends to a continuous map $\mfH^{\otimes(r+s)} \to\CB\big(\CF_a(\mfH)\big)$. 
Here, $N$ is the number operator defined in \eqref{eq:num_operator}. 

Furthermore, there exists a constant $C>0$ independent of $r$, $s$ such that
	\begin{equ}
		\|W_{r,s}(G)\| \leqslant C (|r-s|+1)^{\frac{r+s}{2}} \|G\|_{\mfH^{\otimes (r+s)}}
	\end{equ}
\end{theorem}
We prove this in Appendix~\ref{sec:UltrConProof}.

\begin{remark}
	\label{rem:UnbndOp}
	In particular, we note that this means that for all $F\in \mfH^{\wedge n}$ the closure of $\Psi^{\diamond n}(F)$, which exists by Lemma~\ref{lemma:ClosableOp}, can always be interpreted as an unbounded operator affiliated with $\CB(\CF_a(\mfH))$, the von Neumann closure of $\CA(\mfH)$ w.r.t.\ $\omega$.
\end{remark}

Although we will not use these results here, one can show similar lower bounds, cf.\ \cite{Otte09}, which show that renormalised products in this algebra necessarily yield unbounded operator-valued distributions. Explicitly, we have the following. 
\begin{theorem}\label{thm:otte}
	Let $(e_i)_i$ be an orthonormal basis of $\mfH$ and $A \in \mathrm{HS}(\mfH)$ where $\mathrm{HS}(\mfH)$ is the space of Hilbert-Schmidt operators from $\mfH$ to itself. 	
	Suppose that $\sum_i \alpha^\dagger(e_i) \alpha^\dagger(Ae_i)$ is a bounded operator. 

	It then follows that, for all $\eps > 0$,
	\begin{equ}[eq:LowerBndEst]
		\Trace\Big( \big( A^\dagger A \big)^\frac{1+\eps}{2} \Big) < \infty \; .
	\end{equ}
	The same holds for $\sum_i \alpha(e_i) \alpha(Ae_i)$.
\end{theorem}

\begin{remark}
Note that the map 
\begin{equ}\label{eq:operators-to-tensors}
A \mapsto \sum_{i,j=1}^{\infty} e_{i} \otimes A e_{j}\;
\end{equ}
 is a Hilbert space isomorphism between $\mathrm{HS}(\mfH)$ and the Hilbert space tensor product $\mfH \otimes \mfH$. 
The condition \eqref{eq:LowerBndEst} for $\eps = 1$ is equivalent to $A$ being Hilbert-Schmidt, but for $\eps = 0$ it is equivalent to $A$ being trace-class. 
The map \eqref{eq:operators-to-tensors} is also a continuous isomorphism from the Banach space of trace class operators to the much smaller
projective tensor product $\mfH \wotimes_\pi \mfH \subsetneq \mfH \otimes \mfH$. 
Theorem~\ref{thm:otte} then tells us that if $v = \sum_{i,j} f_{i,j} e_{i} \otimes e_{j} \in \mfH \otimes \mfH$ has the property that the sum $\sum_{i,j} f_{i,j} \alpha^{\dagger}(e_i) \alpha^{\dagger}(e_j)$ converges in $\CA\F(\mfH)$ then $v$ almost belongs to $\mfH \wotimes_\pi \mfH$.
\end{remark}

\subsection{Extended CAR Algebra and Local Convergence}\label{sec:extended_car}

One of our aims is to develop notions of ``pointwise''\slash ``local'' convergence for the CAR algebra with the properties that (i) regularised Hermite polynomials of the free Dirac fields converge pointwise as their regularisation is removed\footnote{We want a codomain which is an $m$-convex algebra for which $\Psi^{\diamond n}$ extends to $\mfH^{\wedge n}$.} and (ii) pointwise convergence and sufficiently strong uniform bounds should imply the convergence of expectations of observables. 

In non-commutative algebraic geometry, irreducible representations are sometimes interpreted as non-commutative analogues of points, just as classical (commutative) points can be defined by the irreducible (one-dimensional) representations of the commutative function algebra over a space \cite{KR00,Ros95}.
For statement (i) above, we will see it makes sense for our ``points'' to correspond to \emph{finite}-dimensional representations. 
We would then want to define representations\footnote{Recall that $\CB(V)$ denotes bounded linear operators on $V$.}
 $\Gamma \ni \pi \colon \CA\F  \to \CB \left( \C^{N(\pi)} \right)$ for some $N(\pi) \in \N$,  and the topology of pointwise convergence would then be encoded by the corresponding seminorms $\|a\|_\pi = \|\pi(a)\|$ for $\pi \in \Gamma$. 
 
However, it is well-known that for an infinite-dimensional Hilbert space $\mfH$ the CAR algebra $\CA\F(\mfH)$ does not admit any non-trivial finite-dimensional representations\footnote{Assume that $\pi \colon \CA\F(\mfH) \to \CB(\C^N)$ is a non-trivial representation for some $N \in \N$. Then there must exist $f \in \mfH$ s.t.\ $\alpha^\dagger(f) \in \ker \pi$, contradicting the anticommutation relations. Indeed, for the same reason a finite-dimensional CAR algebra does not admit a non-trivial representation of dimension lower than itself.}, this is a consequence of the anticommutation relations that are enforced in $\CA\F(\mfH)$.
To overcome this we follow the approach of \cite{DV75} and temporarily drop the anticommutation relations giving us an algebra $\mfA\F(\mfH)$ with fewer relations than $\mathring{\mathcal{A}}\F(\mfH)$, allowing it to admit many different finite-dimensional representations $\pi$. 

We then fix a particular collection of finite dimensional representations $\Gamma$ and partially reintroduce the anticommutation relations by quotienting out the ideal $\mfI_\Gamma$ of all elements of  $\mfA\F(\mfH)$ that vanish in every representation in $\Gamma$ (that is, vanish on all points).  

We topologise and complete $\mfA\F(\mfH) / \mfI_\Gamma$  by using $m$-convex $C^*$-seminorms\footnote{$m$-convexity and $C^*$-seminorms are discussed in Section~\ref{section:AlgebraDefn}.}  $(\|a\|_\pi )_{\pi \in \Gamma}$ mentioned earlier (that is, we topologise and complete using seminorms derived from our ``points'' ).

We now make this construction precise. 
Let $\mfH$ again be a separable, complex Hilbert space with antiunitary involution $\kappa$. 
Let $\mfA\F(\mfH)$\label{p:free} be the free $\star$-algebra generated by $\mfH$, that is $\mfA\F(\mfH)$ is generated by symbols $\balpha(f), \balpha^\dagger(f)$ for every $f \in \mfH$, such that the assignment $f \mapsto \balpha^\dagger(f)$ is linear, $f \mapsto \balpha(f)$ is antilinear, $\balpha(f)$ is the $\star$-adjoint of $\balpha^\dagger(f)$.
Indeed, one should think of $\mfA\F(\mfH)$ as an analogue of $\mathring{\mathcal{A}}\F(\mfH)$ where one does not enforce the anticommutation relations. 

In contrast with $\mathring{\mathcal{A}}\F(\mfH)$, the algebra $\mfA\F(\mfH)$ admits many finite-dimensional representations. 
Given a finite-dimensional subspace $b \subset \mfH$,
and writing $P_b\colon \mfH \to b$ for the associated orthogonal projection, we define $\pi_b\colon \mfA\F(\mfH) \to \CA\F(b)$ to be the unique morphism of $\star$-algebras such that 
\begin{equ}
	\pi_b(\balphas(f)) = \alpha^\dagger(P_b f)\;,
\end{equ}
where $\alpha^{\dagger}$ is the corresponding generator of $\CA\F(b)$.

This allows us to define on $\mfA\F(\mfH)$ the $C^*$-seminorm 
\begin{equ}
	\| \phi \|_{b} \eqdef \| \pi_{b}(\phi) \| \;.
\end{equ}
That this is indeed a $C^*$-seminorm follows from the fact that $\| \bigcdot\|$ is a $C^*$-algebra norm and that $\pi_b$ is a $\star$-algebra morphism.

Thus, for each $b \in \Gr(\mfH)$ we have a $C^*$-seminorm on $\mfA\F(\mfH)$, where $\Gr(\mfH)$ is the Grassmannian of $\mfH$, that is the set of finite-dimensional subspaces of $\mfH$. 
We could therefore use $\Gr(\mfH)$ as an underlying set of ``points'' for our non-commutative algebra $\mfA\F(\mfH)$.

\begin{remark}
It will be convenient to identify $\Im(\pi_{b}) = \CA\F(b)$ with a particular subalgebra of $\CB(\CF_{a}(\mfH))$.
 
Let $\widehat{P}_b$ be the orthogonal projection onto the finite dimensional subspace $\CF_a^{(b)} \subset \CF_a(\mfH)$ given by polynomials of $\left\{ \alpha^\dagger(f) \, \big| \, f \in b \right\}$ acting on $\1$. 
Note that $\dim(\CF_a^{(b)}) = 2^{\dim(b)}$ and that there is a canonical isomorphism $\CF_{a}^{(b)} \simeq \CF_{a}(b)$. 
In particular, this allows us to identify $\CA\F(b)$ with a subalgebra $\CB(\CF_{a}(\mfH))$ by mapping $A \mapsto \widehat{P}_b A \widehat{P}_b = A \widehat{P}_b$. Note that this is an algebra morphism because $A$ and $\widehat{P}_b$ commute.
With this identification, one has
	\begin{equ}
		\pi_{b} (\balpha^\dagger(f) + \balpha(g) ) =  \left( \alpha^\dagger(P_b f) + \alpha(P_b g)  \right) \widehat{P}_b \;,
	\end{equ}
where the $\alpha^{\dagger}$ and $\alpha$ appearing above are the generators of $\mathcal{A}\F(\mfH)$. 
The second equality follows from the fact that $\alpha^\dagger(P_b f) + \alpha(P_b g) $ preserves the subspace $\CF^{(b)}_a$.
\end{remark}

We will find it useful to restrict our set of points so that we can continue to work with super-commuting fields\footnote{See Section~\ref{subsec:extended_field} below.}.
Recall that we have fixed a unitary, $\kappa$-antisymmetric an operator $U\colon \mfH \to \mfH$ in Section~\ref{subsec:fields}. 
We define $\Gr^{U}(\mfH) \subset \Gr(\mfH)$ by setting
\begin{equ}
	\Gr^{U}(\mfH) = \left\{ b + \kappa U b \, \big| \, b \in \Gr(\mfH) \right\}\;.
\end{equ}
Since 
\begin{equ}
	(\kappa U)^2 = \kappa U \kappa U = - U^\dagger U = - \bone_{\mfH}\;,
\end{equ}
it follows that $\Gr^{U}(\mfH)$ is simply the collection of all elements of $\Gr(\mfH)$ which are $\kappa U$ invariant.
We then take our set of points to be given by $\Gamma \eqdef \Gr^{U}(\mfH)$.

\begin{remark}
The set $\Gr(\mfH)$ is a directed set with respect to the order $b \leqslant b'$ if and only if $b \subset b'$. 
In other sections we will sometimes write $\lim_{b \in \Gr(\mfH)}$, it should then be 
understood we are using this directed set structure. 

Also note that $\Gr^{U}(\mfH)$ is cofinal\footnote{That is for every $b \in \Gr(\mfH)$ there exists a $b' \in \Gr^{U}(\mfH)$, s.t.\ $b\leqslant b'$.} in $\Gr(\mfH)$, and therefore any statements we prove regarding limits over $\Gr(\mfH)$ also hold for limits over $\Gr^{U}(\mfH)$. 
\end{remark}

We now wish to  quotient out the elements of $\mfA\F(\mfH)$ that vanish on $\Gamma$, so we set 
\begin{equ}[e:defmfI]
\mfI_\Gamma \eqdef \left\{\phi \in \mfA\F(\mfH)\, \middle| \, \forall b \in \Gr^{U}(\mfH) :  \pi_b(\phi) = 0\right\}\;.
\end{equ}
Since the maps $\pi_b$ are $\star$-algebra morphisms, this is a two-sided $\star$-ideal in $\mfA\F(\mfH)$ so that $\widehat{\mfA}\F(\mfH) \eqdef \mfA\F(\mfH)/\mfI_\Gamma$\label{p:quotient} is a $\star$-algebra. 
In our notation we often do not distinguish between equivalence classes in $\widehat\mfA\F(\mfH)$ and elements in $\mfA\F(\mfH)$. 

As mentioned before, convergence with respect to the seminorms $\bigl( \| \bigcdot \|_{b} \bigr)_{b \in \Gamma}$ should be thought of as pointwise convergence. 
However, it will be convenient to end up with a Fr\'echet space so we will want to work with a countable family of seminorms. 

\begin{definition}\label{def:filtration}
We say that a sequence $(\Gamma_{n})_{n=0}^{\infty}$, with $\Gamma_{n} \subset \Gamma$, is a filtration of $\Gamma$ if the following conditions hold.
\begin{enumerate}
	\item For all $n \in \N$, $\Gamma_n \subset \Gr^U(\mfH)$.
	\item For any $n \geqslant m$, one has $\Gamma_{n} \supset \Gamma_{m}$. 
	\item For any $n \in \N$, $\sup_{b \in \Gamma_{n}} \dim(b) = n$. 
	\item $\bigcup_{n = 0}^{\infty} \bigcup_{b \in \Gamma_{n}} b$  is dense in $\mfH$. 
\end{enumerate}
\end{definition}
The first condition above enforces the compatibility with the conjugation $\kappa U$. 
The second condition ensures that our countable family of seminorms is totally ordered which will make our completion a projective limit \dash see Lemma~\ref{lem:proj_limit}. 
The third condition ensures that our topologies are weak enough to allow the convergence of renormalised Wick products \dash see Proposition~\ref{proposition:Wick_Product}. The specific choice that $\sup_{b \in \Gamma_{n}} \dim(b) = n$ was made for the sake of convenience, in principle it would be enough to simply require $\sup_{b \in \Gamma_{n}} \dim(b) < \infty$.
The last condition guarantees that an algebra element being uniformly bounded in our seminorms means that it corresponds to an operator in $\mathcal{A}(\mfH)$ \dash see Theorem~\ref{thm:ContinuityRep}. 

\begin{example}
An example of a filtration is setting, for each $n \in \N$, 
\begin{equ}\label{eq:strong_filtration}
\Gamma_{n} =
\left\{ b \in \Gamma \, \big| \, \dim(b) \leqslant n \right\}\;.
\end{equ}
An alternative choice is to fix some choice of orthonormal basis $(e_{i})_{i=1}^{\infty}$ of $\mfH$ and then set, for each $n \in \N$, 
\begin{equ}\label{eq:weak_filtration}
\Gamma_{n} 
=
\left\{ \mathrm{span} \big( \{e_{1},\dots,e_{j}\} \big) \, \Big| \, 1 \leqslant j \leqslant n \right\}\;.
\end{equ}
\end{example}
We fix some choice of filtration for the rest of the paper.
Many of our constructions do depend on this choice of filtration, but all of the statements we prove hold independently of the chosen filtration. 

For $n \in \N$ we define $\| \bigcdot \|_n$ to be the $C^*$-seminorm on $\mfA\F(\mfH)$ given by
\begin{equ}\label{eq:seminorm_def}
	\|\phi\|_n \eqdef \sup_{b \in \Gamma_{n}} \|\phi\|_b \; .
\end{equ}
We write $\mfP = \bigl(\| \bigcdot \|_{n}\bigr)_{ n \in \N }$ for our directed family of $C^{*}$-seminorms on $\mfA\F(\mfH)$.

We can now define a 
so-called locally $C^*$-algebra $\cA\F(\mfH)$\label{p:quotientComplete}, cf.\ \cite{Ino71}, as the completion of $\widehat{\mfA}\F(\mfH)$ with respect to the collection of seminorms in $\mfP$. 

The following lemma describes how the original anticommutation relations manifest in  $\cA\F(\mfH)$. 
\begin{lemma}\label{lem:commutation}
For any $f,g \in \mfH$ we have, as elements of $\cA\F(\mfH)$, 
\begin{equs}[eq:ECAR]{}
		[ \balpha^\dagger(f) ,\balpha(g)]_+ &  \in Z(\cA\F(\mfH)) \;,\\{}
		[ \balpha^\dagger(f) ,\balpha^\dagger(g)]_+ &= [ \balpha(f) ,\balpha(g)]_+  = 0\;.
	\end{equs}
Here $Z(\cA\F(\mfH))$ denotes the center of the algebra $\cA\F(\mfH)$. 
\end{lemma}
\begin{proof}
We perform computations in $\mfA\F(\mfH)$ which give us our desired results after quotienting by $\mfI_{\Gamma}$ and taking limits. 

For the second statement, we note that, for every $b \in \Gamma$, 
\[
\pi_{b} \big( [ \balpha^\dagger(f) ,\balpha^\dagger(g)]_+ \big)
=
 [ \alpha^\dagger(P_{b} f) ,\alpha^\dagger(P_{b} g)]_+ = 0\;,
\]
where $\alpha^{\dagger}$ and $\alpha$ are the generators of $\CA\F(b)$.
Therefore, $[ \balpha^\dagger(P_{b} f) ,\balpha^\dagger(P_{b} g)]_+  \in \mfI_{\Gamma}$. 
The assertion for $[ \balpha(f) ,\balpha(g)]_+$ is proven analogously. 

Turning to the first statement, we note that 
\[
\pi_{b} 
\bigl( 
[ \balpha^\dagger(f) ,\balpha(g)]_+ \bigr)
=
[ \alpha^\dagger(P_{b}f) ,\alpha(P_{b}g)]_+ = 
\scal{P_{b}f,P_{b}g}_{b} \bone_{b}\;.
\] 
It follows that, for any $a \in \mfA\F(\mfH)$, 
\begin{equ}
\pi_{b} \Bigl( \bigl[ a , [ \balpha^\dagger(f) ,\balpha(g)]_+ \bigr]_- \Bigr)
=
\scal{P_{b}f,P_{b}g}_{b} 
\bigl[
\pi_{b}(a) , \bone_{b} \bigr]_-
=0\;.\end{equ}
Therefore, $\bigl[a , [ \balpha^\dagger(f) ,\balpha(g)]_+ \bigr]_-  \in \mfI_{\Gamma}$. 
\end{proof}
We emphasize that as a locally convex algebra $\cA\F(\mfH)$ is locally $m$-convex, since all the seminorms are $C^*$-seminorms.

An important $\star$-subalgebra of $\cA\F(\mfH)$ is given by the closure of $\widehat{\mfA}\F(\mfH)$ in $\cA\F(\mfH)$, w.r.t.\ the norm $\|\phi\|_{\infty} = \sup_{n \in \N} \|\phi\|_n$.\footnote{We note here that $\|\phi\|_\infty < \infty$ necessarily holds for all $\phi \in \widehat{\mfA}(\mfH)$. For the generators this holds as we have the uniform bound $\|\balpha^\dagger(f)\|_b =\|\balpha(f)\|_b = \| P_b f\|_{\mfH} \leqslant \|f\|_{\mfH}$ and all othere elments of $\widehat{\mfA}(\mfH)$ are finite sums of finite products of creation and annihilation operators of these generators.} That is we define
\begin{equ}[e:boundedOps]
	\mfA_\infty(\mfH) \eqdef \overline{\widehat{\mfA}\F(\mfH)}^{\|\bigcdot\|_{\infty}} \; .
\end{equ}
In fact, $\mfA_{\infty}(\mfH)$ is a $C^*$-algebra with the norm $\|\bigcdot\|_{\infty}$, see  \cite[Lemma~5]{DV75}.

All four algebras $\mfA\F(\mfH)$, $\widehat{\mfA}\F(\mfH)$, $\cA\F(\mfH)$, and $\mfA_\infty(\mfH)$ naturally carry a $\Z^2$-grading, defined analogously to the grading of $\CA\F(\mfH)$, and we denote the even subalgebra by a superscript $0$ and the odd subspace by a superscript $1$.

Returning to our analogy with algebras of functions, one should think of $\cA\F(\mfH)$ as our
analogue of the locally $C^{*}$-algebra of all continuous functions and $\mfA_{\infty}(\mfH)$ 
as the analogue of the $C^*$-algebra of  bounded continuous functions.

We can return to our original $C^{*}$-algebra via the $\star$-homomorphism $\digamma\colon \mfA\F(\mfH) \to \CA\F(\mfH)$\label{p:digamma} given by mapping
\begin{equ}
	 \balpha^\dagger(f) \longmapsto \alpha^\dagger(f) 
\end{equ}
The next theorem, proven in Appendix~\ref{appendix:theoremproof}, states that $\digamma$ descends to the quotient space $\widehat{\mfA}\F(\mfH)$ and extends to a continuous $C^*$-homomorphism on $\mfA_{\infty}(\mfH)$. 
We abuse notation and also denote the quotient map $\digamma \colon \widehat{\mfA}\F(\mfH) \to \CA\F(\mfH) $ by $\digamma$.
\begin{theorem}
\label{thm:ContinuityRep}
	The $\star$-homomorphism $\digamma \colon \widehat{\mfA}\F(\mfH) \to \CA\F(\mfH)$ is well-defined and extends uniquely to a surjective $C^*$-homomorphism $\mfA_{\infty}(\mfH) \to \CA\F(\mfH)$.
\end{theorem}

Thanks to Theorem~\ref{thm:ContinuityRep}, we can pull back the state $\omega\F \colon \CA\F(\mfH) \to \C$ to a state $\bomega\F \colon \mfA_\infty(\mfH) \to \C$ by setting $\bomega\F \eqdef \omega\F \circ \digamma$ on $\mfA_{\infty}(\mfH)$. For this state we have the following isomorphism of non-commutative $\CL^2$-spaces.

\begin{theorem}\label{thm:local_L2}
	Define $\overline{\digamma}:  \mfA_{\infty}(\mfH) \rightarrow \CL^2(\CA(\mfH), \omega)$ to be given by 
	\begin{equ}
		\overline{\digamma}(a) = \big[ \digamma( a ) \big]\;.
	\end{equ}
	We then have that $\overline{\digamma}$ descends to $\mfA_{\infty}(\mfH) / \cN_{\bomega}$\footnote{For the definitions of $\CN_\omega$ and $\CL^2$ see Section~\ref{sec:GNS}} and then extends to an isomorphism 
		\begin{equ}
			\overline{\digamma} \colon \CL^2(\mfA_\infty(\mfH), \bomega) \to \CL^2(\CA\F(\mfH), \omega) \; .
		\end{equ}
	\end{theorem}
	\begin{proof}
	The fact that the map $\overline{\digamma}$ descends and extends to an isometry follows from the fact that, for any $a \in \mfA_{\infty}(\mfH)$, 
	\begin{equ}
		\|\overline{\digamma}(a)\|_{\CL^2(\CA(\mfH), \omega)}^2 = \omega\big( \digamma(a)^*  \digamma(a ) \big) = \omega \big( \digamma( a^*a) \big) = \| a\|_{\CL^2(\mfA_\infty(\mfH), \bomega)}^2 \; .
	\end{equ}
	Surjectivity of $\bar{\digamma}$ follows from the surjectivity of $\digamma: \mfA_{\infty}(\mfH) \rightarrow \mathcal{A}(\mfH)$. 
	The statement that $\CL^2(\cA(\mfH), \bomega) = \CL^2(\mfA_\infty(\mfH), \bomega)$ follows trivially. 
	\end{proof}

\begin{remark}
\label{rem:localL2}
	If we instead used the faithful state $\omega_\rho$ from Section~\ref{sec:ArakiWyss}, then we could define the following norm on $\widehat{\mfA}(\mfH)$
	\begin{equ}
		\|A\|_{\CL^{2,\infty}(\cA(\mfH),\bomega_\rho)} \eqdef \sup_{b \in \Gamma} \| \pi_b(A)\|_{\CL^2(\CA(\mfH), \omega_\rho)} \; . 
	\end{equ}
	This is a norm, rather than just a seminorm, as $\| \pi_b(\, \bigcdot \,)\|_{\CL^2(\CA(\mfH), \omega_\rho)}$ is equivalent to the operator norm on $\CB(b)$, and thus $\|A\|_{\CL^{2,\infty}(\cA(\mfH),\bomega_\rho)} = 0$ if and only if $\pi_b(A) = 0$ for all $b \in \Gamma$. If we define the non-commutative $\CL^2$-space of $\cA(\mfH)$ to be 
	\begin{equ}
		\CL^{2,\infty}(\cA(\mfH), \bomega_\rho) \eqdef \overline{\widehat{\mfA}(\mfH)}^{ \| \, \bigcdot \, \|_{\CL^{2, \infty} (\cA(\mfH), \bomega_\rho)  } } \subset \cA(\mfH) \; ,
	\end{equ}
	then $\digamma$ naturally extends to a continuous map $\CL^{2, \infty} (\cA(\mfH), \bomega_\rho) \to \CL^{2} (\CA(\mfH), \omega_\rho)$. We note here that $\CL^{2,\infty}(\mfA_\infty(\mfH), \bomega_{\rho})$ is a strictly larger space than $\CL^2(\mfA_\infty(\mfH), \bomega_{\rho})$ as any non-zero element of $\ker \digamma$ is equivalent to $0$ in $\CL^2(\mfA_\infty(\mfH), \bomega_{\rho})$ but will be different from $0$ in $\CL^{2,\infty}(\mfA_\infty(\mfH), \bomega_{\rho})$.
\end{remark}



\subsubsection{Extended Fermionic Fields}\label{subsec:extended_field}


We now turn to building an analogue of our fermionic field $\Psi$ in $\mfA_{\infty}(\mfH)$. 
Recall that we have fixed a unitary, $\kappa$-antisymmetric conjugation operator $U\colon \mfH \to \mfH$ so that 
$\Psi(f) = \alpha^\dagger(f) + \alpha(\kappa U f)$. 
We define a corresponding object in $\mfA_{\infty}(\mfH) \subset \cA\F(\mfH)$
\begin{equ}[eq:ExtfermionicField]
	\bPsi(f) \eqdef \balpha^\dagger(f) + \balpha(\kappa U f) \in \mfA_{\infty}(\mfH)\;,
\end{equ}
which, by the same abuse of notation, we will later on again write simply as $\Psi(f)$ when no confusion 
can arise. 
One obviously has $\digamma \bigl(\bPsi(f) \bigr) = \Psi(f)$.

We claim that $\bPsi$ supercommutes when seen as an element of $\widehat{\mfA}\F(\mfH)$. 
For $b \in \Gr(\mfH)$ and $f,g \in \mfH$ we have, at the level of elements of $\mfA\F(\mfH)$, 
\begin{equ}
	\pi_b\left( [ \bPsi(f), \bPsi(g)]_+ \right) = \scal{ \kappa U f, P_b g} \bone +  \scal{ \kappa U g, P_b f} \bone \; .
\end{equ}
Now we observe that for $b \in \Gamma = \Gr^{U}(\mfH)$, we have that $P_{b}$ commutes with $\kappa U$ from which it follows that $[ \bPsi(f), \bPsi(g)]_+  \in \mfI_{\Gamma}$. 

 We denote by $\cG\F(\mfH)$ the closure of the algebra generated by $\bPsi(f)$ and $Z(\cA\F(\mfH))$ in $\cA\F(\mfH)$.

\begin{remark}
	\label{rem:POVSwitching}
		We have defined $\Psi$ here as a linear functional from some indexing Hilbert space $\mfH$ into our algebra of bounded fermionic observables $\mfA_\infty(\mfH)$ and later on we will in fact see that $\bPsi \in \cD'(\R^3 ; \mfA_\infty(\mfH))$. This is akin to considering the (bosonic) space-time white noise $\xi$ as an element of $\cD'(\R^3 ; \CM(\Omega , \mu))$ although in stochastic analysis one typically thinks of $\xi$ as an element of $\CM(\Omega,\mu ; \cD'(\R^3))$.
 When performing our stochastic estimates we shall take the former point of view, whilst switching to the latter when obtaining local in time existence for our non-linear PDE pathwise in $\Omega$. 
 \end{remark}


%

\subsubsection{Extended Fermionic Wick Product}\label{sec:extendedwickproduct}
Our convergence for renormalised products will occur in the $\CC_{\s}^{\alpha}(\spacetime ; \cA\F )$ topology.
In order to renormalise products of integrated fermionic noises we would need to be able to extend $\Psi^{\diamond n}$ to the full (antisymmetric) Hilbert space tensor product $\mfH^{\wedge n}$, rather than just the algebraic one, with a codomain of observables that is larger than $\CA\F(\mfH)$. 
Again, note that this algebra has to be larger since,  for arbitrary $F \in \mfH^{\wedge n}$, one expects $\Psi^{\diamond n}(F)$ could be an unbounded operator.  

In this section we define a map $\bPsi^{\diamond n} \colon \Lambda_{n}(\mfH) \rightarrow \widehat{\mfA}(\mfH)$ and then show 
that it extends to a map  $\bPsi^{\diamond n} \colon \mfH^{\wedge n} \rightarrow \cA\F(\mfH)$. 
We saw that this is impossible when working solely with $\CA\F(\mfH)$,  however, we can do this if we lift the Wick product to $\widehat{\mfA}\F(\mfH)$ and then extend it to $\cA\F(\mfH)$.

We use an analogue of the prescription given in \eqref{eq:fermionic_isomorphism}.
We set $\bPsi^{\diamond 0}(c) \eqdef c \bone$, $\bPsi^{\diamond 1}(f) \eqdef \bPsi(f)$, and, for $n>1$ and $f_i \in \mfH$ with $i \in [n]$, we  recursively define
\begin{equs}\label{eq:extended_fermionic_wick}
	\bPsi^{\diamond n}\big(f_{1} &\wedge  \cdots \wedge f_{n}\big)
	\eqdef \bPsi(f_1) \bPsi^{\diamond (n-1)} \big(f_{2} \wedge \cdots \wedge f_{n}\big)- \\
		&- \sum_{i = 2}^n (-1)^{i} \bigl[\balpha(\kappa U f_1), \balpha^\dagger( f_i) \bigr]_+ \bPsi^{\diamond (n-2)} \big(f_{2} \wedge \cdots \wedge \widehat{ f_{i} } \wedge \cdots \wedge f_{n} \big)\;,
\end{equs}
where we note that $\bigl[ \balpha(\kappa U  f_1), \balpha^\dagger(f_i) \bigr]_+ $ belongs to the centre of
the algebra and hence our definition is independent of its position in the products. 
This essentially reproduces the
classical definition since, in $\CA\F(\mfH)$, one has
\begin{equ}
	\big[ \alpha(\kappa U  f_1), \alpha^\dagger(f_i) \big]_+ = \omega(\Psi(f_1)\Psi(f_i)) \bone \; .
\end{equ}
In fact, $\digamma(\bPsi^{\diamond n}) = \Psi^{\diamond n}$.

\begin{proposition}
\label{proposition:Wick_Product}
For any $n \in \N\setminus\{0\}$ there exists a constant $C_n < \infty$, such that for all $G \in \mfH^{\wedge n}$ and all $b \in \Gr(\mfH)$
\begin{equ}
	\bigl\| \bPsi^{\diamond n}(G)\bigr\|_b  \leqslant C_n (1+\dim b)^{\frac{n-1}{2}} \|G\|_{\mathfrak {H}^{\wedge n}} \; .
	\end{equ}
In particular, $\bPsi^{\diamond n}$ extends to a continuous map $\mfH^{\wedge n} \to \cA\F(\mfH)$. Furthermore, one also has the inequality
\begin{equ}
	\|G\|_{ \mfH^{\wedge n}} \leqslant \bigl\| \bPsi^{\diamond n}(G)\bigr\|_\infty \; .
\end{equ}

\end{proposition}
\begin{proof}
The first inequality follows directly by applying Theorem~\ref{thm:LocalUltraContr}, since $\pi_b \circ \Psi^{\diamond n}$, viewed as an operator on $\CF_a^{(b)} \subset \CF_a(\mfH)$, is a sum of $2^n$ terms of the type
\begin{equ}
	W_{r,s} \Bigl( \bigl( I^{\otimes r} \otimes \kappa ^{\otimes s} \bigr) P_b^{\otimes n} \bigl( I^{\otimes r} \otimes (\kappa U)^{\otimes s} \bigr) G	\Bigr) (N+1)^{\frac{n-1}{2}} \widehat{P}_b \; ,
\end{equ}
that $\left\| P_b^{\otimes n} \right\| \leqslant 1 $, that $ I^{\otimes r} \otimes \kappa^{\otimes s}$ and $ I^{\otimes r} \otimes (\kappa U)^{\otimes s}$ are isometric operators,
and one has the operator estimate
\begin{equ}
	(N+1)^{\frac{n-1}{2}} \widehat{P}_b \leqslant (\dim b +1)^{\frac{n-1}{2}}\widehat{P}_b\;.
\end{equ}

For the second inequality we note that, using the definitions in Section~\ref{sec:GNS}, we have
	\begin{equ}
		\|G\|_{\mfH^{\wedge n}} = \left\|[\Psi^{\diamond n}(G)]\right\|_{\CL^2(\CA\F(\mfH), \omega\F)}= \bigl\|[\bPsi^{\diamond n}(G)]\bigr\|_{\CL^2(\mfA_\infty(\mfH), \bomega\F)} \leqslant 	\bigl\|\bPsi^{\diamond n}(G)\bigr\|_{\infty}\;,
	\end{equ}
where, to obtain the inequality, we used the fact that for any $A \in \mfA_{\infty}(\mfH)$, one has $|\bomega\F(A^{\dagger}A)| \leqslant \|A^{\dagger}A\|_{\infty} = \|A\|_\infty^{2}$.
\end{proof}

This theorem allows us to analyse the behaviour of regularised singular products at the level of the Fock space and perform renormalisation by removing diverging contractions, analogously to bosonic renormalisation.

\begin{remark}
	In the context of the faithful state $\omega_\rho$ and Remark~\ref{rem:localL2}, we note that this means that $\bPsi^{\diamond n}$ maps into $\CL^{2,\infty}(\cA(\mfH), \bomega_\rho)$.
\end{remark}

\begin{remark}
One may ask what the computation above would look like in the bosonic setting.
The key difference is that $\CF_a^{(b)}$ is finite-dimensional in the fermionic case but not in the bosonic case \dash in particular $N$ is necessarily a bounded operator on $\CF_a^{(b)}$ in the fermionic case but not the bosonic one.
\end{remark}

\subsubsection{Localised Spaces}
As mentioned earlier what we mean by localising is fixing a choice of $n \in \N$ and doing our analysis with respect to the multiplicative seminorm $\| \bigcdot \|_{n}$. 
We will want to perform our localised analysis in Banach spaces.
For $n \in \N$ we define
\begin{equ}
\mfI_{n}	 = \left\{ \phi \in \cA(\mfH) \, \big| \, \forall b \in \Gamma_{n} : \dim(b) \leqslant n \implies \pi_b(\phi) = 0 \right\} \; ,
\end{equ}
and let $\cA_n(\mfH)$ denote the Banach space closure of the quotient space $\cA(\mfH) / \mfI_{n} $ with norm
\begin{equ}
	\| [\phi] \|_{\cA_n(\mfH)} \eqdef \min_{ \chi \in \mfI_n} \| \phi + \chi\|_n = \| \phi \|_{n}\; .
\end{equ} 
For the last equality above recall that, for $\chi, \chi' \in \mfI_{n}$ and $\phi \in \cA$,
\begin{equ}
	\Big| \|\phi + \chi\|_{n} - \|\phi + \chi'\|_{n}   \Big| \leqslant \| \chi - \chi'\|_{n} = 0\;.
\end{equ} 
Since $\mfI_n$ is a two-sided $\star$-ideal of $\cA\F(\mfH)$ and $\|\bigcdot\|_{n}$ is a $C^{*}$-seminorm, it is well-known that $\cA_n(\mfH)$ is a $C^*$-algebra, see also \cite[Section~2]{DV75}.

Since our norms are directed, we have for every $n \geqslant m$ a canonical $C^*$-homomorphism $\pi_{m,n}: \cA_{n}(\mfH) \rightarrow \cA_{m}(\mfH)$. 
In fact, we have the following lemma. 

\def\myref{\cite[Lemma~1]{DV75}}
\begin{lemma}[\myref]\label{lem:proj_limit}
$\cA(\mfH)$ is canonically isomorphic to the projective limit of 
\begin{equ}
	\Big( \big( \cA_{n}(\mfH) \big)_{n=0}^{\infty}, \big(\pi_{m,n} \big)_{m \leqslant n \in \N} \Big)
\end{equ}
with the isomorphism given by 
\begin{equ}
	\cA(\mfH) \ni x \mapsto \big( \pi_{n}(x) \big)_{n=0}^{\infty} \in \varprojlim_{n \in \N} \cA_{n}(\mfH)\;.
\end{equ}
\end{lemma}

\begin{remark}
	To make an analogy with ``classical'' commutative probability, the projection $\pi_n \colon \cA\F(\mfH) \to \cA_n(\mfH)$ 
can be compared to the restriction map of measurable functions $\cM(X) \to \cM(A_{n})$, $f \mapsto f\restr{A_n}$ for some measurable space $X$ and some measurable sets $A_{n} \subset X$ with $A_{n} \uparrow X$ as $n \rightarrow \infty$. 

\end{remark}


\subsection{Mixed Products}

The following Paley--Marcinkiewicz--Zygmund theorem allows us to extend hypercontractivity to combinations of bosonic and $\cA\F(\mfH)$-valued noises. 
\def\myrefhyt{\cite[Theorem~2.1.9]{Hyt16}}
\begin{theorem}[\myrefhyt]\label{thm:PMZ}
	Let $(S_1, \mu_1)$ and $(S_2, \mu_2)$ be measure spaces, $q_1,q_2 \in [1,\infty)$, and $T \colon L^{q_1}(S_1, \mu_1) \to L^{q_2}(S_2, \mu_2)$ be a bounded operator.

Given a Hilbert space $H$, $T \otimes \bone_{H}$ uniquely extends to a bounded operator $ L^{q_1}(S_1, \mu_1;H) \to L^{q_2}(S_2, \mu_2;H)$.
\end{theorem}
The next proposition introduces joint bosonic-fermionic Wick products.
\begin{proposition}
\label{ProdFieldOpProp}
Defining, for $F \in \mfB^{\otimes_s n}$ and $G \in \mfH^{\wedge m}$,
\begin{equ}[e:mixed_product]
\bigl(\xi^{\diamond n} \otimes \bPsi^{\diamond m}\bigr)(F \otimes G) = \xi^{\diamond  n}(F) \bPsi^{\diamond m}(G) \in L^{2}(\Omega, \mu) \otimes \cA\F(\mfH) \; ,
\end{equ}
on the algebraic tensor product of  $\mfB^{\otimes_s n}$ and $\mfH^{\wedge m}$, this map continuously extends to a linear map 
\begin{equ}
	\xi^{\diamond n} \otimes \bPsi^{\diamond m} \colon \mfB^{\otimes_s n} \otimes \mfH^{\wedge m} \longrightarrow L^{2}(\Omega, \mu ; \cA\F(\mfH))\;
\end{equ}
on the Hilbert space tensor product of $\mfB^{\otimes_s n}$ and $\mfH^{\wedge m}$. 
Furthermore, 
\begin{equ}[e:mixed_product_continuity]
	\xi^{\diamond n} \otimes \bPsi^{\diamond m} \colon \mfB^{\otimes_s n} \otimes \mfH^{\wedge m} \longrightarrow  L^{\infty-}(\Omega, \mu ; \cA\F(\mfH))\;
	\end{equ}
	is continuous.
\end{proposition}
\begin{proof}
	From the continuity of $\xi^{\diamond n}$ and the fact that the spaces $\mfB^{\otimes_s n} \otimes \mfH^{\wedge m}$ and $L^{2}(\Omega, \mu ; \mfH^{\wedge m})$ are the completions of the corresponding algebraic tensor products with respect to the Hilbert space structure it follows that $\xi^{\diamond n} \otimes \bone_{ \mfH^{\wedge m}}$ is a continuous map
	\begin{equ}
		\xi^{\diamond n} \otimes \bone_{ \mfH^{\wedge m}}  \colon \mfB^{\otimes_s n} \otimes \mfH^{\wedge m} \longrightarrow L^{2}(\Omega, \mu ; \mfH^{\wedge m})
	\end{equ}
	and therefore, by Nelson's hypercontraction property and Theorem~\ref{thm:PMZ}, also a continuous map
	\begin{equ}
	\label{eq:HalfWickProd}
		\xi^{\diamond n} \otimes \bone_{ \mfH^{\wedge m}} \colon \mfB^{\otimes_s n} \otimes \mfH^{\wedge m} \longrightarrow L^{\infty -}(\Omega, \mu ; \mfH^{\wedge m}) \; .
	\end{equ}
	Finally, by composing the operators in \eqref{eq:HalfWickProd} with $\bPsi^{\diamond m}$, which is continuous by Proposition~\ref{proposition:Wick_Product}, we obtain the desired result.
\end{proof}

\begin{remark}
	Note that the statement above is not making any reference to non-commutative $\CL^q$-spaces. We are only using the bosonic hypercontractivity and Proposition~\ref{proposition:Wick_Product}.
\end{remark}

\section{Fixing the Fermionic Noise and Remainder Equation}
\label{sec:the_model}

We now turn to more carefully defining the fermionic noises $\psi$, $\bar{\psi}$ appearing in \eqref{eq:lin_sig_model}, which will also fix the extended CAR algebra $\cA\F$ that our solutions take
their values in. From now on we write $\balpha$, $\balpha^\dagger$, and $\bPsi$ as $\alpha$, $\alpha^\dagger$, and $\Psi$.

We write $\mfh$ for the Hilbert space obtained by completing $\cD(\spacetime ; \C^4)$ (the set of all smooth, compactly supported $f \colon \spacetime \to \C^4$) with respect to $\| \bigcdot \|_{\mfh} = \scal{ \bigcdot, \bigcdot}_{\mfh}^{1/2}$ where, for any $f,g \in \cD(\spacetime; \C^4)$, we set
\begin{equ}\label{eq:h_inner_product}
\scal{f,g}_{\mfh} \eqdef
\int_{\R}
\int_{\T^{2}}   \Bigl\langle\Bigl((-\Delta+M^2)^{-(1+2\delta)/2} f\Bigr)(t,x),g(t,x) \Bigr\rangle_{\C^4}\ \mrd x \ \mrd t\;.
\end{equ}

Note that $\mfh = L^2(\R) \otimes H^{-\frac{1+2\delta}{2}}(\T^2) \otimes \C^4$, where $H^{-\frac{1+2\delta}{2}}(\T^2)$ is the standard Sobolev space with exponent $-\frac{1+2\delta}{2}$.
We emphasise again that unless otherwise mentioned, all vector spaces are assumed to be
complex and the tensor products are also taken over $\C$. The involution map $\kappa \colon \mfh \to \mfh$ will be the usual complex conjugation (defined as the continuous antilinear
extension of complex conjugation on $\cD$).

The role of the parameter $\delta > 0$ above is to improve the power counting of our problem, allowing us to use the Da Prato--Debussche argument \cite{DPD2}.  The true Yukawa$_{2}$ model corresponds to taking $\delta = 0$ and would require us to employ the framework of \cite{Hai14} as we already mentioned.
The upper bound, $\delta < \frac{1}{2}$, is there to ensure that the model remains singular and thus that one has to perform a renormalisation procedure.

We write $\cA\F \eqdef \cA\F(\mfh)$ for the extended CAR algebra determined by $\mfh$ and $\mfA_\infty \eqdef \mfA_\infty(\mfh)$ for the $C^*$-algebra it contains. 
We shall view its creation and annihilation operators $\alpha, \alpha^{\dagger} \colon \cD(\R^3; \C^4) \stackrel{\iota}{\hookrightarrow} \mfh \to \cA\F$ as $\cA\F$-valued, spatially periodic distributions on $\R^3$, i.e.\ as an element of the spatially periodic subspace of $ \mathfrak{L} (\cD(\R^3; \C^4), \cA\F)$,\footnote{Throughout the article, given two complex topological vector spaces $V$ and $W$, we write  $L(V,W)$ (resp.\ $\bar{L}(V,W)$) for the collection of all continuous linear (resp.\ antilinear) maps from $V$ to $W$. We then set $\mathfrak{L}(V,W) \eqdef L(V,W) \oplus \bar{L}(V,W)$.} using the canonical spatial periodisation and inclusion $\iota \colon \cD(\R^3; \C^4) \hookrightarrow \mfh$. That is, $f \in \cD(\R^3;\C^4)$ gets mapped to
\begin{equ}
	\iota(f)(t,[x]) = \sum_{k \in \Z^2} f(t,x+k) \in \cD(\spacetime; \C^4) \subset \mfh \; .
\end{equ}
Furthermore, we shall also view them as spatially periodic distributions on $\R^3$.

Switching to a more PDE-theoretic viewpoint we make use of the usual abuse of notation viewing
distributions as generalised functions,  writing for instance
\begin{equ}[e:formalNotation]
\alpha^\dagger(f) = \int f(z)  \alpha^\dagger(z)\,\mrd  z\;,\qquad
\alpha(f) = \int (\kappa f)(z) \alpha(z)  \, \mrd  z\;.
\end{equ}
This should only be taken as formal expressions, $\alpha^\dagger(z)$ itself can in general not
be realised as an element of $\cA\F$. It is, however, consistent with \eqref{e:defaa*}
which shows that $f \mapsto \alpha^\dagger(f)$ is linear while $f \mapsto \alpha(f)$ is antilinear.

When given a continuous linear or antilinear operator $A$ on $\cD(\R^3; \C^4)$ that commutes with spatial translations then its adjoint $A^*$, with respect to the $L^2$-inner product, yields again a continuous (anti)linear operator on $\cD(\R^3;\C^4)$, and we can define
\begin{equs}\label{eq:operator_on_CAR}
( A \beta)(f) \eqdef \beta(\kappa A^* \kappa  f) \quad
\text{ and }
\quad
( A \bar\beta)(f) \eqdef \bar\beta(\kappa A^* \kappa f)\;,
\end{equs}
for $\beta \in L(\cD(\R^3; \C^4);\cA\F)$ and $\bar\beta \in \bar L(\cD(\R^3; \C^4);\cA\F)$.
Note that $A^T \eqdef \kappa A^* \kappa$ is essentially a real transpose instead of a complex adjoint. This extends $A$ to a linear map on $\mfL(\cD(\R^3;\C^4), \cA\F)$.

\begin{remark}
\label{remark:Adjoint_Notation}
Although the requirement that $A$ commute with all spatial translation is not necessary for the existence of $A^*$ as a map $\cD(\R^3; \C^4) \to \cD(\R^3;\C^4)$ it is sufficient for our needs. Furthermore, this also allows us to restrict $A$ to a map on the spatially periodic distributions in $\mfL(\cD(\R^3;\C^4), \cA\F)$.
\end{remark}

We also fix a choice of decomposition $\C^4 = \C^2 \oplus \C^2$, this induces a pair of isomorphisms, one being $L( \cD(\R^3; \C^4),\cA\F) \simeq L( \cD (\R^3; \C^2),\cA\F)^2$ and the other being the analogue for the space of antilinear operators.
Then our operator-valued distributions $\alpha,\ \alpha^{\dagger}$ can be written as
\begin{equ}\label{eq:pairs_of_fields}
\alpha(z)  =
\begin{pmatrix}
a(z)\\
\bar{a}(z)
\end{pmatrix}
\quad
\text{and}
\quad
\alpha^{\dagger}(z) =
\begin{pmatrix}
a^{\dagger}(z)\\
\bar{a}^{\dagger}(z)
\end{pmatrix}\;,
\end{equ}
for the spatially periodic, operator-valued distributions
\begin{equ}
	a, \bar{a} \in \bar{L}( \cD(\R^3; \C^2), \cA\F)
	\quad \text{and} \quad a^{\dagger}, \bar{a}^{\dagger} \in L( \cD(\R^3; \C^2), \cA\F)\;.
\end{equ}

Our fermionic noise will be realised as a pair of two-dimensional Dirac spinor fields. For the present work this simply means that it takes values in $\C^2 \oplus \C^2$, however, for completeness we note that these two subspaces correspond to two separate representations of the spin group, namely the first spinor (i.e.\ the first copy of $\C^2$) transforms 
under $\Spin(2) \simeq U(1)$  by $\rho(\theta) (u,v) = (e^{-i\theta}u,e^{i\theta}v)$
and the second according to the complex conjugate representation
$\bar \rho(\theta) (u,v) = (e^{i\theta}u,e^{-i\theta}v)$.

In particular, $a, a^{\dagger}$ and $\bar{a}, \bar{a}^{\dagger}$ can be seen as two pairs of creation and annihilation operators that correspond to two different copies of the extended CAR algebra generated by the Hilbert space $\mathring{\mfh} \eqdef L^2(\R) \otimes H^{-\frac{1+2\delta}{2}}(\T^2) \otimes \C^2$ (with inner product as in \eqref{eq:h_inner_product}) sitting inside the extended CAR algebra $\cA\F$.

Using the generalised function notation, we write our fermionic noise $\Psi \in L(\cD(\R^3; \C^4),\mfA_{\infty})$ from \eqref{eq:ExtfermionicField} as
\begin{equ}\label{eq:noise_field}
\Psi(z) =
(U^* \kappa \alpha )(z)
+
\alpha^{\dagger}(z)\;,
\end{equ}
where we have made the choice 
\begin{equ}[def:chargeconj]
U \eqdef \sqrt{-\Delta+M^2} \begin{pmatrix}
			0 & \big( -\snabla+M \big)^{-1}\\
			-\bigl(\overline{\snabla}+M\bigr)^{-1} & 0
		\end{pmatrix}\;.
\end{equ}
Here we use the two-dimensional Dirac operator defined in \eqref{eq:DiracOp}. Note that $U$ is a unitary operator on $L^2(\R^3;\C^4)$, as well as $\mfh = \mathring{\mfh} \oplus \mathring{\mfh}$, and satisfies $U^T = -U$.

We also view $\Psi$ as a map $L^2(\R) \otimes H^{-\frac{1+2\delta}{2}}(\T^2) \to \mfA_{\infty} \otimes \C^4$
and an element of $\cD'(\spacetime; \mfA_{\infty} \otimes \C^4)$.\footnote{We use the notation
$\cD'(\spacetime;E)$ for subspace of spatially periodic elements of $L(\cD(\R^3),E)$ for any topological vector space $E$ throughout the rest of the article.} Analogously to \eqref{eq:pairs_of_fields} we can decompose the noise $\Psi$ into its two spinor components
\begin{equ}
\Psi(z) =
\begin{pmatrix}
\psi(z)\\
\bar{\psi}(z)
\end{pmatrix}\;,
\end{equ}
where $\psi, \bar{\psi} \in \cD'(\spacetime;  \mfA_{\infty} \otimes \C^2 )$ are given by
\begin{equs}
\psi(z) = \bar{A}^* \kappa \bar{a}(z) + a^{\dagger}(z)
&
\quad
\text{and}
\quad
\bar{\psi}(z) = A^* \kappa a(z) + \bar{a}^{\dagger}(z)\;,\\
\text{with} \enskip
A =  -\bigl(\overline\snabla+M\bigr)^{-1}  \sqrt{-\Delta+M^2}
&
\quad
\text{and}
\quad
\bar{A} = (-\snabla+M)^{-1} \sqrt{-\Delta+M^2}\;.
\end{equs}

\begin{remark}\label{rem:white_noise}
The noises $\chi,\bar\chi$ in \eqref{eq:Model_V1} could in principle be defined analogously
to $\psi$, $\bar\psi$, the difference being that one takes $ \mfh = L^2(\spacetime; \C^4)$ and $
		U  = \begin{psmallmatrix}
			0 & 1 \\
			-1 & 0
\end{psmallmatrix}$.
\end{remark}

If we fix a specific realisation of the bosonic noise $\xi$, then \eqref{eq:lin_sig_model} can be thought
of as a system of singular PDEs taking values in an infinite dimensional vector space which happens to
have some further algebraic properties. 
In fact we will adopt this ``pathwise'' approach, well-known
from $\cite{DPD2}$, rough paths theory, and regularity structures, when actually solving it.\footnote{Here our approach diverges from the one explored in \cite{Gub20} as mentioned in the introduction. There one viewed the dynamic object as a time-dependent algebra homomorphism, akin to considering the push-forward dynamic in a regular SPDE.} The solutions are then distributions with values in $\cA\F$.
We further discuss the issue of combining bosonic and fermionic probability  in Remark~\ref{rem:BosFermProb}.

\begin{remark}
\label{remark:Grassmann_Preserved}
	We note that the bosonic part of \eqref{eq:lin_sig_model} (the first equation) remains in the even part $\cA\F^0$ of $\cA\F$ whereas the other two equations remain in the odd part $\cA\F^1$. 
Furthermore, if the initial conditions of \eqref{eq:lin_sig_model} take their values in $\cG\F$, defined in Section~\ref{subsec:extended_field}, then so do the solutions of this equation, since the non-linearity of \eqref{eq:lin_sig_model} leaves $\cG\F$ invariant. This is important since the physical observables must be contained in $\cG\F$.
\end{remark}


\begin{remark}\label{rem:stoch_quantisation}
	As promised in the introduction we give more detail on how, if one took $\delta = 0$, \eqref{eq:lin_sig_model} is a stochastic quantisation of the state with action \eqref{eq:Yukawa_Action}.
	We recall a simpler situation first.

Let $\mu$ be a probability measure on $\R^d$  of the form
	$\mrd \mu(x) = Z^{-1} e^{-V(x)} \mrd x$,
	determined by the some potential (or Euclidean action) $V \colon \R^d \to \R$ and where $Z$ is a normalising constant.

We describe how, employing various ``changes of coordinates'' for SDEs, one can obtain various dynamics whose equilibrium measures can be used to recover $\mu$.

The simplest case is when $X$ is the solution to the SDE
	\begin{equ}
		\partial_t X(t) = F (X(t)) + \xi(t)
	\end{equ}
	where $F = -\nabla V$, and $\nabla$ is the gradient with respect to the usual Euclidean inner product, and $\xi(t)$ is a temporal white noise, i.e.\ \begin{equ}
		\mathbb{E}[\xi^i(t)\xi^j(s)] =  2 \delta^{ij} \delta(t-s).
	\end{equ}
In this situation, at least formally, the Fokker-Planck equation tells us that $\mu$ is an invariant measure for $X$.

However, we can replace $X$ with $X = A \widetilde X$, where $A \in \GL(n;\R) $ and $S$ with $\widetilde{V} = V \circ A$ and instead solve
	\begin{equ}[eq:Part_Func_SDE]
		\partial_t \widetilde X(t) = -\nabla \widetilde{V} (\widetilde{X}(t)) + \xi(t) = F(A \widetilde{X}(t)) + \xi(t)
	\end{equ}
	with the two stationary solutions being equivalent in the sense that for any measurable function $f$ on $\R^d$
	\begin{equ}
		\E_\mu [f]	= \E_{\widetilde{\mu}}[f \circ A ]
	\end{equ}
	where $\mu$ and $\widetilde{\mu}$ are the corresponding stationary measures.

	However, this would not quite yield \eqref{eq:lin_sig_model}. For this we note that, if instead of the white noise $\xi$ we use a noise $\tilde{\xi} = B \xi $ with $\Sigma = B B^T$, we then have to solve
	\begin{equ}
		\partial_t \widetilde X(t) = \Sigma F ( A \widetilde X(t)) + \tilde{\xi}(t) \; .
	\end{equ}
	By carefully choosing $B$ we can thus ``repair'' the equation.

	The same procedure with $B B^T$ replaced by $B \left( \begin{smallmatrix}
		0 & 1 \\
		-1 & 0
	\end{smallmatrix} \right) B^T$ applies for fermionic systems as well because they satisfy an It\^o formula, see \cite{Gub20}. In this case $A$ is given by
	\begin{equ}
		\begin{pmatrix}
			\big(\snabla + M\big)^{-1} & 0\\
			0 & \big(-\overline{\snabla} + M \big)^{-1}
		\end{pmatrix}
	\end{equ}
	yielding \eqref{eq:Replace_Spinors}, and $B$ is chosen such that
	\begin{equ}
		B \begin{pmatrix}
			0 & 1 \\
			-1 & 0
		\end{pmatrix} B^T =   U	 \; ,
	\end{equ}
	where we defined $U$ in \eqref{def:chargeconj}.

	Finally, we note that we can replace the standard scalar product $\delta_{ij}$ on $\R^d$ with an arbitrary one, $g_{ij}$, without changing \eqref{eq:Part_Func_SDE}, because all the objects transform correctly although the definitions of $\nabla V$ and $B^T$ depend on the inner product. This is important because we have made the change to $\mfh$ from $L^2(\spacetime;\C^4)$.
\end{remark}

\subsection{Heat Kernels and Power Counting}
\label{sec:kernel_not}

Most of the power counting that appears in our arguments comes from estimating the heat kernel and its derivatives by kernels of the form $\|z\|_{\s}^{\zeta}$ and applying basic rules for the pointwise products and convolution of such kernels \dash we follow \cite[Section~10.3]{Hai14} in our presentation here.

We make the following \textit{ad hoc} definition of a space of singular integration kernels that correspond to the action of heat kernels with mass and other kernels derived from these.

\begin{definition}
	Let $\mfN^{\zeta,m}$ denote the space of smooth, compactly supported kernels $K \colon \R^3 \setminus \{0\}\to \C$ such that
	\begin{equ}
		\VERT K \VERT_{\zeta; m} \eqdef \sup_{|k|_{\s} \leqslant m}	\sup_{z \in \R^3 \setminus \{0\} } \|z\|_{\s}^{|k|_{\s}-\zeta} \left| D^k K(z) \right|  < \infty \;.
	\end{equ}
	Furthermore, let $\mfN^{\zeta} \eqdef \bigcap_{m \in \N} \mfN^{\zeta,m}$.
\end{definition}
\begin{remark}
	If $\zeta' < \zeta$ then we have
$		\VERT K \VERT_{\zeta'; m} \lesssim \VERT K  \VERT_{\zeta; m}$,
	where the proportionality constant only depends on the size of the support of $K$.
\end{remark}

The basic power-counting estimate for us is \cite[Lemma~10.14]{Hai14} which we state without proof below.
\begin{lemma}
	\label{lemma:kernel_mult}
	Let $K_1 \in \mfN^{\zeta_1,m}$ and $K_2 \in \mfN^{\zeta_2,m}$. Then $K_1 K_2 \in \mfN^{\zeta_1+\zeta_2,m}$ and
	\begin{equ}
		\VERT K_1 K_2 \VERT_{\zeta_1+\zeta_2;m} \lesssim \VERT K_1  \VERT_{\zeta_1;m} \VERT K_2 \VERT_{\zeta_2;m} \; .
	\end{equ}
	If $\zeta_1 \wedge \zeta_2 > -d_{\s}$ and $\zeta\eqdef \zeta_1+\zeta_2 + d_{\s} < 0$, then $K_1 \ast K_2 \in \mfN^{\zeta,m}$ and
	\begin{equ}
		\VERT K_1 \ast K_2 \VERT_{\zeta;m} \lesssim \VERT K_1  \VERT_{\zeta_1;m} \VERT K_2 \VERT_{\zeta_2;m} \; .
	\end{equ}
\end{lemma}

\begin{lemma}\label{lemma:kernel_conv}
	If $K \in \mfN^\zeta$ with $-d_{\s}<\zeta < 0$, then, writing $K_{\eps} = K \ast \rho_{\eps}$, one has, for all $m \in \N$ and all $\bar\zeta \in [\zeta -1, \zeta)$, the bound
	\begin{equ}
		\VERT K - K_\eps \VERT_{\bar\zeta;m } \lesssim \eps^{\zeta-\bar\zeta} \VERT K \VERT_{\zeta ; m + 1} \;.
\end{equ}
\end{lemma}

We now begin by fixing a compactly supported kernel $\widetilde{\mcK}  \colon \R^3 \to \R$ that near the origin in $\R^3$ is given by the standard heat kernel
\begin{equ}
	\mcK (t,x) \eqdef \frac{\bone_{\{t>0\}}}{4\pi t} e^{- \frac{|x|^2}{4t}}
\end{equ}
but is compactly supported in space and such that for any spatially periodic distribution $u$ supported in $[0,\infty) \times \R^{2}$ one has
\begin{equ}
	\widetilde{\mcK} \ast u (z) = \mcK \ast u(z) \; .
\end{equ}
for all $z \in (-\infty,1] \times \R^2$. This is most easily achieved by taking a smooth bump function $\phi \colon \R^2 \to [0,1]$ supported in a ball around the origin such that for all $x \in \R^2$
\begin{equ}
	\sum_{k \in \Z^2} \phi(x+k) = 1 \; ,
\end{equ}
a smooth function $\chi \colon \R \to [0,1]$ such that $\chi\restr{(-\infty,1]} \equiv 1$ and $\chi\restr{[2,\infty)} \equiv 0$,
and then defining for any $(t,x) \in \R^3$
\begin{equ}
	\widetilde{\mcK}(t,x) \eqdef \sum_{k \in \Z^2} \chi(t) \phi(x) \mcK(t,x+k) \; .
\end{equ}

We also wish to incorporate the $\mfh$-inner product into the kernel picture. For this purpose, let $\CQ \colon \R^2 \setminus \{0\} \to \R$ denote the kernel of the operator $(-\Delta+M^2)^{-\frac{1+2\delta}{2}} \colon L^2(\R^2) \to L^2(\R^2)$. Explicitly it can be calculated, cf.\ \cite{MS22}, that
\begin{equ}
	\CQ(x)  = \frac{1}{2\pi \Gamma\left(\frac{1}{2}+\delta\right)} \left(\frac{2M}{|x|}\right)^{\frac{1}{2}-\delta} K_{\frac{1}{2}-\delta}(M|x|) \; ,
\end{equ}
where $K_{\frac{1}{2}-\delta}$ denotes the modified Bessel function of the second kind of order $\frac{1}{2}-\delta$. From the asymptotic behaviour of the Bessel functions we find that
\begin{equs}[2]
	\CQ(x) &\sim |x|^{2\delta-1}& \qquad & \text{when } |x| \longrightarrow 0 \; ,\\
	\CQ(x) &\sim |x|^{\delta-1} e^{-M |x|}& \qquad   &\text{when } |x| \longrightarrow \infty \; .
\end{equs}
From the exponential decay it follows that we can define the compactly supported kernel $Q \colon \R^2 \setminus \{0\} \to \R$
\begin{equ}
	Q(x) \eqdef \sum_{k \in \Z^2} \phi(x) \CQ(x+k) \; ,
\end{equ}
with singularity of order $1-2\delta$ at the origin. For all periodic distributions on $\R^2$ the convolution with $Q$ agrees with the action of $(-\Delta+M^2)^{-\frac{1+2\delta}{2}} $.

Using $\widetilde{\mcK}$ and $Q$ we list here the relevant kernels we will use:
\begin{itemize}
	\item $\CI_{B} $ will denote the linear operator mapping $\cD(\R^3) \to \cD(\R^3)$ given by convolution with the integral kernel $K \colon \R^3 \setminus \{0\} \to \R$ where
	\begin{equ}
		K(t,x) \eqdef e^{-tm^2} \widetilde{\CK}(t,x)
	\end{equ}
	corresponding to a truncated version of $(\partial_t-\Delta + m^2)^{-1}$. Its $L^2$-transpose $\CI_B^T$ is a truncation of $(-\partial_t-\Delta + m^2)^{-1}$ again mapping $\cD(\R^3) \to \cD(\R^3)$ by convolution with $K(-t,x)$.


	We also introduce the shorthand notation $K_z \eqdef K(\bigcdot - z)$ for $z \in \R^3$. 
	We note that $K \in \mfN^{-2}$. The kernel of $\CI_B \CI^T_B$ truncating $(-\partial_t^2 + (-\Delta +m^2)^2)^{-1} $ will be denoted by $K^{\star 2}$ which belongs to $\mfN^{0-}$ by Lemma~\ref{lemma:kernel_mult}) and we define $K^{\star 2}_z \eqdef K^{\star 2}(\bigcdot - z)$. 
	We use these kernels when proving the regularity properties of the renormalised products as well as in applying Theorem~\ref{SchdrThm}.

	Finally, for $u \in \CC^{\alpha}(\T^2)$ let $G_Bu(t,x)\eqdef e^{(\Delta-m^2) t}u(x)$.

	\item Analogously to above we let $\CI_F$ and $\CI_F^T$ denote integration against the truncation of $(\partial_t-\Delta+M^2 )^{-1}$ and its $L^2$-transpose, and write the kernel of $\CI^T_F$ as $G$ which belongs to $\mfN^{-2}$. The kernel of $\CI_F \CI_F^T$ will be denoted by $G^{\star 2} \in \mfN^{0-}$ and its translate by $G^{\star 2}_z$. 
	\item We also need the derivative of $\CI_F$. Let $\CI_F^{\snabla}$ and $\bigl(\CI^{\snabla}_F\bigr)^T$ denote $(\snabla+M)\CI_F$ and its $L^2$-transpose respectively. The  kernel of $\bigl(\CI_F^{\snabla}\bigr)^T$ will be denoted by $G^{\snabla} \in \mfN^{-3}$, and its translate by $G_z^{\snabla}$. This is a slight abuse of notation as $G^{\snabla}$ is in fact $\C^2$-valued. Explicitly,
	\begin{equ}
		G^{\snabla} = \snabla \left( \begin{smallmatrix}
			G \\
			G
		\end{smallmatrix} \right) \; .
	\end{equ}

	The truncated kernel of
	\begin{equ}
		\frac{ \bigl(\CI_F^{\snabla}\bigr)^{*} \bigl(\CI_F^{\snabla}\bigr)^{T} }{\bigl(\sqrt{-\Delta+M^2}\bigr)^{1+2\delta}} \approx \left(\sqrt{-\Delta+ M^2 }\right)^{1-2\delta} \big(-\partial_t^2 + (-\Delta+M^2)^2\big)^{-1}
	\end{equ}
	is denoted by $G^{\snabla,\star 2} \in \mfN^{-1+2\delta}$\footnote{Although, strictly speaking, $G^{\snabla,\star 2}$ does not fall under the auspices of Lemma~\ref{lemma:kernel_mult} one can easily check that $G^{\snabla,\star 2}$ is smooth away from $0$, as $\delta(t)Q(x)$ is being convolved with both the usual heat kernel and its time-reversed version, and that it has the correct scaling behaviour at $0$.} and its translate by $G^{\snabla, \star 2}_z$. Here $\bigl(\CI_F^{\snabla}\bigr)^{*}$ denotes $\kappa \circ \CI_F^{\snabla} \circ \kappa$, i.e.\ the complex conjugate of $\CI_F^{\snabla}$. For the sake of convenience, we have also included the kernel $Q$ coming from the $\mfh$ inner product.

	\item The notation for the terms containing $-\overline{\snabla}$ instead of $\snabla $ are analogous to the above with $\snabla$ replaced by $\overline{\snabla}$.
\end{itemize}

Convolution with $\rho_\eps$ will be denoted by an index $\eps$, e.g.\ the mollified version of $G_{z}^\snabla$ will be denoted by $G_{z,\eps}^{\snabla}$.

\subsection{Da Prato--Debussche Argument}
\label{sec:PDTrick}

The problem that arises when attempting to solve \eqref{eq:lin_sig_model} is that the driving noises have negative regularities, in the sense of the H\"older--Besov spaces defined above.
A scaling argument shows that $\xi \in \CC_{\s}^{-2-\kappa}$ almost everywhere, for all $\kappa > 0$, and $\psi,\overline{\psi} \in \CC_{\s}^{-\frac{3}{2}+\delta}$.
(Note that as usual there is no loss of regularity for the fermionic noise.)

Since the action of the heat kernel improves regularity by $2$ (this is a formulation of our Schauder estimate in parabolic H\"older--Besov spaces), we expect the solutions to be at most in $\CC^{0-\kappa}_\s$ for the bosonic part $\phi$ and in $\CC^{\frac{1}{2}+\delta}_\s$ for the fermionic parts $\upsilon,\bar\upsilon$. Taking into account the derivative we see that $(\snabla + M)\upsilon$ and $(-\overline{\snabla}+M)\bar\upsilon$ have regularity $\delta -\f12$.
This means we are dealing with a non-linear equation where the solutions are in fact genuine distributions and not functions, especially when $\delta < \f12$. It is well-known, see Theorem~\ref{YngMultThm}, that a product of a function of regularity $\theta > 0$ with a distribution of regularity $\theta' < 0$ is canonically well-defined if and only if  $\theta + \theta' > 0$ \dash the resulting product will have regularity $\theta'$.
This means that all the products in \eqref{eq:lin_sig_model} are ill-defined.

To overcome this problem we adapt the argument of \cite{DPD2}. 
Let $(\<1>, \<1IF>,\<1IA> )$ be the stationary solutions to the linear part of \eqref{eq:lin_sig_model}, i.e.\ we set
\begin{equ}
		\partial_t \<1> = (\Delta-m^2) \<1>  + \xi,	\quad
		\partial_t \<1IF> =  (\Delta-M^2) \<1IF>  + \psi,\quad
		\partial_t \<1IA>  =  (\Delta-M^2) \<1IA> + \bar\psi.
\end{equ}
We then introduce the ansatz
\begin{equs}[e:perturbative_ansatz]
		\phi = U + \<1>, \qquad
		\upsilon =  Y+\<1IF>, \qquad
		\bar\upsilon  =\overline{Y} + \<1IA>.
\end{equs}
Our hope\footnote{The fact that this turns out to be the case is called subcriticality in singular SPDEs, or UV super-renormalisability in field theory.} is that $(U,Y, \overline{Y})$ will have better regularity than $(\<1>, \<1IF>,\<1IA> )$ and that one can define a well-posed equation for them.

We also introduce $\<1F> \eqdef (\snabla + M) \<1IF>$ and $\<1A> \eqdef (-\overline{\snabla}+M) \<1IA>$, which are $\mfA_\infty \otimes \C^2$-valued distributions of regularity $\delta-\f12$.
Using the notation $Y^{\snabla} \eqdef (\snabla + M) Y $ and $\overline{Y}^{\snabla}$, we can rewrite \eqref{eq:lin_sig_model}, at least formally, as follows
\begin{equs}\label{eq:lin_sig_model_DPD1}
	\partial_t U &= (\Delta-m^2) U -\\
	& \qquad - g \Big(  \bilin{\overline{Y}^{\snabla} , Y^\snabla}  + \bilin{\overline{Y}^{\snabla} ,\<1F> } + \bilin{\<1A>, Y^\snabla} + \bilin{\<1A>,\<1F>} \Big)  - \\
	& \qquad - \lambda \left( U^3 + 3 U^2 \<1> + 3 U \<1>^2 + \<1>^3  \right) \; ,	\\
	\partial_t Y &=  (\Delta-M^2) Y  - g \Big(  U  Y^{\snabla}  + \<1>  Y^{\snabla} + U \<1F>  + \<1> \ \<1F> \Big) \; ,\\
	\partial_t \overline{Y} &=  (\Delta-M^2) \overline{Y} - g \Bigl(  U  \overline{Y}^{\snabla}  + \<1> \overline{Y}^{\snabla} + U \<1A>  + \<1> \ \<1A> \Bigr) \; ,
\end{equs}
with initial conditions $U_0 = \phi_0 - \<1>(0)$, $Y_0 = \upsilon_{0} - \<1IF>(0)$, and $\overline{Y}_{0} = \bar\upsilon_{0} - \<1IA>(0)$.
While the products $\bilin{\<1A>,\<1F>}$,  $\<1> \ \<1F> $, $\<1> \ \<1A>$, $\<1>^2$, and $\<1>^3$ are not canonically defined (since the sum of their regularities is not positive), we can define these products by using stochastic methods and renormalising, obtaining
\begin{equ}[e:trees]
	\<2AF> \in \CC_\s^{-1+2\delta}(\spacetime ; \cA\F^0), \qquad \<2F>,\<2A> \in \CC_\s^{-\frac{1}{2}+\delta-\kappa}(\spacetime ; \mfA_\infty^1)\;, \qquad \<2>, \<3> \in \CC_\s^{-\kappa}(\spacetime) \; ,
\end{equ}
see Subsection~\ref{sec:fermionic_Tree} and Subsection~\ref{sec:Mixed_Tree} \dash the  trees $\<2AF>$, $\<2>$, and $\<3>$ above require the insertion of diverging renormalisations $C^1_{\eps}$ and $C^2_{\eps }$ \footnote{One does not see this particular renormalisation in some treatments of Yukawa$_2$ as it is implicitly removed by normal ordering creation and annihilation operators in $\<2AF>$.} in the bosonic equation while the other trees above do not introduce renormalisation constants.

The assumption that the objects in \eqref{e:trees} are the most irregular terms in their respective equations turns out to be a self-consistent one, in fact one can work with the assumption that  $U$ and $Y,\overline{Y}$ have regularity $1+2\delta$ and $\frac{3}{2}+\delta-$ respectively\footnote{Here, we just write ``$-$'' in the regularity to mean minus $n\kappa$ for some $n \in \N$. Recall that $\kappa$ can be taken arbitrarily small so the precise value of $n$ is unimportant.} \dash all the relevant products involving $U$ and $Y,\overline{Y}$ will be canonically defined and will have sufficient regularity so that after they are acted on by the heat kernel they will exceed our regularity assumptions.

\begin{remark}\label{rem:non-local}
If we remove our regularisation (choose $\delta = 0)$ then some of the products involving $U$, $Y$, and $\overline{Y}$ become ill-defined.
One can try to make the remainders $U$, $Y$, and $\overline{Y}$ more regular by including more trees in the ansatz \eqref{e:perturbative_ansatz}, but regardless of how far one goes it will not be possible to make a self-consistent assumption on the regularities of our remainders that lets us close the argument.\footnote{This is the same obstruction encountered when trying to apply the argument of \cite{DPD2} to the stochastic quantisation of $\Phi^4_3$.}

The source of the difficulty can be linked to the need to construct new trees on the right hand side of the $U$ equation, for instance $\<2F2A>$ .
The stochastic construction of $\<2F2A>$ is different from that of \eqref{eq:lin_sig_model_DPD1} in that it involves a ``non-local'' renormalisation, for the tree $\<2F2A>$ this is due to the divergent subtree $\<1F2A>$.
All of these new divergences requires the insertion of a $\tilde{C}_{\eps} \phi$ renormalisation counterterm to the bosonic equation which, at the level of the field theory, would correspond to a ``mass renormalisation'' of the boson. We expect that the usual methods used to go beyond the Da Prato--Debussche regime in the bosonic case can be adapted to this setting to overcome this difficulty. 
\end{remark}

\begin{remark}
\label{rem:BosFermProb}
	Note that we take an ``asymmetric'' approach where we work with the bosons pointwise on $\Omega$ but treat the fermions as algebra elements, that is we work on $\CM(\Omega, \mu; \cD'(\spacetime_+ ; \cA\F))$.

A more symmetric approach would be to view $\CM(\Omega, \mu)$ as a (commutative) topological $\star$-algebra $\cA_B$ that we may consider to be the ``bosonic algebra''.
Then heuristically the solution could be considered as an element of $\cD'(\spacetime_+ ; \cA_B \otimes \cA\F)$. 
However, completing the algebraic tensor product $\cA_B \otimes \cA\F$ using the topology of 
 convergence in probability on  $\cA_{B}$ 
would again give us the space $\CM(\Omega, \mu; \cD'(\spacetime_+ ; \cA\F))$.
\end{remark}


\section{Stochastic Estimates}
\label{sec:Renormalisation}

\subsection{Regularity of the Noises and Linear Solutions}\label{subsec:noise_reg}

We start with regularity estimates for the fermionic noise.
From the definition of the extended CAR algebra, we have the following bound on the fermionic driving noise.
Below, we define, for $\eps > 0$, $\Psi_{\eps} \eqdef \Psi \ast \rho_{\eps}$ and analogously for $\psi, \bar\psi$, and $\xi$. Here $\rho_{\eps}$ is again the scaled mollifier introduced in Section~\ref{subsec:notation}.

\begin{lemma}
	\label{lem:NoiseEst}
For any $f \in \mfh$, one has
\begin{equ}\label{eq:noise_bound}
\|\Psi(f)\|_\infty
\leqslant
\sqrt{2}
 \|f\|_{\mfh}\;.
\end{equ}
It follows that, as an operator-valued distribution, $\Psi \in \CC_{\s}^{-3/2+\delta}(\spacetime ;\mfA_{\infty} \otimes \C^{4})\;$.
Moreover, $\Psi_{\eps} \rightarrow \Psi$ in  $\CC_{\s}^{-3/2+\delta-\nu}(\spacetime ;\mfA_{\infty} \otimes \C^{4})$ as $\eps \downarrow 0$ for any $\nu \in (0,1]$.
\end{lemma}
\begin{proof}
From \eqref{eq:CAR} one has, for all $f \in \cD(\R^3;\C^{4})$ and all $b \in \Gr(\mfH)$,
\begin{equs}
	\big[ \pi_b(\Psi(f)),\pi_b(\Psi(f))^\dagger \big]_+ & = \big[ \alpha(P_b(U \kappa f)), \alpha^\dagger(P_b(U\kappa f)) \big]_+ +\\
	& \qquad + \big[ \alpha^\dagger(P_b f), \alpha(P_b f) \big]_+ = \\
	&= \scal{P_b U\kappa f, P_b U \kappa f}_{\mfh} \, \bone + \scal{P_b f,P_b f}_{\mfh} \, \bone \leqslant\\
	& \leqslant 2 \|f\|_{\mfh}^2 \, \bone\;.
\end{equs}
Therefore, for any $u \in \CF^{(b)}_a$,
\begin{equs}
	\|\pi_b(\Psi(f))u\|^2 & \leqslant \|\pi_b(\Psi(f))u\|^2  + \|\pi_b(\Psi(f))^\dagger u\|^2 = \\
	& =  \big\langle u, \big[ \pi_b(\Psi(f)),\pi_b(\Psi(f))^\dagger \big]_+ u \big\rangle_{\CF} \leqslant\\
	& \leqslant 2 \|f\|^2_{\mfh} \|u\|_{\CF}^2\;.
\end{equs}

The second statement of the lemma follows from \eqref{eq:noise_bound} by estimating the given H\"older--Besov norm \eqref{eq:besov_norm} using the scaling properties of the norm $\| \bigcdot \|_{\mfh}$. Explicitly given $f \in \cD(\R^3; \C^4)$, let
$
	\tilde{f}(t,x) \eqdef \sum_{k \in \Z^2} f(t,x+k)
$
be the corresponding periodised element in $\cC^\infty(\spacetime; \C^4) \subset \mfh$. By linearity we find that
$
	Q \ast \tilde{f} = \widetilde{Q \ast f}
$
and, therefore,
\begin{equ}
	\| \tilde{f} \|_{\mfh}^2 \leqslant \|f\|_{L^2} \big\|Q \ast f\big\|_{L^2} \; .
\end{equ}
By slight abuse of notation, $Q \ast f$ only denotes the convolution along the spatial variable. 
Using the scaling property of $Q$ at the origin, a straighforward calculation then yields
\begin{equ}
	\| \widetilde{\CS^\lambda_{\s, z} f } \|_{\mfh} \lesssim \lambda^{-\frac{4/2}{2}} \lambda^{\frac{1+2\delta- 4/2}{2}} = \lambda^{-\frac{3}{2}+\delta} \; ,
\end{equ}
where the constant only depends on the support of $f$ and its $\cC$-norm. This directly implies that $\Psi \in \CC_{\s}^{-3/2+\delta}(\spacetime ;\mfA_{\infty} \otimes \C^{4})$.

The final statement regarding convergence is a simple consequence of the fact that
\begin{equ}
	\left\|\CS^\lambda_{\s, z} \left(\rho_\eps \ast f \right) - \CS^\lambda_{\s, z} f\right\|_{L^2} \lesssim \left( \lambda^{-\frac{4}{2}- 1} \eps \| f \|_{\cC^{1}} \right) \land \left( \lambda^{-\frac{4}{2}} \|f\|_{\cC} \right)
\end{equ}
and interpolating between the two bounds, i.e.\ for any $\nu \in [0,1]$\;,
\begin{equ}
	\lambda^{-2- 1} \eps \land \lambda^{-2} \leqslant \left( \lambda^{-2-1} \eps \right)^{\nu} \left( \lambda^{-2} \right)^{1-\nu} = \lambda^{-2- \nu} \eps^\nu \; .
\end{equ}
We are allowed to use the $\cC^1$-norm of $f$ thanks to the definition of the space $\CC^{-\frac{3}{2}+ \delta}_{\s}$ (cf.\ Appendix~\ref{sec:Holder_Space}).
Analogously we find that
\begin{equ}
	\left\|\CS^\lambda_{\s, z} \left(\rho_\eps \ast Q \ast f \right) - \CS^\lambda_{\s, z}( Q \ast f)\right\|_{L^2} \lesssim \lambda^{-1+2\delta-\nu} \eps^\nu \| f \|_{\cC^{1}}\;,
\end{equ}
where the constants above depend only on the size of the support of $f$.
\end{proof}

\begin{lemma}
\label{prop:fermionic_Noise_Properties}

One has $\CI_F(\Psi) \in \CC_\s^{\frac{1}{2}+\delta}(\spacetime; \mfA_\infty \otimes \C^4)$.
Moreover, for any $\theta \in \left(0,\frac{1}{2}\right)$, $\CI_F(\Psi) \in \CC^{\theta} \big( [0,1] ; \CC^{\frac{1}{2}+\delta-2\theta} \big( \T^2 ; \mfA_\infty \otimes \C^4 \big)\big)$.

Additionally, for $\nu \in (0,1]$, $\CI_F(\Psi_{\eps}) \rightarrow \CI_F(\Psi)$ as $\eps \downarrow 0$ in the spaces $\CC_\s^{\frac{1}{2}+\delta-\nu}(\spacetime; \mfA_\infty \otimes \C^4)$ and $\cC(\R ; \CC^{\frac{1}{2}+\delta-\nu} (\T^2 ; \mfA_\infty \otimes \C^4) )$.
\end{lemma}

\begin{proof}
	The first and third assertions follow from an application of the Schauder estimate Theorem~\ref{SchdrThm}.

	Concerning the second, observe that we can rewrite the difference needed for the H\"older estimate as 
\begin{equs}
\int \scal{ \big( \CI_F&(\Psi)(t, y)- \CI_F(\Psi)(s,y) \big),  \big( \CS^\lambda_{x} \eta \big) (y)}_{\C^4} \d y \\
{}& = \CI_F(\Psi) \big( ( \delta_t - \delta_s ) \otimes \CS^\lambda_{x} \eta \big)
= - \partial_t \CI_F(\Psi) \big( \one_{[s,t]} \otimes \CS^\lambda_x \eta \big)\;,
\end{equs}
where $\eta \in \cD(\T^2;\C^4)$ such that $\int_{\T^2} \eta = 0$.

The evaluation at points $t$, $s$, and the following calculations are admissible as $\CI_F(\Psi)$ is a continuous function.
	The $\| \bigcdot\|_\infty$-norm of this expression can be calculated by applying $\partial_t \CI_F^T$ to the argument of $\Psi$ and calculating its $\mathfrak h$ norm squared, see Lemma~\ref{lem:NoiseEst}.
	We get
	\begin{align*}
		\left\| \partial_t \CI_F^T \left(\one_{[s,t]} \otimes \widetilde{\CS^\lambda_x \eta} \right) \right\|^2_{L^2 \otimes H^{-\frac{1+2\delta}{2}}} = 4 \left\langle \one_{[s,t]} \otimes \widetilde{\CS^\lambda_x \eta} , - \partial_t^2 \CI_F^T\CI_F \left( \one_{[s,t]} \otimes \widetilde{\CS^\lambda_x \eta} \right) \right\rangle \; ,
	\end{align*}
	where we are using the $L^2 \otimes H^{-\frac{1+2\delta}{2}}$ on the right-hand side. 
	Now, since $s,t \in [0,1]$ and $ \widetilde{\CS^\lambda_x \eta}$ is spatially periodic, we may replace $- \partial_t^2 \CI_F^T\CI_F$ with $- \partial_t^2 (-\partial_t^2 + (-\Delta+M^2)^2)^{-1}$. In order to separately estimate the space and time regularities we
	convert the sum in the denominator into a product which can be done via functional calculus using Young's inequality for products, i.e.\
	\begin{equ}
		|\partial_t|^{1+2 \theta} (-\Delta+ M^2)^{1-2\theta} \leqslant \frac{1+2\theta}{2} (-\partial_t^2) + \frac{1-2\theta}{2} (-\Delta+M^2)^2
	\end{equ}
	giving us
	\begin{equ}
		\frac{- \partial_t^2}{-\partial_t^2 + (-\Delta+M^2)^2} \lesssim \frac{- \partial_t^2}{|\partial_t|^{1+2 \theta} (-\Delta + M^2)^{1-2\theta}}  = \frac{|\partial_t|^{1+2\theta}}{(-\Delta+M^2)^{1-2\theta}}
	\end{equ}
	which has two separate scaling degrees for space and time. We find
	\begin{equ}
		\left\| \partial_t \CI_F^T \left(\one_{[s,t]} \otimes \widetilde{\CS^\lambda_x \eta} \right) \right\|_{L^2 \otimes H^{-\frac{1+2\delta}{2}}} \lesssim |t-s|^{\theta} \lambda^{\frac{1}{2}+\delta-2\theta}
	\end{equ}
	proving the claim. 
	The last claim follows analogously from the third one. \end{proof}


We now turn to the corresponding bosonic objects.

\begin{lemma}
\label{prop:bosonic_Noise_Properties}
	For every $\kappa> 0$ and $\theta \in \left(0,\frac{1}{2}\right)$ we have, almost everywhere (in the probability space of the bosonic noise)
\begin{equ}
	\xi \in \CC^{-2-\kappa}_{\s}(\spacetime )\;,
	\enskip
	\CI_B(\xi) \in \CC^{-\kappa}_{\s}(\spacetime)\;,
	\enskip \text{and}
	\enskip
	\CI_B(\xi) \in \CC^{\theta}\left( \R; \CC^{-\kappa - 2 \theta}\left(\T^2\right) \right)\;.
\end{equ}
Moreover, as $\eps \downarrow 0$, one has, in probability, $\xi_{\eps} \rightarrow \xi$ in $\CC^{-2-\kappa}_{\s}(\spacetime )$ and $\CI_B(\xi_{\eps}) \rightarrow \CI_B(\xi)$ in $\CC^{-\kappa}_{\s}(\spacetime)$ and $ \cC\left( \R; \CC^{-\kappa - \theta}\left(\T^2\right) \right)$.
\end{lemma}
\begin{proof} The first statement is standard, proven by a Kolmogorov type argument.
The second follows from the first and Theorem~\ref{SchdrThm}.
The third can also be proven by a standard Kolmogorov type estimate (analogously to the time regularity estimate in Lemma~\ref{prop:fermionic_Noise_Properties}).
Finally, the statement regarding convergence can be argued via compactness.\end{proof}

As in Section~\ref{sec:PDTrick}, we frequently use graphical tree notations in what follows.
Note that Lemma~\ref{prop:fermionic_Noise_Properties} immediately gives $\<1F>  \in \CC_\s^{-\frac{1}{2}+\delta}(\spacetime; \mfA_\infty \otimes \C^2)$ and $\<1A> \in \CC_\s^{-\frac{1}{2}+\delta}(\spacetime; \mfA_\infty \otimes \C^2)$.
We use subscripts to encode regularisation, writing $\<1>_{\eps} \eqdef \CI_B(\xi_{\eps})$, $\<1IF>_{\eps} \eqdef \CI_F(\psi_{\eps})$ and $\<1IA>_\eps \eqdef \CI_F(\overline\psi_{\eps})$, and similarly for  $\<1F>_{\eps}$ and $\<1A>_{\eps}$.

\subsection{Convergence of Products}\label{RenormSec}

We now turn to constructing the fermionic Wick product $\<2AF>$ along with the mixed boson \slash fermion products $ \<2F>, \<2A>$. For the bosonic Wick products $\<2>, \<3>$ we simply cite the corresponding results as these objects are well-known in the literature. 

\subsubsection{Fermionic Wick Square}
\label{sec:fermionic_Tree}
We start by rewriting $\bilin{\CI_F^{\overline{\snabla}} \big( \overline\psi_\eps\big), \CI^{\snabla}_F \big(\psi_\eps\big) }(z)$ as a sum in zeroth order and second order chaoses. We will then show that the projection to the second chaos gives a well-defined operator-valued distribution in the limit and our renormalisation prescription will be the subtraction of the zeroth order chaos (which is just a constant).
For $f \wedge g \in \mfh^{\wedge 2}$, \eqref{eq:extended_fermionic_wick} simplifies to
\begin{equ}[e:wick_square]
\Psi^{\diamond 2}(f \wedge g)
=
\Psi(f) \Psi(g) - \bigl[ \alpha(\kappa U f), \alpha^\dagger(g) \bigr]_+ \;.
\end{equ}
The key estimate for us from Proposition~\ref{proposition:Wick_Product} is the following: one has for all $n \in \N$, uniformly in  $F \in \mfh^{\wedge 2}$, the estimate
\begin{equ}[e:ferm_wick_bound]
\| \Psi^{\diamond 2}(F)  \|_n
\lesssim \|F\|_{\mfh^{\wedge 2}}\;,
\end{equ}
with the constant only depending $n$.
For writing out various objects explicitly it will be helpful to refer to individual (vector) components of the fermionic noise $\Psi$.
For $i \in [4]$ let $\pi_i \colon \mfh \to \mfh$ be the projection onto the subspace of $\mfh$ where only the $i^{\text{th}}$ component of the functions are non-zero.
Then we set
\begin{equ}
\Psi_i \eqdef \Psi \circ \pi_i \; ,  \quad \alpha^\dagger_i \eqdef \alpha^\dagger \circ \pi_i \; , \quad \alpha_i \eqdef \alpha \circ \pi_i \; .
\end{equ}
Abusing notation, we also view $\Psi_i$, $\alpha_i^\dagger$, and $\alpha_i$ as maps $L^2(\R) \otimes H^{-\frac{1+2\delta}{2}}(\T^2)  \to \mfA_\infty$ and elements of $ \cD'(\spacetime; \mfA_\infty)$.
We also consider the component maps $\Psi_i \diamond \Psi_j $ of $\Psi^{\diamond 2}$ given by setting, for $F \in \mfh^{\otimes 2}$,
\begin{equ}
	(\Psi_i \diamond \Psi_j)(F) \eqdef \Psi^{\diamond 2} \Big(  (\pi_i \otimes \pi_j )  (F) \Big)\;.
\end{equ}
which are well-defined maps on the full Hilbert space tensor product by Proposition~\ref{proposition:Wick_Product}.

We then rewrite the product distribution $\bilin{ \overline{\psi}, \psi}$ on $\mathring{\mfh} \times \mathring{\mfh}$ as
\begin{equs}
	\bilin{ \overline{\psi}(f), \psi(g)}&= \Psi_3(f) \Psi_1(g) + \Psi_4(f) \Psi_2 (g)  \\
	& =  (\Psi_3  \diamond \Psi_1)(f \wedge g) + (\Psi_4 \diamond \Psi_2) ( f \wedge g)  +\\
	& \qquad +\left[ \alpha_3(\kappa U ( 0,f ) ) ,  \alpha_1^\dagger ( g,0 ) \right]_+  + \left[ \alpha_4(\kappa U ( 0,f ) ) ,  \alpha_2^\dagger ( g,0 ) \right]_+ \; . 
\end{equs}
As a shorthand for the first two terms in the above expression we define the linear map
\begin{equs}
\bilin{ \overline{\psi}, \psi}^{\diamond} \colon
		\mfh^{\wedge 2} & \longrightarrow \cA\F\\
		F & \longmapsto (\Psi_3   \diamond \Psi_1)(F) + (\Psi_4  \diamond \Psi_2) (F)\;.
\end{equs}

We define the maps $\mfZ_{\eps}[\<2AF>], \mfZ[\<2AF>] \in \cD'(\spacetime; \mfh^{\wedge 2})$ by setting
\begin{equ}
	\mfZ_\eps[\<2AF>](z)\eqdef \left(0, G^{\overline{\snabla}}_{z,\eps} \right) \wedge \left(G^{\snabla}_{z,\eps},0 \right)
	\quad
	\text{and}
	\quad
	\mfZ[\<2AF>](z) \eqdef \left(0, G^{\overline{\snabla}}_{z} \right) \wedge \left(G^{\snabla}_{z},0 \right)
	\;.
\end{equ}
While it is clear that $\mfZ_{\eps}[\<2AF>]$ belongs to $\cD'(\spacetime; \mfh^{\wedge 2})$, this needs to be verified for $\mfZ[\<2AF>]$ but, to avoid breaking up our exposition here, we leave that to the first statement of Lemma~\ref{lem:conv_fermionic_kernels}.
With these definitions we have
\begin{equs}[e:ferm_wick_decomp]
	\bilin{ \<1A>_\eps, \<1F>_\eps }(z)&\eqdef
	\bilin{\CI^{\overline{\snabla}}_F \left( \overline\psi_\eps\right), \CI^{\snabla}_F \left(\psi_\eps\right)}(z) = \\
	{}&=
	\bilin{ \overline{\psi}, \psi}^{\diamond}
	\Big( \mfZ_{\eps}[\<2AF>](z) \Big) + \mathbf{C}_{\<2AF>}^{\eps}(z)\;,
\end{equs}
where
\begin{equ}[eq:def_RenormElem]
	\mathbf{C}_{\<2AF>}^{\eps}(z) \eqdef \sum_{i = 1}^2 \left[ \alpha_{i+2}\left( \kappa U  \bigl(0, G^{\overline{\snabla}}_{z,\eps} \bigr) \right), \alpha_{i}^\dagger \bigl(G^{\snabla}_{z,\eps}, 0 \bigr) \right]_+ \in Z(\cA\F) \;.
\end{equ}
We then define $\<2AF>_{\eps}, \<2AF> \in \cD'(\spacetime ; \cA\F)$ by setting
\begin{equs}[e:def_of_fermionic_square]
	\<2AF>_{\eps}(z) &\eqdef
	\bilin{ \<1A>_\eps, \<1F>_\eps }(z)
	-
	\mathbf{C}_{\<2AF>}^{\eps} (z)\;
	=
	\bilin{ \overline{\psi}, \psi}^{\diamond}
	\Big( \mfZ_{\eps}[\<2AF>](z) \Big)\;,\\
	\<2AF>(z)
	&\eqdef
	\bilin{ \overline{\psi}, \psi}^{\diamond}
	\Big( \mfZ[\<2AF>](z) \Big)\;.
\end{equs}
Again, the definition of $\<2AF>$ only makes sense once we have proven the first statement of Lemma~\ref{lem:conv_fermionic_kernels}, which we proceed to do now.

\begin{lemma}\label{lem:conv_fermionic_kernels}
One has $\mfZ[\<2AF>] \in \CC_{\s}^{-1+2\delta}(\spacetime ; \mfh^{\wedge 2})$, and one has $\mfZ_{\eps}[\<2AF>] \rightarrow \mfZ[\<2AF>]$ in $\CC_{\s}^{-1+2\delta-\nu}(\spacetime ; \mfh^{\wedge 2})$ as $\eps \downarrow 0$ for any $\nu \in (0,1]$.
\end{lemma}
\begin{proof}
For the first statement note that, for every $f \in \cD(\R^3)$,
\begin{equs}\label{eq:kernel_bound_ferm_square}
\| \mfZ[\<2AF>](f) \|_{\mfh^{\wedge 2}}^2
&\leqslant \iint\limits_{\spacetime \times \R^3} \widetilde{(\kappa f)}(z) G^{\overline{\snabla},\star 2}_z(y) G^{\snabla,\star 2}_z(y) \tilde{f}(y) \d y \d z = \\
{}&=
\scal{\tilde{f}, \bigl(G^{\overline{\snabla},\star 2}  G^{\snabla,\star 2} \bigr) \ast \tilde{f} }\;,
\end{equs}
where $G_z^{\snabla, \star 2}, G_z^{\overline{\snabla}, \star 2}$ and $G^{\snabla, \star 2}, G^{\overline{\snabla}, \star 2}$ are as defined in Section~\ref{sec:kernel_not}, $\tilde{f}$ again denotes the partial periodisation of $f$, and we have estimated the antisymmetrised inner product using the normal one. Since $G^{\overline{\snabla},\star 2} G^{\snabla,\star 2} \in \mfN^{-2+4\delta}$, the first statement directly follows from Proposition~\ref{prop:CCalphaCrit}.
The second statement follows from combining Proposition~\ref{prop:CCalphaCrit} and Lemma~\ref{lemma:kernel_conv} for the convergence of $G_\eps^{\snabla, \star 2}, G_\eps^{\overline{\snabla}}$ to $G^{\snabla, \star 2}, G^{\overline{\snabla}}$. \end{proof}
Note that the renormalisation counterterm in \eqref{eq:def_RenormElem} 
is a lift of the of the counterterm one would expect from Wick renormalisation in $\CA$. Indeed, a straightforward calculation gives
\begin{equ}\label{eq:def_RenormElem2} 
\digamma\big( \mathbf{C}_{\<2AF>}^{\eps} (z) \big)
=
\omega \Big(\bilin{ \overline{\psi}(G^{\overline{\snabla}}_\eps), \psi(G^{\snabla}_{\eps})} \Big) \bone 
\eqdef
C_{\<2AF>}^{\eps} \bone\;.
\end{equ}

Our main result on the fermionic Wick square is given by the following proposition.
Here we use the notation $\|\bigcdot \|_{\alpha;\mfp}$ described in Appendix~\ref{sec:spaces_of_dist}, where $\alpha$ is the regularity exponent for the H\"{o}lder--Besov space and $\mfp$ denotes the seminorm. 
We do not reference a compact set $\K$ here since this estimate is actually uniform over $\spacetime$. 

\begin{proposition}\label{prop:fermionic_wick_conv}
One has $\<2AF> \in \CC_{\s}^{-1+2\delta}(\spacetime ; \cA\F)$. Moreover, for any $\nu \in (0,1]$ and any $n \in \N$
\begin{equ}
	\left\| \<2AF> - \<2AF>_\eps \right\|_{-1+2\delta-\nu ; n} \lesssim \eps^{\nu/{4}} \;,
\end{equ}
so that  $\<2AF>_{\eps} \rightarrow \<2AF>$ in $\CC_{\s}^{-1+2\delta-\nu}(\spacetime ; \cA\F)$ as $\eps \downarrow 0$. The constant of proportionality here only depends on $n$.
\end{proposition}
\begin{proof}
Recalling that $\bilin{ \overline{\psi}, \psi}^{\diamond} \colon \mfh^{\otimes 2} \rightarrow \cA\F$ is continuous  (since the same is true of $\Psi^{\diamond 2}$ thanks to \eqref{e:ferm_wick_bound}) the result follows from \eqref{e:def_of_fermionic_square} and the second statement of Lemma~\ref{lem:conv_fermionic_kernels}. \end{proof}

\begin{remark}\label{rem:topology_renormalisation}
We point out that if $\Gamma$ is chosen so that, for every $b \in \Gamma$, $b$ is spanned by functions of sufficiently good regularity, then the seminorms $\| \cdot \|_{n}$ cannot detect the divergence of $\mathbf{C}_{\<2AF>}^{\eps}$ as $\eps \downarrow 0$. 

In particular, one would be able to find $\gamma > 0$ such that  for any $n \in \mathbb{N}$, 
$\mathbf{C}_{\<2AF>}^{\eps}$ converges in $\CC_{\s}^{-\gamma}(\spacetime ; \cA_n)$ as $\eps \downarrow 0$. 
The control over this limit will not be uniform in $n$ \dash however this is still enough for convergence in $\CC_{\s}^{-\gamma}(\spacetime ; \cA)$.
It follows that this choice of $\Gamma$ gives a topology on $\cA$ where renormalisation is not necessary. 

However, we can show that, regardless of the choice of filtration  $\Gamma$, the choice of a good state imposes we correctly renormalise if we want convergence in a reasonable non-commutative $\CL^2$-space. 
If we use the faithful state $\omega_\rho$ instead of the vacuum state, then we can show that the limit of $\mathbf{C}_{\<2AF>}^{\eps}$ is not a $\CL^{2,\infty}(\cA, \bomega_\rho)$-valued distribution, where this space was defined in Remark~\ref{rem:localL2}.

\end{remark}

\subsubsection{Mixed Product}
\label{sec:Mixed_Tree}

The role of \eqref{e:wick_square} and \eqref{e:ferm_wick_bound} will be played 
by \eqref{e:mixed_product} and \eqref{e:mixed_product_continuity} of Proposition~\ref{ProdFieldOpProp}.

We give details for defining $ \<2F>$, the treatment of $ \<2A>$ is identical.
In analogy with the last section, we start by writing

\begin{equ}[e:mixed_square_moll]
	\<1>_\eps (z) \<1F>_\eps (z)
	=
	(\xi \otimes \Psi) \big( \mfZ_{\eps}[ \<2F>](z) \big)
\end{equ}
where $ \mfZ_{\eps}[ \<2F>] \in \cD'\bigl(\spacetime; \mfb \otimes \mfh \bigr)$ is now given by
\begin{equ}
	\mfZ_{\eps}[ \<2F>](z) \eqdef K_{z,\eps} \otimes (G^{\snabla}_{z,\eps} , 0) \;.
\end{equ}
We then have the following lemma.
\begin{lemma}
\label{lem:kernel_convergence}
	Defining
	\begin{equ}
		\mfZ[ \<2F>](z) \eqdef K_{z} \otimes (G^{\snabla}_{z},0)\;,
	\end{equ}
	one has $\mfZ[ \<2F>] \in \CC_{\s}^{-\frac{1}{2}}(\spacetime ;\mfb \otimes \mfh)$ and $\mfZ_{\eps}[ \<2F>] \rightarrow \mfZ[ \<2F>]$ in $\CC_{\s}^{-\frac{1}{2}-\nu}(\spacetime ;\mfb \otimes \mfh)$ as $\eps \downarrow 0$ for any $\nu \in (0,1]$.
\end{lemma}
\begin{proof}
For all $f \in \cD(\R^3)$
	\begin{equs}
		\| \mfZ[\<2F>](f) \|_{\mfb \otimes \mfh}^2 &=  \iint\limits_{\spacetime \times \R^3} \widetilde{(\kappa f)}(z) K^{\star 2}_z(y) G_z^{\snabla, \star 2}(y) \tilde{f}(y)\ \d y \d z \\
		&= \scal{f,\bigl( K^{\star 2} G^{\snabla, \star 2} \bigr) \ast f }_{L^2}\;.
	\end{equs}
	Since $K^{\star 2} G^{\snabla, \star 2} \in \mfN^{-1-}$ the first assertion follows from Proposition~\ref{prop:CCalphaCrit}.
	Similarly, the second assertion follows from the same proposition as well as Lemmata~\ref{lemma:kernel_mult} and \ref{lemma:kernel_conv}. \end{proof}
We then define $\<2F>_{\eps}, \<2F> \in \CC_\s^{-\frac{1}{2}} \bigl(\spacetime ;  L^{\infty-}(\Omega,\mu;\mfA_{\infty}) \bigr)$ by setting
\begin{equs}[e:def_of_mixed_product]
	\<2F>_{\eps}(z) &= \<1>_\eps(z) \<1F>_\eps(z) \quad \text{and} \quad \<2F>(z) \eqdef (\xi \otimes \Psi) \big( \mfZ[ \<2F>](z) \big)\;.
\end{equs}
$\<2F>$ is indeed well-defined by the Lemma~\ref{lem:kernel_convergence} as well as Proposition~\ref{ProdFieldOpProp}.


As described in Remarks~\ref{rem:POVSwitching} and \ref{rem:BosFermProb}, we now switch our point of view and instead consider $\<2F>$ as a measurable map $\Omega\to \cD'(\spacetime ; \mfA_\infty \otimes \C^2)$.
Here we use the notation $\|\bigcdot \|_{\alpha; \K, \mfp}$ described in Appendix~\ref{sec:spaces_of_dist} where we choose the seminorm  $\mfp = \infty$, that is we are using the norm $\| \bigcdot \|_{\infty}$ on $\mfA_{\infty}$.

\begin{proposition}
	For all $\kappa > 0$, $\K \subset \spacetime $ compact and for a.e.\ $\mcb{p} \in \Omega$,
	\begin{equ}
		\| \<2F>(\bigcdot)(p) \|_{-\frac{1}{2}+ \delta - \kappa ;\K, \infty} < \infty
	\end{equ}
	and $		\<2F>_\eps \xrightarrow{\eps \downarrow 0} \<2F>$
	in probability as measurable functions $\Omega \to \CC_\s^{- \frac{1}{2}+ \delta - \kappa} (\spacetime ; \mfA_\infty \otimes \C^2) $.
\end{proposition}

\begin{proof}
	From Proposition~\ref{prop:LqExchange} we find that for all $n \in \N$
	\begin{equ}[eq:BosonicWickProds]
	 	\left\| \|\<2F> \|_{-\frac{1}{2}+\delta - \kappa, \infty} \right\|_{L^q}  \lesssim \left\| \bigl\| \| \<2F> \|_\infty \bigr\|_{L^q} \right\|_{\CC^{-1/2 + \delta}}
	\end{equ}
	as long as $\kappa > \frac{4}{q}$, the claim now follows from Lemma~\ref{lem:kernel_convergence}.
\end{proof}


\subsubsection{Bosonic Products}

We want to define for $\eps > 0$, 
\begin{equ}
	\<2>_\eps  = \<1>^2_\eps - C_{\<2>}^\eps\qquad \text{and} \qquad  \<3>_\eps  = \<1>^3_\eps - 3 C_{\<2>}^\eps \<1> \; ,
\end{equ}
where we write $C_{\<2>}^\eps \eqdef \E[\<1>_\eps(0)^2] = K_{\eps}^{\star 2}(0)$. 
The convergence of $\<2>_\eps$ and $\<1>^3_\eps$ as $\eps \downarrow 0$ is well known, see the following lemma. 

\def\myrefTW{\cite[Theorem~2.1]{TsatsoulisWeber}}
\begin{proposition}[\myrefTW]
	For all $\kappa > 0$, there are random elements  $\<2>,  \<3> \in \CC_\s^{ - \kappa} $ such that one has the convergence in probability 
 $\<2>_\eps \xrightarrow{\eps \downarrow 0} \<2>$ and $\<3>_\eps \xrightarrow{\eps \downarrow 0} \<3>$. 
\end{proposition}

\begin{remark}
	Alternatively and more in line with the previous sections, we could have defined for $\eps \in [0,1]$
	\begin{equ}
		\mfZ_\eps [\<2>] (z) \eqdef K_z \otimes K_z \qquad \text{and} \qquad \mfZ_\eps [\<3>](z) \eqdef K_z \otimes K_z \otimes K_z
	\end{equ}
	and set
	\begin{equ}
		\<2>_\eps = \xi^{\diamond 2}(\mfZ_\eps[\<2>]) \qquad \text{and} \qquad \<3>_\eps = \xi^{\diamond 3}(\mfZ_\eps[\<3>]) \; .
	\end{equ}
\end{remark}

\section{Local Solution Theory}
Now that we have set up our extended CAR algebra and proven the needed convergence results for the stochastic data, we turn to putting this all together along with results on Schauder theory and multiplication on H\"{o}lder--Besov spaces (see Appendix~\ref{sec:spaces_of_dist}) to solve our SPDE. 


\subsection{Solving the Equation in \TitleEquation{\cA\F_n}{cAFn}}

We start by proving the existence of a solution to (a renormalised version of) the equations \eqref{eq:lin_sig_model_DPD1} from Section~\ref{sec:PDTrick} for short enough times and in the ``localised'' Banach space $\cA_n$ for any fixed $n \in \N$. 
We fix sufficiently small $\delta > \kappa > 0$ and solve the SPDE for generic choices of the driving noises
\begin{equs}[2]\label{e:driving_noises}
	\<1> & \in \CC_\s^{-\kappa,-\kappa}(\spacetime_1) \; , & \quad \<2> & \in \CC_\s^{-2\kappa,-2\kappa}(\spacetime_1) \; , \\
	\<3> & \in \CC_\s^{-3\kappa,-3\kappa}(\spacetime_1) \; , & \quad \<1F> & \in \CC_\s^{-\frac{1}{2}+\delta,-\frac{1}{2}+\delta}(\spacetime_1; \cA_n^1 \otimes \C^2) \; , \\
	 \<1A> & \in \CC_\s^{-\frac{1}{2}+\delta,-\frac{1}{2}+\delta}(\spacetime_1; \cA_n^1 \otimes \C^2) \; , & \quad
	\<2F> & \in \CC_\s^{-\frac{1}{2}+\delta-\kappa, -\frac{1}{2}+\delta-\kappa}(\spacetime_1; \cA_n^1 \otimes \C^2) \; ,\\
	 \<2A> & \in \CC_\s^{-\frac{1}{2}+\delta-\kappa,-\frac{1}{2}+\delta-\kappa}(\spacetime_1; \cA_n^1 \otimes \C^2) \; , & \quad \<2AF> & \in \CC_\s^{-1 +2\delta,-1 +2\delta}(\spacetime_1; \cA_n^0) \;.
\end{equs}
The H\"older spaces $\CC^{\alpha,\eta}_\s$ of functions with singularities at time $0$ are defined below in Appendix~\ref{sec:Holder_Sing}.

The Banach space of driving noises, which is just the Cartesian product of all the Banach spaces\footnote{	Since $\spacetime_1$ is compact, all of the above spaces should be viewed as Banach spaces with a single norm where $\K = \spacetime_1$.} appearing in \eqref{e:driving_noises}, is denoted by $\CN_n$.
Note that for a generic choice of driving noises, that is an arbitrary element $\Xi = (\<1> , \<2>, \<3>, \<1F>, \<1A> , \<2F> ,  \<2A> ,  \<2AF>) \in \CN_n$ \dash there is a-priori there no relationship between the different component noises and $\Xi$ need not be built from renormalised polynomials of the bosonic and fermionic noises. The reader should is cautioned that we write $\Xi$ instead of $\Xi_n$ to lighten the notation. Similarly, we have dropped $n$ from the notation for all components of $\Xi$, $u_0$, and its components, as well as the solutions $U$, $Y$, and $\overline{Y}$.

The aim of this section is to define a continuous map mapping driving noises $\Xi \in \CN_{n}$ to the local solution to the SPDE driven by $\Xi$.

\begin{remark}
	From the continuity of the canonical projection $\pi_n \colon \cA\F \to \cA_n$ it follows that $\pi_n\circ f \in \CC^{\alpha,\eta}_{\s}(\spacetime_1 ; \cA_n)$ if $f \in \CC^{\alpha,\eta}_{\s}(\spacetime_1; \cA\F)$, etc.
\end{remark}

\begin{remark}
	The choice of space $\CC_\s^{\alpha,\alpha}$, rather than $\CC^{\alpha}_\s$, was made because we need to multiply the driving noises constructed in Section~\ref{sec:Renormalisation} by the temporal cut-off $\bone_{\{t \geqslant 0\}} \in \CC_{\s}^{\infty,0}(\spacetime)$. Their product then belongs to $\CC^{\alpha,\alpha}_{\s}$ by Theorem~\ref{YngMultBUThm1}. 
\end{remark}

\begin{theorem}
\label{thm:LocSolution}
	Let $\Xi \in \CN_n$. Then for any $\kappa \in (0,\delta)$ and $\varsigma \in (0,\kappa)$  and any initial condition
	\begin{equ}
		u_0  = (U_0, Y_0, \overline{Y}_0) \in \reminit_n \eqdef \CC^{-\kappa+\varsigma}(\T^2; \cA_n^0) \times \CC^{\frac{1}{2} + \delta-\kappa+\varsigma}(\T^2; \cA_{n}^1 \otimes \C^{4} )
	\end{equ}
	the equations
	\begin{equs}\label{eq:lin_sig_model_DPD}
		\partial_t U &= (\Delta-m^2) U - g \Big(  \bilin{(-\overline{\snabla}+M)\overline{Y} , (\snabla+M)Y}  + \\
		& \qquad +  \bilin{(-\overline{\snabla}+M)\overline{Y} ,\<1F> } + \bilin{\<1A>, (\snabla+M)Y} + \<2AF>  \Big) - \\
		& \qquad - \lambda \Bigl( U^3 + 3 U^2 \<1> + 3 U \<2> + \<3> \Bigr)\; ,	\\
		\partial_t Y &=  (\Delta-M^2) Y - g \Big(  U \big( \snabla+M \big)  Y  + \<1> (\snabla + M) Y + U \<1F>  + \<2F> \Big) \; ,\\
		\partial_t \overline{Y} &=  (\Delta-M^2) \overline{Y} - g \Big(  U \big( - \overline{\snabla}+M \big)  \overline{Y}  + \<1> (-\overline{\snabla} + M) \overline{Y} + U \<1A>  + \<2A> \Big) \; ,
	\end{equs}
	have a solution in
	\begin{equ}\label{eq:local_local_solspace}
		\CC_{T,n} \eqdef \CC_\s^{1+2\delta,-\kappa}(\spacetime_T; \cA_n^0) \times \CC_\s^{\frac{3}{2}+\delta-\kappa, \frac{1}{2}+\delta-\kappa}(\spacetime_T ; \cA_n^1 \otimes \C^4)
	\end{equ}
	for $0 < T = T(n, u_0 , \Xi)\leqslant 1$ small enough.

	This solution is unique on the temporal interval $[0,T]$ and for $\delta_1, \delta_2 > 0$ small enough the solution map is jointly continuous on the balls of radii $\delta_1$ and $\delta_2$ around $u_0$ and $\Xi$ respectively
	\begin{equs}
		\CS^T_n \colon \reminit_n \times \CN_n \supset B_{\delta_1}(u_0) \times B_{\delta_2}(\Xi) & \longrightarrow \CC_{T,n}\\
		(u_0', \Xi') & \longmapsto u'
	\end{equs}
\end{theorem}
\begin{remark}
	The PDE above are ``remainder PDE'' governing the remainder of the solution to the full PDE after subtracting the linear solution \dash see Section~\ref{sec:PDTrick}.
The $\eta$-regularities in \eqref{eq:local_local_solspace}, that is the second exponent in the superscripts on the right side of \eqref{eq:local_local_solspace}, are chosen so that we can accommodate $(\<1>(0), \<1IF>(0), \<1IA>(0) )$ as an initial condition in the remainder equation. 
This means that we can start the full system of PDEs with arbitrary initial data of the same regularity as the linear solution.  
\end{remark}

\begin{remark}
	We note here that the dependence of $T(n,u_0,\Xi)$ only comes from the ambient space $\CC_{T,n}$ within which we are solving the equation. Later on when the initial conditions and noises are chosen as the different projections of single $\cA$ valued distributions $u_0$ and $\Xi$ the dependence will become more direct as $n$ determines the norms of the initial conditions and driving noises via $\pi_n$. 
\end{remark}

\begin{remark}
	In the above theorem one can replace $\cA_n$ by the closure of $\pi_n(\cG\F)$ in $\cA_n$ as per Remark~\ref{remark:Grassmann_Preserved}.
\end{remark}

In what follows, we will denote the norm of $\CC_{T,n}$ by $\VERT \bigcdot \VERT_{T,n}$.

\begin{proof}
	We denote the non-linear part of the right-hand side of \eqref{eq:lin_sig_model_DPD} by $F(u, \Xi)$ for $u = (U,Y,\overline{Y})$ and $\Xi = (\<1>, \<2>, \<3>, \<1F>, \<1A>, \<2F>, \<2A>, \<2AF>)$.

	We also use the abbreviations $Y^{\snabla} \eqdef (\snabla+M) Y$ and $\overline Y^{\snabla} \eqdef (-\overline{\snabla} +M) \overline Y$. Furthermore, we suppress the equation for the ``anti-fermions'' from now on as it is practically the same as the equation for the fermions. We also leave out a detailed treatment of the bosonic self-interaction as the analysis there proceeds by the usual arguments and the regularity of those terms is strictly better than the regularity of the terms containing fermions.

	Integrating \eqref{eq:lin_sig_model_DPD} with respect to $\CI_B$ and $\CI_F$ respectively we arrive at the equation
	\begin{equs}[FxdPntEq]
		U &= G_B U_0 - g \CI_B \Big( \overbrace{\overline{Y}^{\snabla} \bigcdot Y^{\snabla}}^{ \substack{ \frac{1}{2}+\delta - \kappa \\ -1+2\delta-2\kappa } } + \overbrace{\overline{Y}^{\snabla} \bigcdot \<1F> }^{\substack{- \frac{1}{2}+\delta\\ -1 + 2 \delta - \kappa}} + \overbrace{\<1A> \bigcdot Y^{\snabla} }^{\substack{- \frac{1}{2}+\delta\\ -1 + 2 \delta - \kappa}} + \overbrace{ \<2AF> }^{\substack{-1 +2 \delta\\-1 +2 \delta}} \Big) -  \\
		& \qquad \qquad - \lambda \CI_B \Bigl( \overbrace{U^3}^{\substack{1+2\delta \\ - 3 \kappa}} +3 \overbrace{U^2 \<1>}^{\substack{-\kappa \\ - 3 \kappa}}+3 \overbrace{U \<2>}^{\substack{-2\kappa \\ -  \kappa}} + \overbrace{\<3>}^{\substack{-3\kappa \\ - 3 \kappa}} \Bigr)\\
		Y &= G\F Y_0 -g \CI_F \Big( \overbrace{U Y^{\snabla}}^{\substack{\frac{1}{2}+\delta-\kappa\\ -\frac{1}{2}+\delta -2\kappa}} + \overbrace{\<1> Y^{\snabla}}^{\substack{-\kappa \\ - \frac{1}{2}+\delta-2\kappa}} + \overbrace{U\<1F>}^{\substack{-\frac{1}{2}+\delta\\-\frac{1}{2}+\delta- \kappa }} + \overbrace{\<2F>}^{\substack{-\frac{1}{2} + \delta-\kappa\\-\frac{1}{2} + \delta-\kappa}} \Big)
	\end{equs}
	where we have written $\bigcdot$ as shorthand for the bilinear map $\bilin$. 
	We have noted the $\alpha$ and $\eta$ regularity of the products as $\begin{smallmatrix} \alpha\\ \eta \end{smallmatrix}$. 
	All of the products are well-defined and continuous maps by Theorem~\ref{YngMultBUThm1} and Theorem~\ref{YngMultBUThm2}.

	 We denote the right-hand side of \eqref{FxdPntEq} by $\CM_n^\Xi(u)$ and the terms that are being integrated by $\CF(u)$. We begin by showing that $\CM_n^\Xi$ maps $\CC_{T,n}$ into itself for any $T \leqslant 1$. 
	 First note that $G_BU_0 \in \CC_\s^{\gamma, -\kappa+\varsigma} \hookrightarrow \CC_\s^{\gamma, -\kappa}$ for any $\gamma > 0$ because of Proposition~\ref{HKSmootProp} and because the inclusion $\CC_\s^{\gamma, -\kappa+\sigma} \hookrightarrow \CC_\s^{\gamma, -\kappa}$ is continuous.  The analogous statements hold for the other terms depending on the initial condition which we abbreviate as $Gu_0$. The latter also  satisfies uniformly in $T \leqslant 1$ the estimate
	\begin{equ}
		\VERT Gu_0 \VERT_{T,n} \lesssim \|u_0\|_{\reminit_n}
	\end{equ}
	where $\|u_0\|_{\reminit_n}$ is the norm of $u_0$ in its space. 
	Next note that taking spatial derivatives is a continuous map $\CC^{\alpha, \eta}_\s \to \CC_\s^{\alpha-1,\eta-1}$, cf.\ Theorem~\ref{DerivativeBU} By the theorems on the continuity of multiplication, Theorem~\ref{YngMultBUThm1} and Theorem~\ref{YngMultBUThm2}, the map
	\begin{equ}
		\big(
			U , Y , \overline{Y}
\big) \longmapsto \big(
			\overline{Y}^{\snabla} Y^{\snabla}, U Y^{\snabla}, U \overline{Y}^{\snabla}
		\big)
	\end{equ}
	is Lipschitz continuous with constant $\lesssim R$ on the balls of radius $R$ in $\CC_{T,n}$ (independently of $T$), while $U \mapsto U^3$ has constant $\lesssim R^2$. Similarly, the multiplication with the noises is also Lipschitz continuous, and, therefore, $\CF$ as a whole is Lipschitz continuous on the ball of radius $R$. Since both $\CI_F$ and $\CI_B$ are $2$-regularising it follows that $( \CI_B, \CI_F, \CI_F ) \circ \CF$  maps the original space first into $\CC^{\beta,\eta}$, where $\beta,\eta$ are $1+2\delta, 1+2\delta-2\kappa$ and $\frac{3}{2}+\delta-\kappa, \frac{3}{2}+\delta-2\kappa$ for the bosonic and fermionic equation respectively, and then we use the embedding
	\begin{equ}
		\CC^{\beta,\eta} \hooklongrightarrow \CC^{\beta,\eta-1-\sigma}
	\end{equ}
	to return to the original spaces. Here we have set $\sigma = 2\delta-\kappa$ for the bosons and $\sigma = -\kappa$ for the fermions respectively. 
	In particular, this embedding is bounded such that
	\begin{equ}
		\VERT \bigcdot \VERT_{\beta,\eta-1-\sigma; \spacetime_T,\cA_n} \lesssim T^{\frac{1+\sigma}{2}} \VERT \bigcdot \VERT_{\beta,\eta; \spacetime_T,\cA_n} \; ,
	\end{equ}
	cf.\ Proposition~\ref{prop:TimeImprov}. We have thus shown that $\CM_n^\Xi \colon \CC_{T,n} \to \CC_{T,n}$ and that for some constant $C>0$, and for any two $u,v \in \CC_{T,n}$ with $\VERT u \VERT_{T,n}, \VERT v \VERT_{T,n} \leqslant R$
	\begin{equs}
		\VERT \CM_n^\Xi(u) \VERT_{T,n} &\leqslant \VERT Gu_0 \VERT_{T,n} + C T^{\frac{1+\sigma}{2}} R \VERT u \VERT_{T,n} \; , \\
		\VERT \CM_n^\Xi(u)- \CM^\Xi(v) \VERT_{T,n} &\leqslant C T^{\frac{1+\sigma}{2}} R \VERT u- v \VERT_{T,n} \; .
	\end{equs}
	Choosing $R = \VERT Gu_0 \VERT_{T,n}+1$, for $T$ small enough, $\CM_n^\Xi$ therefore maps the ball of radius $R$ around $0$ into itself as a contraction \dash resulting in a map with a fixed point. Furthermore, since $\CF$ is also Lipschitz continuous with respect to our choice of noise $\Xi$ it follows that
	\begin{equ}
		\VERT \CM_n^\Xi(u)- \CM_n^{\Xi'}(u) \VERT_{T,n} \lesssim \| \Xi - \Xi'\|_{\CN_n} \; .
	\end{equ}
	If $\| \Xi - \Xi'\|_{\CN_n}$ is chosen small enough, then one can choose the same $T$ for $\CM_n^\Xi$ and $\CM_n^{\Xi'}$. Thus, the fixed-points are continuous with respect to the choice of (localised) driving noises in a small neighbourhood around a given set of noises $\Xi$.
\end{proof}
\begin{remark}
	For given $R, S > 0$ and $u_0 \in B_{R}(0)$, $\Xi \in B_S(0)$ the choice of $T$ only depends on $R$ and $S$, and not the individual $u_0$ and $\Xi$.
\end{remark}

The next step is to view the solution up to time $T$ as an element of $\cC\bigl( [0,T] ; \treminit_{n} \bigr) \cap \cC\bigl( (0,T] ; \reminit_{n} \bigr) $, where 
\begin{equ}
	\treminit_{n} \eqdef  \CC^{-\kappa}(\T^2; \cA_n^0) \times \CC^{\frac{1}{2} + \delta-\kappa}(\T^2; \cA_{n}^1 \otimes \C^{4} )
\end{equ}
which is of slightly worse regularity than $\reminit_n$. For the term depending on the initial condition, this follows directly from Proposition~\ref{HKSmootProp}. For the non-linearity we argue as follows.

For any $T \in (0,1]$, $t \in (0,T]$, and $\alpha, \eta, \theta > 0$ we have the sequence of inclusions
\begin{equ}
	 \CC_\s^{\alpha,\eta}(\spacetime_T; \cA_n) \hookrightarrow \CC_\s^{\eta,\eta}(\spacetime_T; \cA_n) = \CC_\s^{\eta}(\spacetime_T; \cA_n) \hookrightarrow \cC\bigl([0,T] ; \CC^{\eta-\theta}(\T^2 ; \cA_n)\bigr)
\end{equ}
with
\begin{equ}
	\sup_{t \in [0,T]}\| u(t, \bigcdot) \|_{\alpha; \cA_n } \lesssim \VERT u \VERT_{\alpha,\eta; \spacetime_T,\cA_n} \; ,
\end{equ}
where the constant of proportionality is independent of $n \in \N$. 
The contribution from the non-linearity is actually an element of $\CC_\s^{1+2\delta, 1+2\delta-2\kappa} \times \CC_\s^{\frac{3}{2}+\delta-\kappa , \frac{3}{2}+\delta-2\kappa}$, in particular both $\eta$-regularities are positive and we see that it is in $\cC\bigl( [0,T] ; \treminit_{n} \bigr) \cap \cC\bigl( (0,T] ; \reminit_{n} \bigr)$. 
\begin{corollary}
	In the setting of Theorem~\ref{thm:LocSolution}, the solution map $\CS_n^T$ can also be viewed as a map 
	\begin{equ}
		\CS^T_n \colon B_{\delta_1}(u_0) \times B_{\delta_2}(\Xi) \longrightarrow \cC([0,T]; \treminit) \; .
	\end{equ}
\end{corollary}


We can now evaluate the solution $u$ at time $T<1$ giving us an element in $\reminit_n$ and allowing us to restart the equation at that time with noises shifted temporally by $T$ (which does not change their norms) and new initial condition $u(T, \bigcdot)$. By the usual arguments we construct in this way a maximal solution on an interval $[0,1 \wedge T(n,u_0, \Xi)]$ where either $T(n,u_0, \Xi) = \infty$ or $T(n,u_0, \Xi) \in (0,1]$ and 
\begin{equ}
	\lim_{t \uparrow T(n,u_0, \Xi)} \left\| u(t, \bigcdot ) \right\|_{\reminit_n} = \infty \; .
\end{equ}
For $L > 0 $ define
\begin{equ}
	T^L(n,u_0,\Xi) \eqdef  \inf\left\{ t \in [0,1] \, \middle| \, \|  u (t,\bigcdot)\|_{\reminit_n}  \geqslant L \right\} \; .
\end{equ}
For $(u_0, \Xi) \in \reminit_n \times \CN_n$ we let $\CS^L_n(u_0, \Xi)$ denote the solution \eqref{eq:lin_sig_model_DPD} up until $T^L(n,u_0,\Xi)$.

\begin{corollary}
\label{cor:Local_Cont_Prop}
	Let $L> 0$ be fixed, and let $T^L$ and $\CS^L_n$ be as above. Then for every $\eps > 0$ and $C>0$ there exist $\delta_1, \delta_2 > 0$, such that for $T = 1 \wedge T^L(n,u_0, \Xi) \wedge T^L(n,u_0', \Xi') $ one has the bound
	\begin{equ}
		\VERT \CS_n^L(u_0, \Xi) - \CS_n^L(u_0', \Xi') \VERT_{\CC_{T,n}}	\leqslant \eps
	\end{equ}
	for all $u_0, u_0' \in \reminit_n$, and $\Xi, \Xi' \in \CN_n$ such that $\|\Xi\|_{\CN_n}, \|\Xi'\|_{\CN_n} \leqslant C$, $\|\Xi-\Xi'\|_{\CN_n}\leqslant \delta_2$,
	$\|u_0\|_{\reminit_n}, \|u_0'\|_{\reminit_n} \leqslant \frac{L}{2}$, and $\|u_0-u_0'\|_{\reminit_n} \leqslant \delta_1$.
\end{corollary}

\begin{proof}
	This follows by the same arguments as in Corollary 7.12 in \cite{Hai14}. The argument relies on the fact that, knowing $C$ and that $\|u(t,\bigcdot)\|_{\reminit_n} \leqslant L$, we  can find a lower bound for the interval $\Delta t$ in the contraction argument, which is independent of the explicit noise realisation and initial condition.
\end{proof}

If a solution $u$ blows up at some finite time $T(u)$, then we can add an element $\infty$ to the space of solutions. In particular, let $\CC_n^{\sol} $ be the extension of the space $\cC\bigl([0,1]; \treminit_n \bigr)$, defined in Appendix~\ref{sec:blowup_space}, to include trajectories which blow up in finite time. In this space we can define for arbitrary $u_0 \in \reminit_n$ and $\Xi \in \CN_n$ the solution map
\begin{equ}
	\CS_n \colon \reminit_n \times \CN_n \longrightarrow \CC^\sol_n
\end{equ}
where $\CS_n(u_0, \Xi)\restr {[0,T]} \eqdef \CS^T_n(u_0,\Xi)$ for any $T \in \bigl[0, T(n,u_0,\Xi) \bigr) \cap [0,1]$, while for $t \in \left( T(n, u_0,\Xi),1 \right] $ we set $\CS_n(u_0,\Xi) = \infty$. We have by the above Corollary~\ref{cor:Local_Cont_Prop} and Lemma~\ref{lemma:Csol_cont} the following. 
\begin{proposition}
	For all $n \in \N$, the map $\CS_{n} \colon \reminit_n \times \CN_n \longrightarrow \CC^\sol_n$ is well-defined and continuous.
\end{proposition}

\subsection{Solving the Equation in \TitleEquation{\cA\F}{cAF}}

We now wish to do our best to patch together the solution maps $(\CS_n)_{n \in \N}$ into a single solution map giving solutions taking value in $\cA\F$. 

Recalling that in Lemma~\ref{lem:proj_limit} we showed that $\cA\F$ is the projective limit of the spaces $(\cA_{n})_{n \in \N}$, we define the following projective space: 
\begin{equ}
	\CC^\sol_{\cA} \eqdef \left\{ (u_n)_n \in \prod_{n \in \N} \CC^\sol_n \, \middle| \,
\begin{array}{c}
 \forall n, m \in \N \text{ with } n \leqslant m, \\ 
 \pi_{nm} \circ u_m \restr{[0,T_m)} = u_n \restr{[0,T_m)} 
 \end{array}
\right\}
\end{equ}
where $\pi_{nm} \colon \cA_m \to \cA_n$ is the canonical projection, and 
\begin{equ}
	T_m \eqdef \inf\left\{ t \in [0,1] \, \big| \, u_m(t) = \infty \right\}\;.
\end{equ}
The reason for restricting the equality of $u_{m}$ and $u_{n}$ to the interval $[0,T_m)$ is that, although $\lim_{t \uparrow T_m} u_m(t) = \infty$ the same might not be true for $u_n(t)$ \dash we have $T_{n} \geqslant T_{m}$ and cannot rule out a strict inequality. 
Since the maps $\pi_{mn}$ are $C^*$-algebra morphisms, it follows from continuity that 
\begin{equ}
	\pi_{nm} \Bigl(  \CS_m( \pi_m \circ u_0, \pi_m \circ \Xi ) \restr{[0,T_m)} \Bigr) =   \CS_n( \pi_n \circ u_0, \pi_n \circ \Xi ) \restr{[0,T_m)} \; .
\end{equ}
Furthermore, let
\begin{equ}
	\reminit_{\cA} \eqdef \CC^{-\kappa} \left( \T^2 ; \cA^0  \right) \times \CC^{\frac{1}{2}+\delta-\kappa} \left(\T^2 ;  \cA^1 \otimes \C^4\right) \;
\end{equ} 
and analogously define the space $\CN_{\cA}$ by replacing, for $i \in \{1,2\}$, all instances of $\cA^{i}_{n}$ in \eqref{thm:LocSolution} with $\cA^{i}$.

We can then write down the desired solution map for $\cA$-valued solutions as
\begin{equ}
\begin{aligned}
	\CS_{\cA} \colon  \reminit_{\cA} \times \CN_{\cA}    & \longrightarrow \CC^\sol_{\cA} \\
	(u_0,\Xi) & \longmapsto \bigl(   \CS_n \left( \pi_n \circ u_0,  \pi_n \circ \Xi \right) \bigr)_{n \in \N}  
\end{aligned}\; .
\end{equ}

\begin{remark}
	\label{rem:NotGlobSol}
	Note that here we fall short of honestly building local solutions in $\cA$, in particular we cannot rule out that $T_{\cA} \eqdef \inf_{n \in \N} T_n  = 0$. If $T_{\cA} > 0$ however, then we can show that iterates of the fixed-point map converges in the extended CAR algebra $\cA$ on the temporal interval $[0,T_{\cA})$, and the solution makes sense as an element $\cC\bigl( [0,T_{\cA}) ; \reminit_{\cA} \bigr)$.
\end{remark}

\begin{theorem}\label{thm:precise_main_thm}
	The map $\CS_{\cA}$ is well-defined and continuous. In particular there exist $Z(\cA\F)$-valued functions $\big(C_{\eps}\big)_{ \eps \in (0,1] }$, independent of $g$ such that the solutions $u_\eps \eqdef \big(\phi_{\eps},( \upsilon_{\eps},\bar{\upsilon}_{\eps}) \big) \in \CC_\cA^\sol$ to the system of equations \eqref{eq:Yukawa_eps}
	with initial conditions $u_0 \eqdef (\phi_0, \upsilon_0, \bar\upsilon_0) \in \reminit_{\cA}$, converge in probability in
$		\CM\bigl(\Omega, \mu;  \CC^\sol_{\cA} \bigr)$
	as $\eps \downarrow 0$ to
	\begin{equ}
		\CS_{\cA}\left(u_0 -\big( \<1>, \<1IF>, \<1IA> \big)(0) , \Xi\right) + \big( \<1>, \<1IF>, \<1IA> \big) \; ,
	\end{equ}
	where
	\begin{equ}
		\Xi =  (\ \<1>, \<1F>, \<1A>, \<2F>, \<2A>, \<2AF>\ )
	\end{equ}
	are the $\xi$-measurable, $\cA\F$-valued, space-time distributions constructed in Section~\ref{sec:Renormalisation}.
\end{theorem}
\begin{remark}
	Before beginning with the proof, we note that one arrives at the regularised version of \eqref{eq:lin_sig_model_DPD} from the regularised version of \eqref{eq:lin_sig_model_DPD1} by subtracting the central element $C^\eps_{\<2AF>}$ as well as $3 \lambda C^\eps_{\<2>} u$, defined in \eqref{eq:def_RenormElem} and \eqref{eq:BosonicWickProds}, from the bosonic equation as
	\begin{equ}
		\bilin{\<1A>, \<1F>} = \<2AF> + C^\eps_{\<2AF>} \; \quad \text{and} \quad 
		u^3 - 3 C^\eps_{\<2>} u = U^3 + 3 U^2 \<1> + 3 U \<2> + \<3> \; ,
	\end{equ}
	where $u = U + \<1>$. Thus, we are indeed solving the renormalised equation \eqref{eq:Yukawa_eps} instead of the original \eqref{eq:lin_sig_model}.
\end{remark}
\begin{proof}
This follows almost immediately from the previous results.
The final ingredients necessary are the regularities of the noises used in the Da Prato--Debussche decomposition. From Proposition~\ref{prop:bosonic_Noise_Properties} and Proposition~\ref{prop:fermionic_Noise_Properties} we have that $\big( \<1>, \<1IF>, \<1IA> \big) \in \mathscr C(\R ; \treminit_{\cA} )$ and it is straighforward to use the $\eps \downarrow 0$ convergence statement in each of the three lemmas of Section~\ref{subsec:noise_reg} to show that $\big( \<1>_\eps, \<1IF>_{\eps}, \<1IA>_{\eps} \big)$ converges to $\big( \<1>, \<1IF>, \<1IA>\big)$ in $\CC^\sol_{\cA}$, this gives us the result.
\end{proof}

\begin{remark}
	The above argument can be easily repeated to extend the solutions beyond time $1$ by suitably shifting the noises in time and modifying the integration kernels correspondingly.
\end{remark}

\subsection{ \TitleEquation{\mfA_\infty}{fA}-Valued and \TitleEquation{\CA\F}{AF}-Valued Solutions}\label{subsec:original_eqn}

Since $\<1F>_\eps,\<1A>_\eps$ are smooth $\mfA_\infty \otimes \C^2$-valued functions, it follows that 
\begin{equ}
	\<2AF>_\eps \in \cC^\infty(\spacetime; \mfA_\infty) \subset \CC^{-1+2\delta}_\s(\spacetime; \mfA_\infty) \; ,
\end{equ}
etc.\ for the other mollified noises. In particular, for $\eps > 0$, we can solve \eqref{eq:lin_sig_model_DPD} in an analogously defined space $\CC^\sol_{\mfA}$, in the same way we solved the equation with values in the spaces $\cA_n$. We call this solution map $\CS_\mfA$. The same is true for $\CA\F$, and we call the corresponding solution map $\CS_{\CA}$ with solution space $\CC^\sol_{\CA}$.

As the maps $\pi_n \restr{\mfA_\infty}$ and $\digamma$ are $C^*$-algebra homomorphisms from $\mfA_\infty$ to $\cA_n$ and $\CA\F$ respectively the same compatibility condition between solutions with values in $\mfA_\infty$ and values in $\cA_n$ or $\CA$ holds. Again the blow-up time in $\mfA_\infty$ might be strictly smaller than  in $\cA_n$ and $\CA$.

We summarise this interrelationship between the different solutions in the following commutative diagram, where we only note the dependence on the mollified noise input denoted by $\CN_\eps$.
\begin{equ}
\tikzexternaldisable
	\begin{tikzcd}
		\CN_\eps \arrow[rd,"\CS_{\CA}"] \arrow[d, "\CS_{\mfA}"] \arrow[dd, bend right = 50, "\CS_n"'] \arrow[rdd, bend left = 80, "\CS_{\cA}"]  \\
		\CC^\sol_{\mfA} \arrow[r, "\digamma"] \arrow[d, "\pi_n"] \arrow[rd, hookrightarrow, "\pi"] & \CC^\sol_{\CA} \\
		\CC^\sol_{n} & \CC^\sol_{\cA} \arrow[l] \arrow[r, dashrightarrow] & \cC([0,1]; \treminit_{\cA})
		\end{tikzcd}
\tikzexternalenable
\end{equ}
If one is able to prove that $T(n, \pi_n \circ u_0, \pi_n \circ \Xi_\eps) = \infty$ for all $n \in \N$ and $\eps > 0$, then the solution will converge in the closure of the image of $\cC([0,1]; \treminit_{\mfA} )$ in $\prod_{n \in \N} \cC([0,1]; \treminit_{\cA_n})$ and thus also in $\cC([0,1]; \treminit_{\cA})$ which we have denoted here by a dotted arrow. 

This commutative diagram tells us that when $\eps > 0$, the solutions that we can construct in $\cA_{n}$, or $\CA$, and $\mfA$ are consistent with each other \dash in particular they are indistinguishable at the level of correlation functions.

%
%

\appendix
\section{Spaces of Distributions}\label{sec:spaces_of_dist}

\subsection{LCTVS-Valued H\"older--Besov Spaces}
\label{sec:Holder_Space}

\begin{definition}
	Let $(E,\mfP)$ be a locally convex topological vector space (LCTVS) with collection of seminorms $\mfP$. We denote by $\cD'(\R^d ; E)$ the set continuous linear maps $\cD(\R^d) \to E$ equipped with the following topology. A net $(\xi_\alpha)_{\alpha} \subset \cD'(\R^d ; E)$ converges to $\xi$ if and only if for every bounded set $A \subset \cD(\R^d)$ and every seminorm $\mfp \in \mfP$
	\begin{equ}
		\lim_{\alpha} \sup_{\phi \in A} \mfp\bigl( \xi_\alpha(\phi) - \xi(\phi) \bigr)  = 0 \; .
	\end{equ}
	A subset $A \subset \cD(\R^d)$ is said to be bounded if and only if there exists some compact $\K \subset \R^d$, s.t.\ $A \subset \cD(\K)$, and, for every seminorm $\| \bigcdot \|_{n, \K}$ of $\cD(\K)$, $\|A\|_{n,\K}$ is a bounded subset of $[0,\infty)$.
\end{definition}

In the following we denote by $B_{\s}(x,r) \eqdef \left\{y \in \R^d \, \big| \, \|x-y\|_{\s} \leqslant r \right\}$.

\begin{definition}[Test Functions $\CB^r_{\s,x}$]
	Let $r \in \N$, $\s$ a scaling on $\R^d$, and $x \in \R^d$. Define
	\begin{equ}
		\CB^r_{\s,x} \eqdef \left\{ \eta \in \cD(\R^d) \, \big| \, \supp \eta \subset B_{\s}(x,1), \; \|\eta\|_{\cC^r} \leqslant 1 \right\}	\; ,
	\end{equ}
	where $\|\bigcdot\|_{\cC^r}$ denotes the usual $\cC^r$-norm.
	Furthermore, for $\alpha \in \R$, let $\CB^{r,\alpha}_{\s, x}$ be the subset of $\CB^r_{\s, x}$ such that for all polynomials $P$ of degree $\deg P \leqslant \alpha$ and all $\phi \in \CB^{r,\alpha}_{\s,x}$
	\begin{equ}
			\int \phi(x) P(x) \d x	= 0 \; .
	\end{equ}
\end{definition}


We have the following definition of H\"older--Besov Spaces on open domains.

\begin{definition}
\label{def:OpnHolderSpace}
	Let $(E,\mfP)$ be a locally convex topological vector space (LCTVS) with a collection of seminorms $\mfP$. Let $U \subset \R^d$ be an open set. We define for $\xi \in \cD' \left( U ;  E \right) $, $\K \subset U$ compact, the seminorm
	\begin{equ}\label{eq:besov_norm}
		\|\xi\|_{\alpha;U,\K,\mfp} \eqdef \sup_{x \in \K} \sup_{\eta \in \CB^{r,\alpha}_{\s,0}} \sup_{\lambda \in (0,\lambda_x]} \lambda^{-\alpha} \mfp \left( \xi \left( \CS^\lambda_{\s,x} \eta \right) \right) +  \sup_{x \in \K} \sup_{\eta \in \CB^{r}_{\s,0}} \mfp \left( \xi \left( \CS^{\lambda_x}_{\s,x} \eta \right) \right) \;   ,
	\end{equ}
	where $\lambda_x \eqdef 1 \wedge \frac{1}{2} \dists(x, \R^d \setminus U)$ with 
	\begin{equ}
		\dists(x, \R^d \setminus U) \eqdef \inf_{y \in \R^d \setminus U} \|x-y\|_\s\;.
	\end{equ}
	Then the space of $E$-valued $\alpha$-H\"older functions is given by
	\begin{equ}
		\CC^\alpha_{\s} \big( U; E\big) \eqdef \left\{ \xi \in  \cD' \left( U ; E \right)  \, \middle| \, \forall \mfp \in \mfP \, \forall \K \subset U \text{ compact }: \|\xi\|_{\alpha; U, \K,\mfp}  < \infty \right\}
	\end{equ}
	which is a complete subspace of $\cD'(U;E)$. If the index $\K$ is dropped, the supremum is taken over all of $U$. When no confusion may arise we often drop the index $U$.

\end{definition}

We can also naturally define H\"older--Besov spaces over closed domains.

\begin{definition}
\label{def:ClsdHolderSpace}
	 In the above setting, let $A \subset \R^d$ be a closed subset, and  $\CU$ the system of open neighbourhoods $U$ of $A$, where in particular $\bigcap \CU = A$, and let $\alpha \in \R$. We set $\CC^{\alpha}_{\s}(A; E)$ to be the space of germs of distributions defined on open sets containing $A$. 

	That is, for $A \subset U, V$ open we say that $f \in \CC^{\alpha}_{\s}(U;E)$ and $g \in \CC^{\alpha}_{\s}(V;E)$ are equivalent if and only if $f(\phi) = g(\phi)$ for all test functions $\phi$ compactly supported in $U \cap V$. Together with the natural restriction $\CC^{\alpha}_{\s}(V;E) \to \CC^{\alpha}_{\s}(U;E)$ for $A \subset U \subset V$, we can thus define the direct limit
	\begin{equ}
		\CC^{\alpha}_{\s}(A;E) \eqdef \varinjlim_{U \supset A} \CC^{\alpha}_{\s}(U;E) \; ,
	\end{equ}
	with canonical continuous restrictions $\pi_U \colon \CC^{\alpha}_{\s}(U;E) \to \CC^{\alpha}_{\s}(A;E)$. 

	This space is topologised by the set of seminorms 
	\begin{equ}
		\|f\|_{\alpha; A ,\K,\mfp} \eqdef \inf_{U \supset A} \inf_{\substack{g \in \CC^{\alpha}_{\s}(U;E)\\ \pi_U(g) = f}} \|g\|_{\alpha;U,\K,\mfp}  \;  ,
	\end{equ}
	where $\mfK \subset A$ compact. If the index $\K$ is dropped, the supremum is taken over all of $A$. When no confusion may arise we often drop the index $A$. 

\end{definition}

\begin{remark}\label{rem:besovtensorproduct}
	By following the argument of \cite[Theorem~44.1]{Trev67} in the case of $\cC^r(U; E)$ for $U \subset \R^d$ open, it is straightforward to prove  that $\CC^\alpha_{\s} \left( U; E\right)$ is isomorphic to the injective tensor product $\CC^\alpha_{\s} \left( U\right) \wotimes_{\eps} E$. 
For a closed subset $A \subset \R^n$ as in Definition~\ref{def:ClsdHolderSpace}, this follows from the universal property of the direct limit. 
\end{remark} 

\begin{remark}
	For Proposition~\ref{prop:LqExchange}, we also need another equivalent set of seminorms on $\CC^{\alpha}_\s(\R^d;E)$. As in \cite{Hai14}, let $\chi \colon \R \to \R$ be a scaling function of regularity $r \in \mathbb{N}_+$. We define the $\s$-scaled lattice to be
	\begin{equ}
		\Lambda^\s_n \eqdef 2^{-n \s} \Z^d \subset \R^d \; .
	\end{equ}
	and as in \cite{Hai14} an $\s$-scaled wavelet basis on $\R^d$ with mother wavelets $\chi_z$ for $z \in \Lambda^{\s}_0$ and father wavelets $\phi^n_z$ for $z \in \Lambda^{\s}_n$, $n \in \N$ and $\phi \in \Phi$ where $\Phi$ is the finite set of generating father wavelets.

	We define the seminorm
	\begin{equ}
		\|\xi\|_{w,\alpha;\K,\mfp} \eqdef\sup_{n \in \N} \sup_{z \in \Lambda_n^{\s} } \sup_{\phi \in \Phi}	2^{\frac{n|\s|}{2}+n \alpha } \mfp \left( \xi(\phi^n_z) \right) \vee \sup_{z \in \Lambda_0^{\s} } \mfp\left( \xi(\chi_z) \right)\;,
	\end{equ}
	which is equivalent to $\| \bigcdot \|_{\alpha; \K, \mfp}$ by a straightforward argument.

\end{remark}

\begin{remark}
	For $\alpha > 0$, $\alpha \notin \N$ the spaces $\mathcal C^\alpha_{\mathfrak s} \left( \R^d; E\right)$ agree with the classical H\"older spaces.
\end{remark}

\begin{definition}
	A kernel $K \colon \R^d \setminus\{0\} \to \C$ smooth except for a singularity at the origin is called $\beta$-regularising if it is compactly supported and for every $k \in \N^d$ there exists a constant $C$ such that for all $x \in \R^d$
	\begin{equ}
		\big| D^k K(x) \big| \leqslant C \|x\|_{\s}^{\beta-|\s|-|k|_{\s}} \; .
	\end{equ}
\end{definition}

\begin{theorem}[Schauder Estimate]
\label{SchdrThm}
	Let $\alpha \in \R$, $\beta >0$ and $K$ a $\beta$-regularising kernel.
	Convolution with $K$ is a continuous map $\mathcal C^\alpha_{\mathfrak s}(\R^d; E) \to \mathcal C^{\alpha+\beta}_{\mathfrak s}(\R^d; E)$.
\end{theorem}
\begin{proof}
	Thanks to Remark~\ref{rem:besovtensorproduct} and \cite[Proposition~43.6]{Trev67} the general case follows from the scalar case $E = \mathbb{R}$ proven in the classical Schauder estimate, cf.\ \cite[Theorem~14.17]{Hai20}.
\end{proof}

\begin{theorem}
\label{YngMultThm}
	Let $\CA$ be a locally $m$-convex topological algebra.

	For $\alpha, \beta \in \R$, the map $(f,g) \mapsto f\cdot g$ extends to a continuous bilinear map $\CC_{\s}^\alpha(\R^d; \CA) \times \CC_{\s}^\beta(\R^d; \CA) \to \CC_{\s}^{\alpha \wedge \beta}(\R^d; \CA)$ if and only if $\alpha + \beta > 0$. W.r.t.\ each seminorm $\mfp \in \mfP$ it is Lipschitz continuous with constant $R$ on the ball of radius $R$
	\begin{equ}
		B_R(0) = \left\{ (f,g) \in 	\CC_{\mathfrak s}^\alpha(\R^d; \CA) \times \CC_{\s}^\beta(\R^d; \CA) \,  \middle| \, \|f\|_{\alpha ; \mfp}  \vee \|g\|_{\beta; \mfp} \leqslant R \right\} \; .
	\end{equ}
	The statement above also holds if one replaces $\R^d$ with an open or closed subset of $\R^{d}$.  
\end{theorem}
\begin{proof}
	This can be argued analogously to the proof of the case $\CA = \R$ in \cite[Proposition~4.14]{Hai14}. 
This proof carries over since, for $r > 0$, one has $\|f g\|_{\cC^{r} ; \K, \mfp} \lesssim \|f \|_{\cC^r ; \K, \mfp}\|g\|_{\cC^r ; \K, \mfp}$ on the point-wise product of the bonafide functions $f$ and $g$, and one has similar estimates for the point-wise multiplication of modelled distributions, cf.\ \cite[Theorem~4.7]{Hai14}
\end{proof}

\begin{proposition}
\label{prop:CCalphaCrit}
	Let $\xi \in \cD'(\R^d;E)$. Suppose that there exists a kernel $K \in \mfN^{\zeta}$ with $\zeta \in (-|\s|,0)$, s.t.\ for all $f \in \cD(\R^d)$ and all $\mfp \in \mfP$
	\begin{equ}
		\| \xi(f) \|_{\mfp}^2 \lesssim \left| \scal{f, K \ast f}_{L^2(\R^d)} \right| \; ,
	\end{equ}
	where the constant may depend on $\mfp$ but is independent of $f$. Then $\xi \in \CC^{\zeta/2} (\R^d ; E)$. The same holds true for $\R^d$ replaced with $\R^d \times \T^k$.
\end{proposition}
\begin{proof}
	This follows straighforwardly from the identity
	\begin{equs}
		\scal{ \CS^\delta_{\s ,z } f, K \ast ( \CS^\delta_{\s ,z }f)}_{L^2} &=\scal{ \CS^\delta_{\s ,0 } f, K \ast ( \CS^\delta_{\s ,0 }f)}_{L^2} = \delta^{-|\s|} \scal{  f, ( \CS^{\delta^{-1}}_{\s ,0 } K ) \ast f}_{L^2}
	\end{equs}
	and the bound
	$
		\delta^{-|\s|} \bigl| (\CS^{\delta^{-1}}_{\s ,0 } K)( x ) \bigr| = | K( \delta^\s x ) | \leqslant \bone_{\supp K}(x) \VERT K \VERT_{\zeta} \| x \|_{\s}^{\zeta} \delta^{\zeta}
	$.
\end{proof}

\begin{proposition}
\label{prop:LqExchange}
	Let $E$ be a Fr\'echet space and $(S,\mu)$ a measure space. Then for any $q \in [1,\infty)$ we have a continuous imbedding
	\begin{equ}
		\CC^{\alpha}_{\s} \left( \R^d ; L^{q}(S,\mu ; E) \right) \subset L^q\left(S, \mu ; \CC^{\alpha - \kappa} (\R^d ; E) \right) \; ,
	\end{equ}
	for any $\kappa > \frac{|\s|}{q}$. The same holds when one replaces $\R^d$ by $\R^d \times \T^{k}$. 	
\end{proposition}
\begin{proof}
	Let $(\chi_x,\phi^k_y)$ be the above introduced a wavelet basis. 
	Let  $q > \frac{|\s|}{\kappa}$ and $\xi \in \CC^{\alpha}_{\s} \left( \R^d ; L^{\infty -}(S,\mu ; E) \right) $. First we interpret $\xi$ as the measurable function $S \to \cD'(\R^d ; E)$ mapping
	\begin{equ}
		\mcb{p} \longmapsto \left( \phi \mapsto \xi(\phi) (\mcb{p}) \right)
	\end{equ}
	where $\xi(\phi)(\mcb{p})$ is the value in $E$ of $\xi(\phi)$.

	Now using the wavelet $\CC^{\alpha - \kappa}$-seminorms we find that for $\mcb{p} \in S$
	\begin{equs}
		\| \xi( \bigcdot )(\mcb{p}) &\|_{w,\alpha -\kappa ;\K, n}^q \lesssim \\
		& \lesssim \sup_{k \in \N} \sup_{z \in \Lambda_n^{\s} \cap \K} \sup_{\phi \in \Phi}	2^{\frac{k|\s|q}{2}+k (\alpha-\kappa) q} \| \xi(\phi^k_z) (\mcb{p}) \|_{n}^q  \vee \sup_{z \in \Lambda_0^{\s} \cap \K} \| \xi(\chi_z)(\mcb{p}) \|_{n}^q \; .
	\end{equs}
	We will estimate the integral with respect to $\mcb{p}$ of the terms on the right-hand side to conclude that $\xi$ is in $L^q\left(S, \mu ; \CC^{\alpha - \kappa} (\R^d ; E) \right)$. By estimating the suprema with sums we can pass the integral through them to directly estimate
	\begin{equ}
		\left\| \bigl\| \xi ( \phi^k_z ) \bigr\|_n \right\|_{L^q} \lesssim 2^{- \frac{k |\s|}{2} - k \alpha} \; .
	\end{equ}
	Thus,
	\begin{equs}
		\left\| \|\xi(\bigcdot)\|_{w,\alpha-\kappa; \K, n} \right\|^q_{L^{q}} \lesssim \sum_{k \in \N}  \sum_{z \in \Lambda_k^{\mathfrak s} \cap \mathfrak K} \sum_{\phi \in \Phi}  2^{k (\alpha-\kappa) q} 2^{ - k \alpha q} \lesssim \sum_{k \in \N} 2^{|\s|k -\kappa k q}
	\end{equs}
	which converges since we assumed that $\kappa > \frac{|\s|}{q}$.
\end{proof}

Before defining the H\"older spaces on $\T^k$ we note that there is a continuous embedding $\cD'(\R^{d-k} \times \T^k ; E) \hookrightarrow \cD'(\R^d ; E)$ by extending a distribution  $\cD'(\R^{d-k} \times \T^k ; E)$ periodically. We will identify the distribution $\R^{d-k} \times \T^k$ with the corresponding periodic distributions on $\R^d$.

\begin{definition}
	We define for any $\alpha \in \R$ a partially periodic $\alpha$-H{\"o}lder-Besov space given by $\CC^\alpha_{s}(\R^{d-k}\times \T^k;E) \eqdef \CC^\alpha_{s}(\R^d;E) \cap \mathscr D'(\R^{d-k}\times \T^k;E)$.

Note that $\CC^\alpha_{s}(\R^{d-k}\times \T^k;E)$ is a closed subspace of $ \CC^\alpha_{s}(\R^d;E)$ and moreover convolution with a translation invariant kernel preserves the periodicity properties as convolution commutes with translation. In particular, all the results remain true when replacing $\R^d$ with $\R^{d-k} \times \T^k$.
\end{definition}

\subsection{H\"older Spaces with Singularities}
\label{sec:Holder_Sing}
In this section we discuss H\"older--Besov spaces of space-time functions but which can have singular behaviour near the $T=0$ hyperplane. 
This is important in order to allow to take initial data for our system of equations that is as rough as the linear solution and is also convenient for obtaining a contractive factor when formulating our PDE as a fixed-point problem. 

Recall that the theory of regularity structures when applied to regularity structures of polynomials reproduces the classical theory of H\"older--Besov spaces. In what follows we formulate a small tweak of this classical theory in order to allow blow-up at the $T=0$ hyperplane.
We adopt the language of regularity structures to do this but only because it is convenient, more classical methods could also be used here. 

\begin{definition}
\label{def:SingHolerSp}
	For $\alpha, \eta \in \R$ we say that a distribution $\xi \in \CC_\s^{\alpha,\eta}(\R^d; E)$ if and only if $\xi \in \CC^{\eta \wedge 0}_\s(\R^d; E)$ and for every $\K \subset \R^d$ compact and every seminorm $\mfp \in \mfP$ there exists a constant $C > 0$, s.t.\ for all $(t,x)\in \R \times \R^{d-1} \setminus \{t=0\} \cap \K$, all $\psi \in \CB^{r,\alpha}_{\s ,0}$ where $r \eqdef 0 \lor - \lfloor \alpha \rfloor$, and all $\lambda \in [0, 1]$ satisfying  $2 \lambda \leqslant |t|^{1/\s_1}$,
	\begin{equ}
		\mfp \left( \xi \left( \CS^{\lambda}_{\s,z} \psi \right)  \right) \leqslant C (|t| \wedge 1)^{\frac{\eta-\alpha}{\s_1}} \lambda^{\alpha} \; ,
	\end{equ}
	and for all $\phi \in \CB^r_{\s, 0}$
	\begin{equ}
		\mfp \left( \xi \left( \CS_{\s,z}^{\lambda_t} \phi \right)  \right) \leqslant C (|t|\wedge 1)^{\frac{\eta \wedge 0}{\s_1}  } \; ,
	\end{equ}
	where $\lambda_t \eqdef \frac{|t|^{1/\s_1}}{2}\wedge 1$. 

The smallest possible constant for the above inequalities will be denoted by $\VERT \xi \VERT_{\alpha,\eta;\K,\mfp}$.
For $U \subset \R^d$ open and $A \subset \R^d$ closed, we define $\CC^{\alpha,\eta}(U;E)$ and $\CC^{\alpha,\eta}(A;E)$ analogously to Definitions~\ref{def:OpnHolderSpace} and \ref{def:ClsdHolderSpace} respectively.

\end{definition}

\begin{remark}
	From the definition one has that $\CC^{\alpha,\alpha}_{\s} \supset \CC^{\alpha}_{\s}$, for $\alpha > 0$, $\CC^\alpha_\s (\R_+ \times \R^{d-1}; E) = \CC^{\alpha,\alpha}_\s (\R_+ \times \R^{d-1}; E)$, and that for any $\eta \leqslant \eta'$, one has for all $f \in \CC^{\alpha,\eta'}_{\s}$
	\begin{equ}
		\VERT f\VERT_{\alpha,\eta; \K, \mfp} \leqslant \VERT f\VERT_{\alpha,\eta'; \K, \mfp} \; .
	\end{equ}
\end{remark}

We now state a slight modification of the reconstruction theorem of \cite{Hai14}, we defer its proof to the end of the section. 
\begin{theorem}
	\label{thm:EtaRecon}
	Let $V$ be a sector of regularity $\alpha$, let $H = \{t = 0\}\subset \R^d$ be the time zero hyperplane, and let $\gamma,\eta \in \R$, s.t.\ $\eta\wedge \alpha > d-1$ and $\gamma > 0$. The reconstruction operator $\CR \colon \CD^{\gamma,\eta}_H(V) \to \CC_\s^{\alpha \wedge \eta}(\R^d)$ maps into the subspace $\CC_\s^{\alpha,\alpha\wedge \eta}(\R^d)$.	
\end{theorem}

\begin{remark}
	We need this modification of the original theorem as we are running the fixed-point argument at the level of the H\"older--Besov spaces directly rather than in a space of singular modelled distributions. If we were to use only \cite[Proposition~6.9]{Hai14}, the resulting regularities would be suboptimal and would not allow us to close the fixed-point argument.

	For the definition of concepts in this theorem please refer to \cite{Hai14}, in particular Definitions~2.1 \& 2.5 for sectors, Definitions~3.1 \& 6.2 for the space $\CD^{\gamma,\eta}_H(V)$, and Theorem~3.10 \& Proposition~6.9 for the reconstruction operator $\CR$.
\end{remark}

The next statements in this section are analogues of
Theorems~\ref{SchdrThm} and~\ref{YngMultThm} and they can be proven using the same arguments in conjunction with Theorem~\ref{thm:EtaRecon}. 


\begin{theorem}
	\label{SchauderEstBU}
	Let $\alpha, \eta \in \R$ and $K$ be a $\beta$-regularising kernel. Convolution with $K$ is a continuous map
	\begin{equ}
		\CC_\s^{\alpha,\eta}(\R^d ; E) \longrightarrow \CC_\s^{\alpha+\beta,\eta+\beta} (\R^d ; E) \; .
	\end{equ}
	
	In addition, if $K$ is non-anticipative, i.e.\ $K(t,x) \equiv 0$ on $\{t < 0\}$, then the same is true for $\R^d$ replaced with $\spacetime_T$ for any $T>0$.

\end{theorem}

\begin{theorem}
	\label{YngMultBUThm1}
	Let $\alpha > 0$, $\beta, \eta \in \R$, $\eta \leqslant \alpha$, and $D = \spacetime_{T}$ or $D = \R^{d}$.

	The map $(f,g) \mapsto f\cdot g$ extends to a continuous bilinear map $\CC_{\s}^{\alpha,\eta}(D; \CA) \times \CC_{\s}^\beta(D; \CA) \to \CC_{\s}^{\beta,  \beta + \eta }(\D; \CA) $ if $\alpha + \beta > 0$ and $(\beta + \eta) \wedge \beta > \mfm \eqdef - |\s|+ \s_1 $.
\end{theorem}

\begin{theorem}
	\label{YngMultBUThm2}
	Let $\alpha, \beta > 0$, $\eta \leqslant \alpha$, $ \eta'\leqslant \beta$, and $D = \spacetime_{T}$ or $D = \R^{d}$.

	The map $(f,g) \mapsto f\cdot g$ extends to a continuous bilinear map $\CC_{\s}^{\alpha,\eta}(D; \CA) \times \CC_{\s}^{\beta, \eta'}(D; \CA) \to \CC_{\s}^{\alpha \wedge \beta , \eta + \eta' }(D; \CA) $ if $\eta + \eta' > - \mfm$.
\end{theorem}


We also have the following versions of \cite[Proposition~6.15]{Hai14} and \cite[Theorem~7.1]{Hai14}. 
\begin{theorem}
	\label{DerivativeBU}
	Let $\alpha, \eta \in \R$, $i \in \{1,\dots, d\}$ and $D = \spacetime_{T}$ or $D = \R^{d}$. Then the derivative $\partial_i$ is a continuous map
	\begin{equ}
		\CC_\s^{\alpha,\eta}(D ; E) \longrightarrow \CC_\s^{\alpha-\s_i,\eta-\s_i} (D ; E) \; .
	\end{equ}
\end{theorem}

\begin{proposition}
	\label{prop:TimeImprov}

	Let $\alpha > 0$, $\eta \leqslant \alpha$, $0 < T \leqslant 1$, $\K \subset \R^{d-1}$ compact, and $\K_T \eqdef (-\infty, T] \times \K$. Suppose that $f \in \CC^{\alpha,\eta}_{\s}$ is supported in $[0,\infty) \times \R^{d-1}$, then for any $\kappa \geqslant 0$
	\begin{equ}
		\VERT f \VERT_{\alpha, \eta - \kappa; \K_T,\mfp} \lesssim T^{\frac{\kappa}{\s_1}} \VERT f \VERT_{\alpha, \eta; \K_T, \mfp} \; .
	\end{equ}
\end{proposition}

\begin{proof}[of Theorem~\ref{thm:EtaRecon}]
	We assume that $\eta \leqslant \alpha$ for otherwise the assertion is trivial. Let $\K \subset \R^d$ be compact, $(t, x') = x \in \K \setminus H$, and $\psi \in \CB^r_{\s,0}$. 

	First, we split
	\begin{equ}
		\left| \CR f \bigl( \psi^\lambda_t \bigr) \right| \leqslant \left| \left(\CR f  - \Pi_x f(x) \right) \bigl( \psi^\lambda_x \bigr) \right|+  \left| \left( \Pi_x f(x) \right) \bigl( \psi^\lambda_x \bigr) \right| \; .
	\end{equ}
	From the properties of the reconstruction operator it follows that for all $\lambda \in (0, \lambda_t]$
	\begin{equs}
		\left| \left(\CR f  - \Pi_x f(x) \right) \bigl( \psi^\lambda_x \bigr) \right| & \lesssim \lambda^\gamma \sup_{\substack{y,z \in \supp \psi^\lambda_x  \\ y \neq z }} \sup_{\ell < \gamma} \frac{\| f(y) - \Gamma_{yz}f(z)\|_\ell}{\|y-z\|_\s^{\gamma-\ell}}  \leqslant \\
		& \leqslant \lambda^\gamma \sup_{\substack{y,z \in \supp \psi^\lambda_x  \\ y \neq z }} \sup_{\ell < \gamma} \frac{\| f(y) - \Gamma_{yz}f(z)\|_\ell}{\|y-z\|_\s^{\gamma-\ell} \|y,z\|_H^{\eta-\gamma} }  \|y,z\|_H^{\eta-\gamma}  \; .
	\end{equs}
	From $\lambda \leqslant \lambda_t$, it follows that $\|y,z\|_H \geqslant \frac{\|x\|_H}{2}$ and thus 
	\begin{align*}
		&\lesssim \lambda^\gamma \|x\|_H^{\eta-\gamma} \sup_{\substack{y,z \in \supp \psi^\lambda_x  \\ y \neq z }} \sup_{\ell < \gamma} \frac{\| f(y) - \Gamma_{yz}f(z)\|_\ell}{\|y-z\|_\s^{\gamma-\ell} \|y,z\|_H^{\eta-\gamma} }  \leqslant \\
		& \leqslant \lambda^\gamma \|x\|_H^{\eta-\gamma}  \VERT f \VERT_{\gamma , \eta; \overline{\mathfrak{K}}} = \lambda^\gamma \|x\|_H^{\alpha-\gamma} \|x\|_H^{\eta-\alpha}   \VERT f \VERT_{\gamma , \eta; \overline{\mathfrak{K}}} \lesssim \\
		& \lesssim \lambda^\gamma \lambda^{\alpha-\gamma} \|x\|_H^{\eta-\alpha}   \VERT f \VERT_{\gamma , \eta; \overline{\mathfrak{K}}} = \lambda^{\alpha} \|x\|_H^{\eta-\alpha}   \VERT f \VERT_{\gamma , \eta; \overline{\mathfrak{K}}}   \; ,
	\end{align*}
	where we used that $\lambda \leqslant \frac{\|x\|_H}{2}$ and $\alpha - \gamma < 0$, and $\overline{\K}$ is the 1-fattening of $\K$.

	For the second term, noting that $\eta - \ell \leqslant 0$ for all $\ell \in A$, we have 
	\begin{equs}
		\left| \left( \Pi_x f(x) \right) \bigl( \psi^\lambda_x \bigr) \right| & \lesssim \sup_{\ell < \gamma} \lambda^\ell \| f(x) \|_\ell = \\
		& = \sup_{\ell < \gamma} \lambda^\ell \|x\|_H^{\eta-\ell} \frac{\| f(x) \|_\ell}{\|x\|_H^{(\eta-\ell)\wedge 0}} \leqslant \\
		& \leqslant \|x\|_H^\eta \VERT f \VERT_{\gamma,\eta; \overline{\K}} \sup_{\ell < \gamma} \left( \frac{\lambda}{\|x\|_H} \right)^\ell \lesssim \\
		& \lesssim  \|x\|_H^\eta \VERT f \VERT_{\gamma,\eta; \overline{\K}} \frac{\lambda^\alpha}{\|x\|_H^\alpha} = \lambda^{\alpha} \|x\|_H^{\eta-\alpha}   \VERT f \VERT_{\gamma , \eta; \overline{\K}} \; ,
	\end{equs}
	where we used that $\lambda \leqslant \frac{\|x\|_H}{2}$.
\end{proof}

\subsection{Spaces of Local Solutions}
\label{sec:blowup_space}

In order to describe solutions that might blow up at a finite time we will introduce the following spaces which include a point at infinity. Within this section we assume that $E$ is a Banach space.
We write $\reminit$ for a Banach space of distributions on $\T^d$ with values in $E$ denote its norm by $\| \bigcdot \|_{\reminit}$. 

We add a point $\infty$ to $\reminit$ and define the space $\widehat{\CC}_0 \eqdef \reminit \sqcup \{\infty\}$ with the topology of $\reminit$ extended by the system of neighbourhoods of $\infty$ given by
\begin{equ}
	\left\{ g \in \reminit \, \middle| \, \|g\|_{\reminit} > N \right\}
\end{equ}
for any $N > 0$. 
We will also use the convention that $\|\infty\|_{\reminit} = \infty$.

For $f \in \cC \bigl([0,1]; \widehat{\CC}_0\bigr)$,  let $T[f] \eqdef \inf \left\{ t \geqslant 0 \, \big| \, f(t) = \infty \right\}$. Then we define the space
\begin{equ}
	\CC^\sol \eqdef \left\{ f \in \cC \big([0,1]; \widehat{\CC}_0\big) \, \middle| \,f(t) = \infty\quad \forall t > T[f]  \right\} \; .
\end{equ}
For $f \in \CC^\sol$ and $t \leqslant 1$ let
\begin{equ}
	S_f(t) \eqdef \sup_{s \leqslant t} \|f(s) \|_{\reminit} \in [0,\infty] \; .
\end{equ}
For a fixed smooth $\phi  \colon \R \to [0,\infty)$ with support in $[0,1]$ and $\int \phi = 1$ we define the mollified version of $S_f$
\begin{equ}
	S^L_f(t) \eqdef \tan \left( \int\limits_0^1 \arctan \left( S_f\left( t+ \frac{s}{L} \right) \right) \phi(s) \d s \right)
\end{equ}
which is increasing, and satisfies $S_f^L(t) \geqslant S_f(t)$ as well as \begin{equ}
	\inf \left\{ t \in [0,1] \, \middle| \, S^L_f(t) = \infty \right\} =	\inf \left\{ t \in [0,1] \, \middle| \, S_f(t) = \infty \right\} \; .
\end{equ}
We also fix a smooth function $\psi \colon \R \to [0,1]$ that is decreasing and such that  $\psi\restr (-\infty,1] \equiv 1$ and $\psi\restr [2,\infty) \equiv 0$. We combine $\psi$ and $S_f^L$ to cut off $f$ when its norm
reaches some threshold of order $\CO(L)$:
\begin{equ}
	\CC^\sol \ni f \longmapsto \Theta_L(f)\eqdef \psi\Big( \tfrac{S_f^L(\bigcdot)}{L} \Big) f \in \cC \left( [0,1] ; \reminit \right)\;,
\end{equ}
where we have adopted the convention that $0 \cdot \infty = 0$.

The cut-off function is smooth in $t$ since $S_f^L$ is smooth on $\big\{t \in [0,1] \, \big| \, S_f^L(t) < 3L \big\}$. 
We also have the bound $\sup_{t \in [0,1]} \left\| \Theta_L(f)(t) \right\|_{\reminit} \leqslant 2 L$. With this in place we can define the  metric $d(\bigcdot, \bigcdot) \eqdef \sum_{L = 1}^\infty 2^{-L} d_L(\bigcdot,\bigcdot)$ on $\CC^\sol$ where for $f,g \in \CC^\sol$
\begin{equ}
	d_L(f,g) \eqdef 1 \wedge \sup_{t \in [0,1]} \left\| \Theta_L(f)(t) -  \Theta_L(g)(t) \right\|_{\reminit} \; .
\end{equ}
The importance of this metric lies within the following lemma which tells us that the ``local'' continuity property of the solution map $\CS$ carries over to global continuity in the space $\CC^\sol$.
\begin{lemma}
\label{lemma:Csol_cont}
	A sequence $\left( f_n \right)_n \subset \CC^\sol$ converges to $f \in \CC^\sol$ if
	and only if, for every $L \in \N$, one has
	\begin{equ}
		\lim_{n \to \infty }  \sup_{t \in [0,T(L,n)]} \| f(t)-f_n(t) \|_{\reminit} = 0
	\end{equ}
	where
	\begin{equ}
		T(L,n)  =  \inf\left\{ t \in [0,1] \,	\middle| \, \|f(t) \|_{\reminit} \geqslant L \;\text{or}\; \|f_n(t) \|_{\reminit} \geqslant L \right\} \; .
	\end{equ}
\end{lemma}
\begin{proof}
	See Lemma~2.19 of the arXiv version of \cite{BCCH21}.
\end{proof}

\subsection{Heat Kernel Smoothing Property}

\begin{proposition}
\label{HKSmootProp}
	Let $E$ be a Fr\'{e}chet space. For all $\eta \in \R \setminus \N$, $\xi \in \CC^\eta(\T^d;E)$ and $\gamma > 0$ one has for all $\kappa > 0$
	\begin{equ}
		 G\xi \in \CC^{\gamma,\eta}_{\s} \bigl( \R \times \T^d; E \bigr) \cap \cC\bigl([0,\infty) ; \CC^{\eta-\kappa}_\s(\T^d;E)\bigr)
	\end{equ}
	where
	\begin{equ}
		G\xi(t,x) = \left(e^{t \Delta} \xi \right) (x) \;	.
	\end{equ}
\end{proposition}
\begin{proof}
	See Lemma~7.5 of \cite{Hai14} for $G\xi \in \CC^{\gamma,\eta}_{\s} \bigl( \R_+ \times \T^d; E \bigr)$. The second claim follows from the weak continuity of $t \mapsto e^{t \Delta} \xi$, the boundedness of $\| e^{t \Delta} \xi \|_{\eta; \T^d, \mfp}$ for all $\mfp \in \mfP$ and $t \in [0,\infty)$, and compactness of the embedding $\CC^\eta_\s(\T^d; E) \hookrightarrow \CC^{\eta-\kappa}_\s(\T^d; E)$ for all $\kappa >0$.
\end{proof}

	\section{\TitleEquation{N}{N}-Contractive Estimates}
	\label{sec:UltrConProof}

	We prove Theorem~\ref{thm:LocalUltraContr} in several steps, adapting \cite[Proposition~1.2.3]{GJ71}.

	Let $\mfH$ be a separable complex Hilbert space with antiunitary involution $\kappa$ and let
	\begin{equ}
		\CF(\mfH) \eqdef \bigoplus_{n \in \N} \mfH^{\otimes n}\;,\qquad \|h\|^2 = \sum_{n \in \N}\|h_n\|^2\;,
	\end{equ}
	denote the full Fock space of $\mfH$, where here and in the following $\otimes$ always denotes the usual Hilbert space tensor product. Let $P_- \colon \CF(\mfH) \to \CF_a(\mfH)$ be the projection operator where $P_- \restr{\mfH^{\otimes n}} \eqdef S_n^-$ as defined in \eqref{eq:Def_Asym_Prod}. 


	We can define a second set of ``creation'' and ``annihilation'' operators $\beta^\dagger$, $\beta$ on $\CF_a(\mfH)$
	which, for $f \in \mfH$ and $F \in \CF_a(\mfH)$, are given by
	\begin{equ}
		\beta^\dagger(f) F = f \wedge F\;,
	\end{equ}
	with $\beta(f)$ being the adjoint. For any $G \in \mfH^{\otimes n}$ one has
	\begin{equ}
		\bigl( \bigl(\beta^\dagger\bigr)^{\otimes n}(G) \bigl) F = P_- (G \otimes F) = (P_- G) \wedge F \; .
	\end{equ}
	Thus, $ \left( \beta^\dagger \right)^{\otimes n}$ can be viewed as a bounded operator $\mfH^{\otimes n} \otimes \CF_a(\mfH) \to \CF_a(\mfH)$ (and therefore $\beta^{\otimes n} \colon \CF_a(\mfH) \to (\mfH^*)^{\otimes n} \otimes \CF_a(\mfH)$ where $\mfH^*$ is the dual space of $\mfH$), and both
	are contractions.

	We note that $\alpha^\dagger = \sqrt{N+1} \beta^\dagger = \beta^\dagger \sqrt{N}$ and $\alpha = \sqrt{N} \beta = \beta \sqrt{N-1}$. More generally for $n \in \N$ we define the operators
	\begin{equs}
		(N+k)^{(n)} &\eqdef (N+k) \cdots (N+k+n-1)\\
		(N+k)_{(n)} &\eqdef (N+k)(N+k-1) \cdots (N+k-n+1)
	\end{equs}
	imitating the usual notation for the Pochhammer symbols which allows us to write down the identities
	\begin{equs}
		(\alpha^\dagger)^{\otimes n} &= (\beta^\dagger)^{\otimes n}  \sqrt{(N+1)^{(n)}} =\sqrt{(N)_{(n)}} (\beta^\dagger)^{\otimes n}\;, \\
		\alpha^{\otimes n} & =\beta^{\otimes n} \sqrt{(N)_{(n)}} = \sqrt{(N+1)^{(n)}} \beta^{\otimes n}\;.
	\end{equs}

	In the following we will interpret $\left((N)_{(n)}\right)^{-1}$ as the pseudo-inverse,
	i.e.\ it vanishes on $\{N < n\}$.
	From this discussion the following lemma follows.

	\begin{lemma}
	\label{lemma:NmbOpComRel}
		For all $n \in \N$ the maps
		\begin{equs}
			\bigl(\alpha^\dagger\bigr)^{\otimes n} \circ \left((N+1)^{(n)}\right)^{-\frac{1}{2}}  =  \left((N)_{(n)}\right)^{-\frac{1}{2}}\circ \bigl(\alpha^\dagger\bigr)^{\otimes n} &\colon \mfH^{\otimes n} \otimes  \CF_{a}(\mfH) \longrightarrow \CF_a(\mfH) \\
			\alpha^{\otimes n} \circ \left((N)_{(n)}\right)^{-\frac{1}{2}} = \left((N+1)^{(n)}\right)^{-\frac{1}{2}} \circ \alpha^{\otimes n} &\colon \CF_{a}(\mfH) \longrightarrow \mfH^{\otimes n} \otimes \CF_a(\mfH)
		\end{equs}
		are bounded operators.
	\end{lemma}

	Given a bounded operator $ A \colon  (\mfH^*)^{\otimes s}\otimes \CF_a(\mfH) \longrightarrow \mfH^{\otimes r} \otimes \CF_a(\mfH)$, we can combine $\beta^{\otimes s}$ and $\bigl( \beta^{\dagger} \bigr)^{\otimes r}$ to a single map
	\begin{equ}
		\widetilde W_{A}^{r,s} \eqdef \bigl(\beta^\dagger\bigr)^{\otimes r} \circ A \circ \beta^{\otimes s} \colon \CF_a(\mfH) \longrightarrow \CF_a(\mfH)
	\end{equ}
	with operator norm $\|\widetilde{W}_{A}^{r,s}\| \leqslant \|A\|$, as well as the potentially unbounded operators
	\begin{equ}
		W_{A}^{r,s} \eqdef \bigl(\alpha^\dagger\bigr)^{\otimes r} \circ A \circ \alpha^{\otimes s}
	\end{equ}
	with the quadratic form domain containing at least $\CD\bigl( (N)_{(s)} \bigr) \times \CD\bigl( (N)_{(r)} \bigr)$.

	If the operator $A$ is of the form $B \otimes I$ where $B$ is a bounded operator $(\mfH^*)^{\otimes s} \to \mfH^{\otimes r}$, then
	we can write
	\begin{equ}
		\widetilde{W}_A^{r,s} = \left((N)_{(r)}\right)^{-\frac{1}{2}} \circ W_{A}^{r,s}  \circ \left((N)_{(s)}\right)^{-\frac{1}{2}} = W_{A}^{r,s} \circ \left((N+1-r)^{(r)}\right)^{-\frac{1}{2}} \circ \left((N)_{(s)}\right)^{-\frac{1}{2}}
	\end{equ}
	using Lemma~\ref{lemma:NmbOpComRel}, that $\left((N+1)^{(r)}\right)^{-\frac{1}{2}}$ commutes with $A$, and
	\begin{equ}
		\left((N+1)^{(r)}\right)^{-\frac{1}{2}} \circ \alpha^{\otimes s} \circ \left((N)_{(s)}\right)^{-\frac{1}{2}} = \alpha^{\otimes s} \circ\left((N+1-s)^{(r)}\right)^{-\frac{1}{2}} \circ  \left((N)_{(s)}\right)^{-\frac{1}{2}}.
	\end{equ}
	We conclude that in this case $W_{A}^{r,s} \circ (1+N)^{-\frac{r+s}{2}}$ is a bounded operator as we have the estimate for all $k \in \Z$
	\begin{equ}
		\left\| \frac{N+k}{N+1} \right\| \leqslant |k|+1
	\end{equ}
	giving us the constant in Theorem~\ref{thm:LocalUltraContr} as $|k| \leqslant |r-s|$ for all $k$ appearing in $\left((N+1-s)^{(r)}\right)^{-\frac{1}{2}} \circ  \left((N)_{(s)}\right)^{-\frac{1}{2}}$.

	We can summarise this as follows:
	\begin{lemma}
	\label{lemma:ClosableOp}
		For any bounded operator $B \colon (\mfH^*)^{\otimes s} \to \mfH^{\otimes r}$, the operator $W_{B\otimes \bone}^{r,s} \circ (N+1)^{-\frac{r+s}{2}}$ is bounded with norm satisfying
		\begin{equ}
			\bigl\| W_{B\otimes \bone}^{r,s} \circ (N+1)^{-\frac{r+s}{2}} \bigr\| \lesssim (|r-s|+1)^{\frac{r+s}{2}} \| B \| .
		\end{equ}
		Furthermore, $W_{B\otimes \bone}^{r,s}$ is a closable operator with core $\CD \big( (N+1)^{-\frac{r+s}{2}} \big) $.
	\end{lemma}

	\begin{remark}
		We note that everything up to this point works verbatim for bosonic creation and annihilation operators.
	\end{remark}

	We will be mainly interested in the case where $B$ is a Hilbert--Schmidt operator, i.e.\
	\begin{equ}
		\Trace\bigl(B^\dagger B\bigr) = \Trace\bigl(BB^\dagger \bigr) < \infty \; .
	\end{equ}
	Recall that the space of Hilbert--Schmidt operators $\mathrm{HS}(\mfH, \mfH)$ is 
	canonically isomorphic to the Hilbert space tensor product $\mfH \otimes \mfH$ via
	$(f \otimes g)(h) = f \scal{\kappa g, h}$ (the appearance of $\kappa$ is necessary since otherwise 
	the right-hand side would be antilinear in $g$), and similarly for $\mathrm{HS}((\mfH^*)^{\otimes s}, \mfH^{\otimes r})$. In particular for
	\begin{equ}
		B = \sum_j f_1^j \otimes \cdots \otimes f_{r+s}^j \in \mfH^{\otimes(r+ s)}
	\end{equ}
	we have that
	\begin{equ}
		W_{B \otimes I}^{r,s} = \sum_{j} \alpha^\dagger(f_1^j) \cdots \alpha^\dagger(f_r^j) \alpha(\kappa f_{r+1}^j) \cdots \alpha(\kappa f_{r+s}^j)
	\end{equ}
	is a well-defined, closable operator.

	To obtain the estimate with bound $(1+N)^{\frac{r+s-1}{2}}$ instead of $(1+N)^{\frac{r+s}{2}}$ we have to interpret one of the $\alpha$'s appearing in $W^{r,s}_A$ instead as part of the bounded operator $(\mfH^*)^{\otimes s}\otimes \CF_a(\mfH) \longrightarrow \mfH^{\otimes r} \otimes \CF_a(\mfH)$ as we can still commute $\left((N+1)^{(r)}\right)^{-\frac{1}{2}} $ past it without changing the final expression. In particular, if
	\begin{equ}
		B = \sum_j f_1^j \otimes \cdots \otimes f_{r+s}^j \in \mfH^{\otimes(r+ s)}
	\end{equ}
	is Hilbert-Schmidt operator, then
	\begin{equ}
		A = \sum_j f_1^j \otimes \cdots \otimes f_{r+s-1}^j \otimes \alpha(\kappa f_{r+s})
	\end{equ}
	is a bounded operator. Performing the same commutation calculation with the number operators finally proves Theorem~\ref{thm:LocalUltraContr}.

\section{Proof of Theorem~\ref{thm:ContinuityRep}}\label{appendix:theoremproof}

We start by reviewing some results of operator algebra theory we need, for further details cf.\ \cite{BR87}.

	Let $\mfH$ be a Hilbert space, and $\CB(\mfH)$ be the $C^*$-algebra of bounded operators on $\mfH$. This space is the continuous dual space of the Banach space of trace class operators
	\begin{equ}
		\CL^1(\CB(\mfH), \Trace ) \eqdef \left\{ A \in \CB(\mfH) \, \middle| \, \Trace  \sqrt{A^\dagger A}  < \infty \right\}
	\end{equ}
 with the pairing $\scal{\bigcdot, \bigcdot} \colon \CB(\mfH) \times \CL^1(\CB(\mfH), \Trace ) \to \C$
	\begin{equ}
		\scal{A,B} \eqdef \Trace(AB) \; .
	\end{equ}
This allows us to conclude from the Banach-Alaoglu theorem that closed balls of finite radius in $\CB(\mfH)$ are compact in the weak$^*$ topology. 

We say that a net $(A_\alpha)_\alpha \subset \CB(\mfH)$ converges strongly to an operator $A \in \CB(\mfH)$ if and only if for all $v \in \mfH$
	\begin{equ}
		\lim_{\alpha} A_\alpha v = A v \; .
	\end{equ}

	We now introduce an order on $\CB(\mfH)$ along with a corresponding positive cone. 
For two operators $A,B \in \CB(\mfH)$ we say that $A \leqslant B$ if and only if $B-A$ is a non-negative operator. Furthermore this order satisfies the property \begin{equ}
		A \leqslant B \; \implies \; \|A\| \leqslant \|B\| \; .
	\end{equ}
	The positive cone of $\CB(\mfH)$ is the set
	\begin{equ}
		\CB(\mfH)_+ \eqdef \left\{ A \in \CB(\mfH) \, \big| \, A \geqslant 0 \right\} \; .
	\end{equ}
	For a subset $U \subset \CB(\mfH)$ we say that $A \in \CB(\mfH)$ is an upper bound if and only if $A \geqslant B$ for all $B \in U$. If a least such upper bound of $U$ exists, we denote it by $\sup U$. Note that any least upper bound is necessarily unique, but it need not exist in general.

	\begin{theorem}\label{thm:strong_conv}
		Let $(A_\alpha)_\alpha \subset \CB(\mfH)_+$ be an increasing net of operators, i.e.\
		\begin{equ}
			\alpha \leqslant \beta \; \implies \; A_\alpha \leqslant A_\beta \; .
		\end{equ}
		If $(A_\alpha)_\alpha $ has an upper bound $B \in \CB(\mfH)_+$, $A \eqdef \sup_\alpha A_\alpha$ exists and the net converges strongly to $A$.
	\end{theorem}
	\begin{proof}
		Since the net has upper bound $B$ it is contained in the ball of radius $\|B\|$ which is a weak$^*$-compact set. Thus let $C_\alpha$ be the weak$^*$-closure of the set $(A_\beta)_{\beta\geqslant \alpha}$, which as closed subsets of a compactum are themselves compact. It follows that $\bigcap_{\alpha} C_\alpha \neq \emptyset$ for otherwise some finite intersection of these sets would be empty. Let $A$ be an element of $\bigcap_{\alpha} C_\alpha$.

Next we show that $A$ is $\sup_\alpha A_\alpha$. The sets $\mfC_\alpha \eqdef \left\{ B \in \CB(\mfH) \, \big| \, B \geqslant A_{\alpha} \right\}$ are weak$^*$ closed, as for any $v \in \mfH$, the following functional is per definitionem weak$^*$-continuous
		\begin{equ}
			\ell_v (B) = \scal{v,Bv} = \Trace(P_v B)
		\end{equ}
		where $P_v$ is the projection onto the one dimensional space spanned by $v$, so
		\begin{equ}
			\mfC_\alpha = \bigcap_{v \in \mfH} \ell_v^{-1} \left( [\ell_v(A_\alpha), \infty) \right)
		\end{equ}
		is weak$^*$-closed. Since the net is increasing it follows that $(A_\beta)_{\beta \geqslant \alpha} \subset \mfC_\alpha$ and therefore also $C_\alpha \subset \mfC_\alpha$. On the other hand, $\bigcap_{\alpha} \mfC_\alpha$ is the set of upper bounds of $(A_\alpha)_\alpha$, thus $A$ is an upper bound of the net. Since the functionals $\ell_v$ are weak$^*$-continuous it follows that if $B$ is an upper bound for the net, i.e.\ $\ell_v(B - A_\alpha) \geqslant 0$ for all $\alpha$ and $v$, that is also an upper bound for any element in the weak$^*$-closure of the net, in particular $A$. Therefore, $A$ is a least upper bound and the weak$^*$-limit of the net $(A_\alpha)_\alpha$. 

		To show strong convergence, we write
		\begin{equs}
			\| (A-A_\alpha) v\|^2 &= \scal{v,(A-A_\alpha)^2 v} =  \scal{v,\sqrt{A-A_\alpha}(A-A_\alpha)\sqrt{A-A_\alpha} v} \leqslant \\
			&\leqslant  \scal{v,\sqrt{A-A_\alpha}\|A-A_\alpha\|\sqrt{A-A_\alpha} v} \leqslant 2 \| A \| \scal{v,(A-A_\alpha)v} =\\
			&= 2 \| A \| \ell_v(A-A_\alpha) \rightarrow 0\;,
		\end{equs}
		which we used that $A_{\alpha} \rightarrow A$ in the weak$^*$-topology.
	\end{proof}

Next we note the following criterion for whether an element belongs to the center of a topological algebra.

	\begin{lemma}
	\label{lemma:ComCondition}
		Let $\CA$ be a topological algebra and $C \subset \CA$ a set whose algebraic closure, i.e.\ the set of finite linear combinations of finite products of elements in $C$, is dense in $\CA$. Then $A \in Z(\CA)$ if and only if it commutes with all elements of $C$.
	\end{lemma}
	\begin{proof}
		The necessity is obvious. Concerning the sufficiency, note that since the commutator of two elements is bilinear and satisfies for $A,B,B' \in \CA$
		\begin{equ}[eq:DerivPropCmttr]
			[A, BB']_- = [A,B]_-B' + B[A,B']_-
		\end{equ}
		the assumption implies that $A$ commutes with the algebraic closure of $C$ and therefore by density all elements of $\CA$.
	\end{proof}

Next, we prove some lemmas regarding bounded operators on $\mathcal{F}_{a}(\mfH)$. 

	\begin{lemma}\label{lemma:center}
		Suppose that $A \in \CB(\CF_a(\mfH))$ commutes with  $\alpha^\dagger(f)$ and $\alpha(f)$ for all $f \in \widetilde\Gamma$, where $\widetilde{\Gamma} \subset \mfH$ is dense. Then $A = \lambda \bone$ for some $\lambda \in \C$.
\end{lemma}
	\begin{proof}
		By Lemma~\ref{lemma:ComCondition} the assumption implies that $A$ commutes with all of $\CA\F(\mfH)$. Now as $\CA\F(\mfH) \1$ is dense in $\CF_a(\mfH)$ this implies that $A$ is uniquely determined by its action on $\1$. Since, for all $f \in \widetilde{\Gamma}$, 
		\begin{equs}
			\alpha(f) A \1 = A \alpha(f) \1 = 0
		\end{equs}
		it follows that $A \1$ must be multiple of $\1$, because $\scal{\1}$ is the unique subspace contained in the kernel of all $\alpha(f)$. Therefore, there exists some $\lambda \in \C$
		such that
			$A  = \lambda \bone$ as claimed.
	\end{proof}
	
In the next lemma, we take limits over nets in $\Gr(\mfH)$ and associated projections. 

\begin{remark}
Note that  $b \mapsto \widehat{P}_b$ is a monotonously increasing net as
	\begin{equ}
		b \leqslant b' \; \iff \; \widehat{P}_b \leqslant \widehat{P}_{b'} \; ,
	\end{equ}
	where $\widehat{P}_b \leqslant \widehat{P}_{b'}$ means that the operator $\widehat{P}_{b'}-\widehat{P}_b$ is positive.
\end{remark}
\begin{lemma}\label{lem:proj_converge}
		Let $(\Gamma_n)_n$ be as in Definition~\ref{def:filtration} and define $\Gamma_\infty \eqdef \bigcup_{n \in \N} \Gamma_n$.
		The net $(\widehat{P}_b)_{b \in \Gamma_\infty}$ converges strongly to the identity $\bone$ in the von Neumann algebra $\CB(\CF_a(\mfH))$. 
\end{lemma}
\begin{proof}
	Since $P_b \leqslant \bone$ for all $b \in \Gamma_\infty$ it follows that the net converges strongly to some projection
	\begin{equ}
		\widehat{P} \eqdef \sup_{b \in \Gamma_\infty}  \widehat{P}_{b} \; ,
	\end{equ}
	by Theorem~\ref{thm:strong_conv}.
		
	We prove that for all $f \in \widetilde{\Gamma} \eqdef \bigcup_{b \in \Gamma_\infty} b$
	\begin{equ}[eq:IrredCommRel]
		[\widehat{P}, \alpha(f)]_- = 0 \qquad \text{and} \qquad [\widehat{P}, \alpha^\dagger(f)]_- = 0\;,
	\end{equ}
	it will then follow from Lemma~\ref{lemma:center} that $\widehat{P} = \lambda \bone$ for some $\lambda \in \C$ and since $\widehat{P}$ is a (non-zero) projection it will have to be the case that  $\lambda = 1$.
	
	Since $\widehat{P}$, $\alpha(f)$, and $\alpha^\dagger(f)$ are bounded operators, it is enough to show \eqref{eq:IrredCommRel} on a dense subset, e.g.\ $\mathring{\mathcal{A}}\F(\widetilde{\Gamma}) \1$, i.e.\ the $\star$-algebra generated just by elements of $\widetilde{\Gamma}$ instead of all of $\mfH$. Its closure is still $\CA(\mfH)$ by density.
	Let $v = A \1$ for some $A \in \mathring{\mathcal{A}}\F(\widetilde{\Gamma})$ and fix some $f \in \widetilde{\Gamma}$. 
	One can write $A$ as a polynomial in $\alpha^{\dagger}(g)$ and $\alpha(g)$ for elements $g \in \widetilde{\Gamma}$. 
	Fix some $b \in \Gamma_\infty$ that contains all the vectors $ g \in \widetilde{\Gamma}$ that  appear in this polynomial expression for $A$ and also contains the vector $f$.
		
	It follows that for any $b' \in \Gamma_\infty$ with $b \leqslant b'$, one has $v, \alpha(f)v, \alpha^{\dagger}(f)v \in \widehat{P}_{b'} \CF_a(\mfH)$. 
	Therefore, $\bigl[ \widehat{P}_{b'}, \alpha(f) \bigr]_- v = 0$ and $\bigl[\widehat{P}_{b'}, \alpha^\dagger(f)\bigr]_- v = 0$, and by taking limits we see that 
	\begin{equ}
		\bigl[ \widehat{P}, \alpha(f) \bigr]_- v = 0 \qquad \text{and} \qquad \bigl[\widehat{P}, \alpha^\dagger(f)\bigr]_- v = 0\;.
	\end{equ}
	Since $v$ and $f$ were arbitrary we have proven the desired statement. 
\end{proof}

We can finally give the main proof of this section. 

\begin{proof}[of Theorem~\ref{thm:ContinuityRep}]
	We first establish that $\digamma$ descends to the quotient $\star$-algebra $\widehat\mfA\F(\mfH)$. 
	Suppose that for some $A \in \mfA(\mfH)$ we have $\pi_b(A) =  0$ for all $b \in \Gr^{U}(\mfH)$. 
	Then, for any finite-dimensional subspace $b$ containing all the vectors of $\mfH$ appearing in the expression for $A$, we have
	\begin{equ}[e:largeb]
		\digamma(A)\widehat{P}_b = \pi_b(A) = 0 \; .
	\end{equ}
	
Using Lemma~\ref{lem:proj_converge}, we see that $\digamma(A)$ is the strong limit of $\digamma(A)\widehat{P}_b$ and must therefore itself be zero. 

To show that $\digamma$ extends to a $C^*$-homomorphism on $\mfA_{\infty}$, it suffices to show that $\digamma(\bigcdot)$ is bounded with respect to the norm $\| \bigcdot \|_{\infty}$ on $\mfA_{\infty}(\mfH)$. 
In particular, we show that for all $A \in \widehat{\mfA}\F(\mfH)$
		\begin{equ}[e:boundF]
			\| \digamma(A) \| \leqslant \| A \|_\infty \; ,
		\end{equ}

Fix $A$ and let $b \in \Gamma_\infty$ such that \eqref{e:largeb} holds. 
Then,
		\begin{equ}
			\pi_b(A) = \digamma(A) \widehat{P}_b \xrightarrow{b \in \Gamma_\infty}	\digamma(A)
		\end{equ}
		in the topology of strong convergence, as $\widehat{P}_b \uparrow \bone$ by the above lemma.

It follows that
		\begin{equs}
			\| \digamma(A) \| &=  \sup_{\substack{v\in \CF_a(\mfH) \\ \|v\| = 1}} \| \digamma(A) v\| = \sup_{\substack{v\in\CF_a(\mfH) \\ \|v\| = 1}} \lim_{b \in \Gamma_\infty} \| \pi_b(A) v\|  \leqslant \\
			& \leqslant \sup_{\substack{v\in\CF_a(\mfH) \\ \|v\| = 1}} \sup_{b \in \Gamma_\infty} \| \pi_b(A) v\| = \sup_{b \in \Gamma_\infty} \sup_{\substack{v\in\CF_a(\mfH) \\ \|v\| = 1}} \| \pi_b(A) v\| =\\
			&=\sup_{b \in \Gamma_\infty} \sup_{\substack{v\in \CF_a^{(b)} \\ \|v\| = 1}} \| \pi_b(A) v\| = \sup_{b \in \Gamma_\infty} \|A\|_b = \|A\|_{\infty} \;,
		\end{equs}
		which is indeed \eqref{e:boundF}.
		
Regarding surjectivity, recall that the image of a $C^*$-algebra under a $C^*$-homomorphism is itself always a $C^*$-algebra \dash this means the image $\Im(\digamma) \subset \CA$ is a $C^*$-subalgebra of $\CA$. This follows directly from the standard result that a $C^*$-isomorphism is an isometry, cf.\ \cite[Proposition~2.3.3]{BR87}, applied to $\CA/\ker(\digamma)$. On the other hand, $\Im(\digamma)$ contains the all the generators $\big\{\alpha(f), \alpha^{\dagger}(g) \, \big| \, f,g \in \mfH \big\}$ of $\CA$ so we must have $\Im(\digamma) = \CA$. 
\end{proof}

\section{Glossary}
\label{app:notation}

Here, we list various constants, norms and other objects used in this article together 
with their meanings and references to definitions.

\begin{center}
\begin{longtable}{p{.12\textwidth}p{.65\textwidth}p{.12\textwidth}}
\toprule
Object & Meaning & Ref. \\
\midrule
\endhead
\bottomrule
\endfoot
$\alpha(f)$ & Annihilation operator & p.~\pageref{e:defaa*}\\
$\alpha^\dagger(f)$ & Creation operator & p.~\pageref{e:defaa*}\\
$\CA\F(\mfH)$ & CAR $C^*$-algebra generated by $\mfH$ & p.~\pageref{p:CAR}\\
$\mfA\F(\mfH)$ & Free $\star$-algebra generated by $\mfH$ & p.~\pageref{p:free}\\
$\widehat\mfA\F(\mfH)$ & Quotient of $\mfA\F(\mfH)$ by $\mfI_\Gamma$ & p.~\pageref{p:quotient}\\
$\cA\F(\mfH)$ & Completion of $\widehat\mfA\F(\mfH)$ under collection of seminorms & p.~\pageref{p:quotientComplete}\\
$\mfA_\infty(\mfH)$ & Uniformly bounded elements in $\cA\F(\mfH)$ & p.~\pageref{e:boundedOps}\\
$\mfB^{\otimes_s n}$ & Symmetric Hilbert tensor power & p.~\pageref{p:symPow} \\
$\digamma$ & Morphism $\mfA_\infty(\mfH) \to \CA\F(\mfH)$ & p.~\pageref{p:digamma}\\
$\CF_{a}(\mfH)$ & Fermionic Fock space generated by $\mfH$ & p.~\pageref{e:Fock}\\
$\CF_{s}(\mfB)$ & Bosonic Fock space generated by $\mfB$ & p.~\pageref{p:FockB}\\
$\mfG\F(\mfH)$ & Subalgebra of $\CA\F(\mfH)$ generated by fields & p.~\pageref{p:fields}\\
$\mfH$, $\mfB$ & Fermionic / bosonic single particle Hilbert spaces & p.~\pageref{p:mfH}, \pageref{p:boson} \\
$\mfH^{\wedge n}$ & Antisymmetric Hilbert tensor power & p.~\pageref{e:Fock} \\
$\mfI_\Gamma$ & Intersection of all the $\ker \pi_b$ & p.~\pageref{e:defmfI}\\
$\Lambda_n(\mfH)$ & Antisymmetric algebraic tensor power & p.~\pageref{e:defPsiBasic} \\
$\Lambda(\mfH)$ & Grassmann algebra generated by $\mfH$ & p.~\pageref{p:Grassmann} \\
\end{longtable}
\end{center}

\endappendix

\bibliographystyle{Martin}
\bibliography{./refs}

\end{document}